\documentclass[12pt]{amsart}
\usepackage{amsmath,amstext,amsthm,amsfonts,amssymb,amsxtra}
\usepackage{mathrsfs}
\usepackage{enumitem}
\usepackage{txfonts} 
\usepackage[colorlinks=true,citecolor=red,linkcolor=red,urlcolor=red,
            hypertexnames=false]{hyperref}
\allowdisplaybreaks

\newtheorem{theorem}{Theorem}[section]
\newtheorem{proposition}[theorem]{Proposition}
\newtheorem{corollary}[theorem]{Corollary}
\newtheorem{lemma}[theorem]{Lemma}
\newtheorem{definition}[theorem]{Definition}

\newtheorem{remark}[theorem]{Remark}

\begin{document}

\title[Riesz inequalities on Grushin spaces]{Sharp Riesz and reverse Riesz inequalities for degenerate elliptic operators of Grushin type}

\author{Dangyang He}
\address{Department of Mathematics, Sun Yat-sen University, Guangzhou, China}
\email{dangyang.he@hdr.mq.edu.au, hedy28@mail.sysu.edu.cn}

\begin{abstract}

We determine the sharp $L^p$ ranges of the Riesz transform and the reverse
Riesz inequality for a class of power-degenerate elliptic operators of
Grushin type. When the degenerate variable has dimension at least two, the
Riesz transform is bounded throughout the reflexive range. In dimension one,
a capacity-driven transition occurs: weak degeneracy produces a
sharp finite upper endpoint, whereas strong degeneracy separates the two
half-lines and restores the full range. The reverse inequality displays a
different sharp phenomenon. In the weakly degenerate one-dimensional regime
it holds throughout the reflexive range except at a single Hardy-critical
exponent, while in all remaining regimes it holds for every $1<p<\infty$.


\end{abstract}



\maketitle

 \tableofcontents

\section{Introduction}\label{section_1}

\noindent
Riesz transforms associated with non-negative self-adjoint operators are central objects in harmonic analysis and geometric analysis. On Euclidean space, the classical Riesz transforms are Calder\'on--Zygmund operators and are bounded on $L^p$ for every $1<p<\infty$; see, for example, \cite{R,CZ}. In dimension one, the Riesz transform reduces, up to normalization, to the Hilbert transform. On general metric measure or geometric spaces, however, the corresponding boundedness may fail outside a restricted range of exponents. The sharp range of $p$ is often sensitive to the heat-kernel geometry, the validity of Poincar\'e inequalities, and the structure of harmonic functions. A large body of work relates Riesz transform estimates to heat-kernel gradient bounds, reverse H\"older inequalities for gradients of harmonic functions, and Hodge projection estimates; see, among many others, \cite{ACDH,Shen,CJKS,Jiang,CD2}.

In this paper, we study these questions for a class of degenerate elliptic operators of Grushin type. Let $n,m\geq1$, $0<\alpha<1$, and $\beta\geq0$. We consider the Friedrichs realization on $L^2(\mathbb R^{n+m})$ of
\begin{equation}\label{Grushin}
   L=-\nabla_x\cdot\bigl(|x|^{2\alpha}\nabla_x\bigr)
     +|x|^{2\beta}\Delta_y,
\end{equation}
where $\Delta_y$ denotes the non-negative Laplacian on $\mathbb R^m$. The corresponding intrinsic gradient is
\begin{equation}\label{Grushin2}
    \nabla_L
    =
    \left(|x|^\alpha\nabla_x,\ |x|^\beta\nabla_y\right).
\end{equation}
The degeneracy occurs on the singular set $\{x=0\}$. The corresponding
intrinsic geometry is anisotropic. Indeed, the exponent in the dilation
\begin{equation*}
    \delta_r(x,y)
    =
    \bigl(rx,r^{\beta+1-\alpha}y\bigr)
\end{equation*}
is dictated by the homogeneity of the operator: for every smooth function
$f$,
\begin{equation*}
    L(f\circ\delta_r)
    =
    r^{2(1-\alpha)}(Lf)\circ\delta_r,
    \qquad
    \nabla_L(f\circ\delta_r)
    =
    r^{1-\alpha}(\nabla_Lf)\circ\delta_r.
\end{equation*}
Consequently, the intrinsic distance satisfies
\begin{equation*}
    d(\delta_r z,\delta_r z')
    =
    r^{1-\alpha}d(z,z').
\end{equation*}
The associated distance, volume growth, subelliptic estimates, Poincar\'e inequalities, and heat-kernel upper bounds were studied in the foundational work of Robinson and Sikora \cite{RS,RS2,RS_1D}. Classical Grushin operators and their Riesz transforms, spectral multipliers, and Bochner--Riesz means have also been studied by Jotsaroop, Sanjay, and Thangavelu \cite{JST}, Martini and Sikora \cite{MS}, Chen and Sikora \cite{ChenSikora}, and others.

We say that the Riesz transform associated with $L$ is bounded on $L^p$ if
\begin{equation}\tag{$\mathrm{R}_p$}\label{R_p}
    \|\nabla_LL^{-1/2}f\|_p
    \leq
    C_p\|f\|_p,
    \qquad f\in C_c^\infty(\mathbb R^{n+m}).
\end{equation}
Equivalently, 
\begin{equation*}
    \|\nabla_L f\|_p
    \leq
    C_p\|L^{1/2}f\|_p.
\end{equation*}
Conversely, we say that the reverse Riesz inequality holds on $L^p$ if
\begin{equation}\tag{$\mathrm{RR}_p$}\label{RR_p}
    \|L^{1/2}f\|_p
    \leq
    C_p\|\nabla_L f\|_p,
    \qquad f\in C_c^\infty(\mathbb R^{n+m}).
\end{equation}

We shall use the following Poincar\'e inequalities. For $1\leq p<\infty$, we say that \eqref{P_p_intro} holds if there exist constants $C>0$ and $\lambda\geq1$ such that, for every ball $B$ and every locally Lipschitz function $f$,
\begin{equation}\tag{$\mathrm{P}_p$}\label{P_p_intro}
    \left(\fint_B|f-f_B|^p\right)^{1/p}
    \leq
    C r(B)
    \left(\fint_{\lambda B}|\nabla_Lf|^p\right)^{1/p}.
\end{equation}
Here $f_B:=\fint_Bf$, $r(B)$ denotes the radius of $B$, and $\lambda B$ denotes the concentric dilate of $B$.

A key role in the proofs is played by reverse H\"older inequalities for gradients of harmonic functions. For $2<p<\infty$, we say that
\eqref{eq:RH:Jiang} holds if there exist constants $C>0$ and $\lambda\geq1$ such that, whenever $u$ is $L$-harmonic in $\lambda B$,
\begin{equation}\tag{$\mathrm{RH}_p$}\label{eq:RH:Jiang}
     \left(\fint_B|\nabla_Lu|^p\right)^{1/p}
     \leq
     \frac{C}{r(B)}
     \fint_{\lambda B}|u|.
\end{equation}
The endpoint condition $(\mathrm{RH}_\infty)$ is defined by replacing the left-hand side with $\operatorname*{sup}_B|\nabla_Lu|$. The use of an enlarged ball reflects the local nature of the Poincar\'e and Caccioppoli inequalities.

In the setting of doubling Dirichlet spaces supporting a scale-invariant
$L^2$-Poincar\'e inequality, Coulhon, Jiang, Koskela, and Sikora \cite{CJKS}
proved that, for $2<p<\infty$, the $L^p$ gradient estimate for the heat
semigroup, \eqref{eq:RH:Jiang}, and \eqref{R_p} are equivalent. Their endpoint
result identifies $(\mathrm{RH}_\infty)$ with a Gaussian pointwise gradient
bound for the heat kernel; this will give
Corollary~\ref{cor:heat-kernel-gradient} below. Related versions under
Poincar\'e inequalities outside a compact set were obtained by Jiang
\cite{Jiang}. This abstract principle is one of the main engines of the
present paper. In the Grushin setting, the principal difficulty is to
establish the appropriate reverse H\"older estimates near the singular set
$\{x=0\}$, where the Friedrichs realization selects different endpoint
conditions in different parameter regimes.

Grushin-type operators have long served as model examples in the analysis of degenerate elliptic and subelliptic phenomena. The terminology originates in the two-dimensional degenerate operators
\begin{equation*}
    -\partial_x^2-x^{2k}\partial_y^2,
    \qquad k\in\mathbb N,
\end{equation*}
introduced by Grushin \cite{Grushin}. The heat-kernel bounds of Robinson and Sikora \cite{RS}, together with the abstract result of Sikora \cite{Sikora}, imply that, for the general operator \eqref{Grushin}, \eqref{R_p} holds for every $1<p\leq2$. Beyond the threshold $p=2$, however, the previously known full-range results concern special algebraic models.

In the case $\alpha=0$, $\beta=1$, and $m=1$, namely for the classical model
\begin{equation*}
    L=\Delta_x-|x|^2\partial_y^2,
\end{equation*}
Jotsaroop, Sanjay, and Thangavelu \cite{JST} proved $L^p$-boundedness for the full range $1<p<\infty$, using a closely related first-order system based on the creation and annihilation operators for the Hermite operator. Subsequently, Robinson and Sikora \cite{RS3} established full $L^p$ bounds for the integer-power models
\begin{equation*}
    L=\Delta_x+|x|^{2\beta}\Delta_y,
    \qquad
    n,m\geq1,\quad \alpha=0,\quad \beta\in\mathbb N,
\end{equation*}
by lifting the corresponding polynomial vector fields to a nilpotent Lie
group. This method depends essentially on the integer-power structure and
does not directly apply to the general power-degenerate Friedrichs model
\eqref{Grushin}. In particular, once $\alpha>0$, the degeneracy in the
$x$-derivatives introduces an endpoint mechanism at $\{x=0\}$ which is
invisible to the nilpotent lifting argument. More generally, degenerate
divergence-form operators may exhibit qualitative features absent from the
uniformly elliptic theory; see \cite{ARSZ}.

For arbitrary powers, the nilpotent lifting argument does not cover the
power-degenerate normal part in \eqref{Grushin}. Our first result establishes
the full $L^p$ range for $n\geq2$ and $0<\alpha<1$; the case $\alpha=0$
is outside its stated scope.

\begin{theorem}\label{thm_main1}
Let $L$ be the Grushin operator defined by \eqref{Grushin}, with parameters
$\alpha\in(0,1)$, $\beta\geq0$, $m\geq1$, and $n\geq2$. Let $\nabla_L$ be the
corresponding gradient operator defined in \eqref{Grushin2}. Then
\eqref{R_p} holds for every $1<p<\infty$.
\end{theorem}

The proof of Theorem~\ref{thm_main1} proceeds through
$(\mathrm{RH}_\infty)$. When $n\geq2$, the Grushin space satisfies volume doubling, Gaussian heat-kernel upper bounds, and the required Poincar\'e inequality. The theorem of Coulhon, Jiang, Koskela, and Sikora therefore reduces the $L^p$-boundedness for $p>2$ to the reverse H\"older estimate \eqref{eq:RH:Jiang}. Away from the singular set, the desired estimate follows from standard elliptic arguments after rescaling. The main analytic difficulty is local and lies in proving the estimate on balls intersecting $\{x=0\}$.

In higher dimensions, one must decompose in spherical harmonics in the
$x$-variable. The order of the resulting modified Bessel functions depends on the angular frequency and tends to infinity with it. Consequently, the proof cannot rely only on the usual fixed-order asymptotic formulae for Bessel functions. Instead, one needs estimates that are uniform jointly in the Bessel order, the Fourier frequency in the $y$-variable, and the distance to the singular set. This uniform-in-order analysis is the principal technical point in the proof of the higher-dimensional theorem; see Part~\ref{part1}.

The one-dimensional case is qualitatively different. Suppose first that $n=1$ and $0<\alpha<1/2$. Then
\begin{equation*}
    \int_{-1}^1|x|^{-2\alpha}\,dx<\infty,
\end{equation*}
and functions in the local Friedrichs form domain admit a common trace across $\{x=0\}$. Moreover, weak solutions satisfy a conormal flux-matching condition; see \cite{RS_1D}. Consequently, after decomposing into even and odd parts, the full-line problem reduces to two different half-line problems: the even component satisfies a zero-flux condition at the origin, whereas the odd component satisfies a zero-trace condition. The leading odd branch behaves like
\begin{equation*}
    \operatorname{sgn}(x)|x|^{1-2\alpha},
\end{equation*}
and hence its intrinsic gradient has the singular behavior
$|x|^{-\alpha}$. This produces the exact obstruction at
$p=\alpha^{-1}$; see Part~\ref{part2}.

\begin{theorem}\label{thm:n=1:alpha<1/2}
Let $L$ be the Grushin operator defined by \eqref{Grushin}, with parameters
$\alpha\in(0,\frac12)$, $\beta\geq0$, $m\geq1$, and $n=1$. Let $\nabla_L$ be
the corresponding gradient operator defined in \eqref{Grushin2}. Then
\eqref{R_p} holds if and only if
\begin{equation*}
    1<p<\alpha^{-1}.
\end{equation*}
\end{theorem}

The same odd branch yields the positive range and the endpoint obstruction.

The remaining one-dimensional regime, $n=1$ and $1/2\leq\alpha<1$, is treated below through the Hodge projector. In this range,
\begin{equation*}
    \int_{-1}^1|x|^{-2\alpha}\,dx=\infty,
\end{equation*}
so the trace-matching argument breaks down. The set $\{x=0\}$ has zero
capacity with respect to the Friedrichs form, and the two sides
$\{x>0\}$ and $\{x<0\}$ are separated by the Friedrichs realization. In
particular, the global Poincar\'e inequality is no longer available in the
form needed to invoke directly the equivalence between \eqref{R_p} and
\eqref{eq:RH:Jiang}; see \cite{RS2}. We therefore postpone the proof of the
Riesz transform estimate in this strongly degenerate regime and first turn
to the reverse Riesz inequality.

The reverse Riesz inequality is governed by a different mechanism. Although
\eqref{R_p} implies $(\mathrm{RR}_{p'})$ by duality (see \cite{CD}),
reverse Riesz estimates may hold in ranges where the Riesz transform itself
fails. This phenomenon occurs, for example, for exterior-domain Laplacians
and manifolds with ends; see \cite{KVZ,JL,H2}. For exponents below $2$, a
classical route to \eqref{RR_p} uses volume doubling together with \eqref{P_p_intro}; see, for instance, \cite{AC}. In the strongly
degenerate one-dimensional regime, however, the relevant Poincar\'e
inequality across the singular set fails \cite{RS2}. We therefore develop a different argument, introduced by the author in \cite{H2}.

The guiding principle is that the obstruction to the Riesz transform is
often carried by a leading harmonic term in the kernel. When the square-root
operator is written in bilinear form and paired with the gradient, this
leading harmonic contribution can be eliminated by an integration-by-parts
argument. We refer to this cancellation mechanism as harmonic annihilation.

The model case to keep in mind is the Riesz transform on manifolds with Euclidean ends. The low-energy analysis of Carron, Coulhon, and Hassell \cite{CCH} identifies a leading contribution to the kernel of
$\Delta^{-1/2}$ whose coefficient is harmonic. When the manifold has more than one end, this coefficient is nonconstant, and its gradient is responsible for the failure of the Riesz transform in the endpoint range. For the reverse Riesz problem, however, the harmonicity of this leading coefficient becomes a source of cancellation: after passing to the bilinear form, the leading profile is annihilated. This idea was developed in \cite{H2} and further refined in \cite{H3}.

More precisely, the author's earlier work \cite{H3} treats manifolds with
quadratically decaying curvature. Under the geometric hypotheses imposed
there, it relates the reverse Riesz inequality to Hardy and weighted Sobolev
inequalities and applies the resulting framework to the higher-codimensional
Grushin operator
\begin{equation*}
    L_0=\Delta_x+|x|^{2\beta}\Delta_y,
    \qquad
    (x,y)\in\mathbb R^n\times\mathbb R^m,
    \quad n\geq2,\quad m\geq1,\quad\beta\geq0,
\end{equation*}
obtaining the reverse Riesz inequality for every $1<p<\infty$. The present
paper treats the power-degenerate normal operator
$-\nabla_x\cdot(|x|^{2\alpha}\nabla_x)$ and, in codimension one, reveals the
Friedrichs-capacity transition at $\alpha=1/2$ and the sharp failure at
$p=(1-\alpha)^{-1}$. These interface phenomena are absent from the model
considered in \cite{H3}.

In the present Grushin setting, the same principle can be implemented
without relying on a global scattering calculus. Following an idea of
Carron \cite{Gilles}, we decompose the relevant kernel into geometric pieces, isolate the critical region in which the leading harmonic profile appears, and subtract this profile before estimating the remainder. The remaining term is then controlled by a Hardy inequality adapted to the Grushin weights. This yields the following reverse Riesz theorem.

\begin{theorem}\label{thm_main2}
Let $L$ be the Grushin operator defined by \eqref{Grushin}, with parameters $\alpha\in(0,1)$, $\beta\geq0$, and $n,m\geq1$. Let $\nabla_L$ be the corresponding gradient operator defined in \eqref{Grushin2}. Then \eqref{RR_p} holds if and only if
\begin{equation}\label{p2}
   p\in
   \begin{cases}
        \left(1,\frac{1}{1-\alpha}\right)
        \cup
        \left(\frac{1}{1-\alpha},\infty\right),
        & n=1,\quad 0<\alpha<\frac12,\\[3mm]
        (1,\infty),
        & n\geq2
        \quad\text{or}\quad
        n=1,\quad \frac12\leq\alpha<1.
   \end{cases}
\end{equation}
\end{theorem}

The principal novelty in Theorem~\ref{thm_main2} is again one-dimensional.
When $n\geq2$, the full range follows from Theorem~\ref{thm_main1} and
duality. In the weakly degenerate one-dimensional case
$0<\alpha<1/2$, Theorem~\ref{thm:n=1:alpha<1/2} and duality give
\eqref{RR_p} for $p>(1-\alpha)^{-1}$, whereas harmonic annihilation yields the complementary range
$1<p<(1-\alpha)^{-1}$. The excluded exponent
$p=(1-\alpha)^{-1}$ is the natural Hardy threshold. Thus, after the leading
harmonic profile has been annihilated, the reverse Riesz inequality persists
well beyond the range in which the Riesz transform fails.

The failure at the excluded exponent has a complementary semigroup
interpretation. In the odd sector, the full-line problem becomes the
half-line Dirichlet problem, and the loss of heat through the absorbing
boundary produces the profile $x^{-(1-\alpha)}$. At
$p=(1-\alpha)^{-1}$ this profile is exactly logarithmically
non-integrable; see Theorem~\ref{thm:critical-RR-failure}.

In the strongly degenerate case $1/2\leq\alpha<1$, the same method yields the
reverse inequality without a global Poincar\'e inequality across the
singular set. Moreover, the reverse Riesz estimate provides an essential
ingredient in the proof of the remaining Riesz transform theorem.

For $1/2\leq\alpha<1$, the separated Friedrichs realization selects
the zero-flux branch on each half-line \cite{RS_1D}. We prove a reverse
H\"older estimate for that branch and use the separated Poincar\'e inequality
\cite{RS2} in Shen's criterion for the Hodge projector. Combining the
resulting projector bound with Theorem~\ref{thm_main2} gives the remaining
Riesz transform estimate.

\begin{theorem}\label{thm:n=1:alpha>1/2}
Let $L$ be the Grushin operator defined by \eqref{Grushin}, with parameters
$\alpha\in[\frac12,1)$, $\beta\geq0$, $m\geq1$, and $n=1$. Let $\nabla_L$ be
the corresponding gradient operator defined in \eqref{Grushin2}. Then
\eqref{R_p} holds for every $1<p<\infty$.
\end{theorem}

The three Riesz ranges are determined by the finite-energy harmonic
branch selected at the singular set and the integrability of its intrinsic
gradient. For completeness, Section~\ref{sec:weak-Hodge} also determines
the Hodge projector range when $n=1$ and $0<\alpha<1/2$:
\[
 \frac1{1-\alpha}<p<\frac1\alpha.
\]
Thus the Riesz, reverse Riesz and Hodge inequalities have distinct sharp
ranges in the weakly degenerate regime.

\medskip

\noindent\textbf{Outline.}
The paper is organized as follows. Section~\ref{sec2} introduces the
Friedrichs realization, recalls the Grushin geometry, and proves the
Poincar\'e inequalities needed later. Part~\ref{part1} proves
Theorem~\ref{thm_main1} through anchored reverse H\"older estimates and uniform-in-order Bessel analysis. Part~\ref{part2} proves
Theorem~\ref{thm:n=1:alpha<1/2} using the even--odd half-line decomposition and the endpoint counterexample. Part~\ref{part3} proves the reverse Riesz inequality, Theorem~\ref{thm_main2}, using a Grushin Hardy inequality and harmonic annihilation. Part~\ref{part4} proves Theorem~\ref{thm:n=1:alpha>1/2} using separated
half-line kernels and the Hodge projector, and then establishes the sharp
Hodge range in the weakly degenerate regime.

\section{Preliminaries}\label{sec2}

\subsection{Friedrichs realization}\label{ssec:Friedrichs}

The singular set $\{x=0\}$ is not merely a technical nuisance; it determines which endpoint branches are admissible in the kernel construction (see Parts~\ref{part1}, \ref{part2}, and~\ref{part4} below). We therefore begin by fixing the Friedrichs realization of $L$ and by recording the interface conditions it imposes. For a more detailed discussion in this direction, we refer the reader to \cite{RS_1D,RS_flux} and the references therein.

Let $n,m\ge 1$, $\alpha\in (0,1)$ and $\beta \ge 0$. On
$C_c^\infty(\mathbb R^{n+m})$, define
\begin{equation*}
    \mathfrak q_0(u,v)
    =
    \int_{\mathbb R^{n+m}}
    \left(
        |x|^{2\alpha}\nabla_xu\cdot \overline{\nabla_xv}
        +
        |x|^{2\beta}\nabla_yu\cdot \overline{\nabla_yv}
    \right)dxdy.
\end{equation*}
This non-negative symmetric form is closable in $L^2(\mathbb R^{n+m})$ (cf. \cite[Lemma~2.1]{RS}). We denote its
closure by $(\mathfrak q,\mathcal F)$, where
\begin{equation*}
    \mathcal F
    =
    \overline{C_c^\infty(\mathbb R^{n+m})}^{\|\cdot\|_{\mathcal F}},
    \qquad
    \|u\|_{\mathcal F}^2
    =
    \|u\|_2^2+\mathfrak q(u,u).
\end{equation*}
The Friedrichs realization of $L$ is the non-negative self-adjoint operator associated
with $(\mathfrak q,\mathcal F)$. Thus $u\in\mathcal D(L)$ and $Lu=f$ if and only if
$u\in\mathcal F$, $f\in L^2$, and
\begin{equation*}
    \mathfrak q (u,\varphi)=\langle f,\varphi\rangle_{L^2}
    \qquad
    \text{for all }\varphi\in\mathcal F.
\end{equation*}
We also use the corresponding local notion. If
$\Omega\subset\mathbb R^{n+m}$ is open, then
$u\in\mathcal F_{\mathrm{loc}}(\Omega)$ means that for every relatively
compact open set $\widetilde{\Omega}\Subset\Omega$ there exists
$v\in\mathcal F$ such that $u=v$ a.e. on $\widetilde{\Omega}$. We say that
$u$ has locally finite energy if $u\in\mathcal F_{\mathrm{loc}}$.
In particular,
\begin{equation*}
    \int_{\widetilde{\Omega}}
    \left(
        |x|^{2\alpha}|\nabla_xu|^2
        +
        |x|^{2\beta}|\nabla_yu|^2
    \right)dxdy
    <\infty.
\end{equation*}
A function $u\in\mathcal F_{\mathrm{loc}}(\Omega)$ is called $L$-harmonic in $\Omega$ if
\begin{equation}
    \mathfrak q(u,\varphi)=0
    \qquad
    \text{for all }\varphi\in\mathcal F
    \text{ compactly supported in }\Omega.
\end{equation}

We now record the only subtle point concerning the singular set when $n=1$. Near $x=0$, the relevant part of the energy is $\int |x|^{2\alpha}|\partial_xu|^2dx$. The difference between $0<\alpha<1/2$ and $1/2 \le \alpha<1$ is essential. In the former regime, the form domain has a common trace across $x=0$, while in the latter regime the singular set has zero local Friedrichs capacity and separates the two half-lines.

Assume first that $0<\alpha<1/2$. Then $\int_{-1}^{1}|x|^{-2\alpha}dx<\infty$. Hence, for smooth $u$ and $x_1 < 0 < x_2$, Cauchy-Schwarz inequality gives
\begin{align*}
|u(x_2,y)-u(x_1,y)|^2\le \left(\int_{x_1}^{x_2}|t|^{-2\alpha}dt\right) \left(\int_{x_1}^{x_2}|t|^{2\alpha}|\partial_tu(t,y)|^2 dt\right).
\end{align*}
By density, the same estimate holds for functions in the local Friedrichs form
domain. Hence, for almost every $y$, the one-sided traces $u(0^+,y)$ and
$u(0^-,y)$ exist. Moreover, since both factors on the right-hand side above tend to zero as $x_1\to0^-$ and
$x_2\to0^+$, one obtains
\begin{equation}\label{eq_subcritical_trace_matching}
    u(0^+,y)=u(0^-,y)
    \qquad
    \text{for a.e. }y.
\end{equation}
Thus the common trace condition comes from the Friedrichs form domain.

The flux condition comes from the weak equation. Suppose that $u$ is piecewise smooth away from $x=0$, belongs locally to the Friedrichs form domain, and solves
\begin{equation}\label{equation}
    Lu=f
\end{equation}
weakly across $x=0$, with no singular source supported on $\{x=0\}$. Suppose that the one-sided conormal fluxes
\begin{equation*}
    F_+(y):=\lim_{x\to0^+}x^{2\alpha}\partial_xu(x,y),
    \qquad
    F_-(y):=\lim_{x\to0^-}|x|^{2\alpha}\partial_xu(x,y)
\end{equation*}
exist. Integrating by parts separately on $(-\varepsilon,0)$ and $(0,\varepsilon)$
gives the interface contribution
\begin{equation*}
\int_{\mathbb R^m}
    \left[F_-(y)\varphi(0^-,y)-F_+(y)\varphi(0^+,y)\right] dy.    
\end{equation*}
By \eqref{eq_subcritical_trace_matching}, test functions in the Friedrichs space
have a common trace at $x=0$. Therefore the interface contribution becomes
\begin{equation*}
    \int_{\mathbb R^m}
    \left[F_-(y)-F_+(y)\right]\varphi(0,y) dy.
\end{equation*}
Since the weak equation contains only the volume term $\int f\varphi$ and has no
delta contribution on $\{x=0\}$, this interface term must vanish for all test
functions $\varphi$. Hence
\begin{equation}\label{eq_subcritical_flux_matching}
    F_-(y)=F_+(y)
    \qquad\text{for a.e. }y.
\end{equation}

For later use, now consider even and odd weak solutions of \eqref{equation}. If $u_e(x,y)=v_e(|x|,y)$, then
\begin{equation*}
F_+(y)=\lim_{\rho\to0^+}\rho^{2\alpha}\partial_\rho v_e(\rho,y),
    \qquad
    F_-(y)=-\lim_{\rho\to0^+}\rho^{2\alpha}\partial_\rho v_e(\rho,y). 
\end{equation*}
The matching condition \eqref{eq_subcritical_flux_matching} therefore gives $\lim_{\rho\to0+}\rho^{2\alpha}\partial_\rho v_e(\rho,y)=0$. Thus the even solution satisfies the zero-flux condition. If $u_o(x,y)=\mathrm{sgn}(x)v_o(|x|,y)$, then the trace matching condition \eqref{eq_subcritical_trace_matching} gives $v_o(0+,y)=-v_o(0+,y)$ and hence $v_o(0+,y)=0$. Thus the odd solution satisfies the zero-trace condition.

When $1/2\leq\alpha<1$, the situation is different. The integral $\int_{-1}^{1}|x|^{-2\alpha}dx$ diverges, so the above trace argument no longer applies. In fact, the singular point has zero
Friedrichs capacity. This can be seen from the approximation of the jump function $\mathrm{sgn}(x)$.
For $\alpha>1/2$, take $\eta_\varepsilon\in C^\infty([0,\infty))$ such that
\begin{equation*}
    \eta_\varepsilon(r) = \begin{cases}
        0, & r\le \varepsilon,\\
        1, & r\ge 2\varepsilon,
    \end{cases}
\end{equation*}
and $|\eta_\varepsilon'(r)|\leq C\varepsilon^{-1}$. Then
\begin{equation*}
    \int_{-1}^{1}
    |x|^{2\alpha}
    \left|
        \partial_x\bigl(\mathrm{sgn}(x)\eta_\varepsilon(|x|)\bigr)
    \right|^2 dx
    \le
    C\varepsilon^{2\alpha-1}
    \to 0.
\end{equation*}
At $\alpha=1/2$, using the logarithmic cut-off
\begin{equation*}
\eta_\varepsilon(r)= \begin{cases}
        0, & 0\le r\leq\varepsilon^2,\\
        \frac{\log(r/\varepsilon^2)}{\log(1/\varepsilon)},
        & \varepsilon^2<r<\varepsilon,\\
        1, & r\ge \varepsilon,
    \end{cases}
\end{equation*}
one obtains
\begin{equation*}
\int_{-1}^{1}|x|\left|\partial_x\bigl(\mathrm{sgn}(x)\eta_\varepsilon(|x|)\bigr)\right|^2 dx \le \frac{C}{\log(1/\varepsilon)} \to 0.
\end{equation*}
Thus $\mathrm{sgn}(x)$ belongs to the local Friedrichs form domain when
$1/2\leq\alpha<1$, and
\begin{equation}\label{eq:gradient:sgn}
    \nabla_L \mathrm{sgn}(x)=0
    \qquad
    \text{in the form sense}.
\end{equation}
Consequently, $\mathfrak q(\mathrm{sgn}(x),\varphi)=0$ for every compactly supported test function $\varphi$, so $\mathrm{sgn}(x)$ is
locally $L$-harmonic. Therefore, locally near $x=0$, the two half-lines are separated and there is no trace matching condition across $x=0$. 

We now explain why the separated half-line Friedrichs problem carries the zero-flux endpoint condition. Consider a sufficiently regular solution $u$ on the half-line $(0,2)$ whose conormal flux $w(\rho):=\rho^{2\alpha}\partial_\rho u(\rho)$ has a finite limit as $\rho\to0^+$. If $\lim_{\rho\to0^+}w(\rho)=C\ne 0$, then, for all sufficiently small $\rho$, $|\partial_\rho u(\rho)|= \rho^{-2\alpha}|w(\rho)| \gtrsim \rho^{-2\alpha}$. Consequently,
\begin{equation*}
\int_0^1 \rho^{2\alpha}|\partial_\rho u(\rho)|^2 d\rho
\gtrsim
\int_0^1 \rho^{-2\alpha} d\rho
=\infty,    
\end{equation*}
because $\alpha\ge 1/2$. This contradicts the Friedrichs finite-energy condition. Therefore every such Friedrichs solution must satisfy
\begin{equation}\label{eq_supercritical_zero_flux}
    \lim_{\rho\to0^+}\rho^{2\alpha}\partial_\rho u(\rho)=0.
\end{equation}
Equivalently, the local homogeneous normal equation $-\partial_\rho\left(\rho^{2\alpha}\partial_\rho u\right)=0$ has two endpoint branches: $u(\rho)\sim 1$ and
\begin{equation*}
u(\rho)\sim
\begin{cases}
\rho^{1-2\alpha}, & \alpha>1/2,\\
\log\rho, & \alpha=1/2.
\end{cases}    
\end{equation*}
The second branch has nonzero conormal flux and infinite Friedrichs energy,
whereas the first branch has zero conormal flux. Thus the Friedrichs
extension selects the first branch on each separated half-line. This
separation phenomenon is a manifestation of the general capacity--flux
criterion for degenerate divergence-form operators: a hypersurface may
become an impenetrable barrier for the Friedrichs semigroup when its normal
capacity vanishes (see \cite{RS_flux}).

For the explicit Fourier-Bessel modes used in Parts~\ref{part2} and~\ref{part4} below, the conormal flux exists and finite energy excludes the nonzero-flux branch.

\subsection{Grushin geometry}

We next recall the intrinsic geometry associated with $L$. The relevant distance is obtained by declaring the vector fields $|x|^{\alpha} \nabla_x$ and $|x|^\beta \nabla_y$ to have unit size (cf.~\cite{JS}). 

Recall the homogeneous dimension of Grushin space:
\begin{equation*}
    \mathcal Q:= \frac{n+m(\beta+1-\alpha)}{1-\alpha}.
\end{equation*}
We record the following result from \cite[Proposition~5.1, Theorem~6.4, Corollary~6.6]{RS}.

\begin{theorem}\cite{RS}\label{RS_Grushin}
Let $n,m\ge 1$, $\alpha \in (0,1)$ and $\beta \ge 0$. The following estimates hold:
\begin{enumerate}[label=(\roman*)]
\item the distance estimate
\begin{align*}
d(\xi,\eta) = d((x,y);(x',y')) \sim \frac{|x-x'|}{\left( |x| + |x'| \right)^\alpha} + \frac{|y-y'|}{(|x|+|x'|)^{\beta} + |y-y'|^{\frac{\beta}{\beta+1-\alpha}}} =: D(\xi,\eta),
\end{align*}
\item the volume estimate
\begin{align*}
V(\xi,r) = V((x,y),r) \sim \begin{cases}
r^{\mathcal{Q}}, & r\ge |x|^{1-\alpha},\\
r^{n+m}|x|^{n \alpha + m\beta}, & r\le |x|^{1-\alpha},
\end{cases}
\end{align*}
\item the volume doubling property
\begin{align}\label{D}\tag{$\mathrm{D}$}
V(\xi,sr) \le C s^{\mathcal{Q}} V(\xi,r)
\end{align}
for all $\xi \in \mathbb{R}^{n+m}$ and all $s\ge 1$, $r>0$,
\item the heat kernel estimate
\begin{equation}\label{UE}\tag{$\mathrm{UE}$}
e^{-tL}(\xi,\eta) \le \frac{C}{V(\xi, \sqrt{t})} \textrm{exp} \left( - \frac{d(\xi,\eta)^2}{ct} \right).
\end{equation}
\end{enumerate}
\end{theorem}

With Theorem~\ref{RS_Grushin}, we establish the following estimates.

\begin{lemma}\label{le_volume}
Let $n,m\ge 1$, $\alpha\in (0,1)$ and $\beta\ge 0$. Let $\xi=(x,y)\in \mathbb{R}^{n+m}$. There exist $0<c_1<c_2$ and $\varepsilon_0 \in (0,1)$ such that for all $r\le \varepsilon_0 |x|^{1-\alpha}$,
\begin{equation}\label{le_volume2}
B_n(x, c_1 r |x|^\alpha) \times B_m(y, c_1 r |x|^\beta) \subset B(\xi, r) \subset B_n(x, c_2 r |x|^\alpha) \times B_m(y, c_2 r |x|^\beta).
\end{equation}
Moreover, there exist $0<c_3<c_4$ such that for all $r>0$ and all $y\in \mathbb{R}^m$,
\begin{equation}\label{le_volume3}
B_n(0, c_3 r) \times B_m(y, c_3 r^{\beta+1-\alpha}) \subset B((0,y), r^{1-\alpha}) \subset B_n(0, c_4 r) \times B_m(y, c_4 r^{\beta+1-\alpha}).
\end{equation}
\end{lemma}

\begin{proof}
Fix a constant $C_0>1$ such that $C_0^{-1} D(\xi,\eta) \le d(\xi,\eta) \le C_0 D(\xi,\eta)$. First, we establish that
\begin{equation}\label{le_volume1}
B(\xi, r) \subset B_n(x, c_2 r |x|^\alpha) \times B_m(y, c_2 r |x|^\beta).
\end{equation}
Let $\eta = (x',y')\in B(\xi, r)$. Then $D(\xi,\eta)\le C_0 r$. The following claim holds:
\begin{equation}\label{claim1}
|x-x'| \le \frac{|x|}{2},
\end{equation}
provided $\varepsilon_0$ is sufficiently small. Indeed, if this were false, then $|x|+|x'|\le |x-x'| + 2|x| \le 5|x-x'|$, and hence
\begin{equation*}
\frac{|x-x'|}{\left(|x|+|x'|\right)^\alpha} \ge 5^{-\alpha} |x-x'|^{1-\alpha} \ge 5^{-\alpha} 2^{\alpha-1} |x|^{1-\alpha}.
\end{equation*}
Consequently, $D(\xi,\eta)\ge 5^{-\alpha} 2^{\alpha-1} |x|^{1-\alpha}$. Choose $\varepsilon_0 < C_0^{-1} 5^{-\alpha} 2^{\alpha-1}$. Since $r\le \varepsilon_0 |x|^{1-\alpha}$, the previous lower bound contradicts $D(\xi,\eta)\le C_0 r$. This proves \eqref{claim1}.

Therefore,
\begin{equation}\label{eq_volume2}
\frac{|x|}{2} \le |x'| \le \frac{3}{2}|x|\quad \textrm{and}\quad |x|+|x'| \le \frac{5}{2} |x|.
\end{equation}
Since $D(\xi,\eta)\le C_0 r$, we obtain
\begin{align*}
\frac{|x-x'|}{\left(|x|+|x'|\right)^\alpha} \le C_0 r \implies |x-x'|\le C_0 r \left(|x|+|x'|\right)^\alpha \le C_0 \left(\frac{5}{2}\right)^\alpha r |x|^\alpha.
\end{align*}
For the $y$-part of $D(\xi,\eta)$, if $\left(|x|+|x'|\right)^{\beta+1-\alpha}\ge |y-y'|$, then
\begin{align*}
\frac{|y-y'|}{2\left(|x|+|x'|\right)^{\beta}} \le \frac{|y-y'|}{\left(|x|+|x'|\right)^{\beta} + |y-y'|^{\frac{\beta}{\beta+1-\alpha}}} \le C_0 r.
\end{align*}
Combining this with \eqref{eq_volume2} yields
\begin{equation*}
|y-y'|\le 2C_0 r \left(|x|+|x'|\right)^{\beta} \le 2C_0 \left(\frac{5}{2}\right)^\beta r |x|^\beta.
\end{equation*}
If instead $\left(|x|+|x'|\right)^{\beta+1-\alpha}\le |y-y'|$, then
\begin{align*}
|y-y'|^{\frac{1-\alpha}{\beta+1-\alpha}} \le 2 C_0 r \implies |y-y'|\le \left(2C_0\right)^{\frac{\beta+1-\alpha}{1-\alpha}} r^{\frac{\beta+1-\alpha}{1-\alpha}}\le \left(2C_0\right)^{\frac{\beta+1-\alpha}{1-\alpha}} r |x|^{\beta}.
\end{align*}
Thus, setting
\begin{equation*}
c_2 = \max \left( C_0 \left(\frac{5}{2}\right)^\alpha,  2C_0 \left(\frac{5}{2}\right)^\beta, \left(2C_0\right)^{\frac{\beta+1-\alpha}{1-\alpha}} \right),
\end{equation*}
gives \eqref{le_volume1}.

It remains to prove the converse inclusion
\begin{equation*}
B_n(x, c_1 r |x|^\alpha) \times B_m(y, c_1 r |x|^\beta) \subset B(\xi, r).
\end{equation*}
This is immediate. Let $\eta = (x',y')$ satisfy $|x-x'|\le c_1 r |x|^\alpha$ and $|y-y'|\le c_1 r |x|^\beta$. It is plain that
\begin{equation*}
d(\xi,\eta)\le C_0 D(\xi,\eta)\le C_0 \frac{c_1 r |x|^\alpha}{\left(|x|+|x'|\right)^\alpha} + C_0 \frac{c_1 r |x|^\beta}{\left(|x|+|x'|\right)^\beta+|y-y'|^{\frac{\beta}{\beta+1-\alpha}}} \le 2 C_0 c_1 r.
\end{equation*}
Choosing $c_1 = (4C_0)^{-1}$ implies $\eta \in B(\xi,r)$, and hence \eqref{le_volume2} follows.

Finally, \eqref{le_volume3} follows directly from a similar analysis as above and involves no additional difficulty, so the details are omitted.
\end{proof}

\begin{remark}
Note that the constants $c_1$ and $c_2$ depend only on $C_0,\alpha,\beta$. Therefore, by requiring
\begin{equation*}
\varepsilon_0< \min \left( C_0^{-1} 5^{-\alpha} 2^{\alpha-1}, \frac{1}{2c_2} \right),
\end{equation*}
it follows that for all $r\le \varepsilon_0 |x|^{1-\alpha}$ and all $\eta = (x',y') \in B_n(x, c_2 r |x|^\alpha) \times B_m(y, c_2 r |x|^\beta)$,
\begin{equation*}
|x|\le |x-x'|+|x'| \le c_2 r |x|^{\alpha} + |x'| \le \frac{|x|}{2} +|x'|,
\end{equation*}
and hence $|x|\le 2|x'|\le 3|x|$.
\end{remark}

\begin{definition}
A Grushin ball $B=B(\xi,r) = B((x,y), r)$ is called
\begin{itemize}
\item remote if
\begin{equation*}
r \le \varepsilon_0 |x|^{1-\alpha},
\end{equation*}
where $\varepsilon_0$ comes from Lemma~\ref{le_volume}, so that for all $\eta = (x',y')\in B$, one has $|x-x'|\le \frac{|x|}{2}$ and hence $|x|\sim |x'|$;
\item anchored if $B=B(\xi,r) = B((0,y),r)$.
\end{itemize}
\end{definition}

We briefly recall the curvature-dimension notation used below. Let $\mathcal L$ be a non-positive diffusion operator with carr\'e du champ
\begin{equation*}
    \Gamma(f,g)
    :=
    2^{-1} \bigl(\mathcal L(fg)-f\mathcal Lg-g \mathcal Lf\bigr),
    \qquad
    \Gamma(f):=\Gamma(f,f).
\end{equation*}
The iterated carr\'e du champ is defined by
\begin{equation*}
    \Gamma_2(f,g)
    :=
    2^{-1} \bigl(\mathcal L\Gamma(f,g)-\Gamma(f, \mathcal Lg)-\Gamma(g, \mathcal Lf)\bigr),
    \qquad
    \Gamma_2(f):=\Gamma_2(f,f).
\end{equation*}
Given $N_1 \in\mathbb R$ and $N_2 \in(0,\infty]$, one says that $\mathcal L$ satisfies the curvature-dimension inequality $\operatorname{CD}(N_1,N_2)$ if
\begin{equation*}
    \Gamma_2(f)
    \ge
    N_1\Gamma(f)+ N_2^{-1} (\mathcal Lf)^2
\end{equation*}
for all sufficiently smooth functions $f$, with the last term omitted when $N_2=\infty$.

By setting $\mathcal L=-L$, we have the following gradient estimates for the heat kernel.

\begin{lemma}\cite{H4}\label{Grushin_gradient}
Let $n,m\ge 1$, $\alpha\in (0,1)$, and $\beta\ge 0$. Then 
\begin{enumerate}[label=(\roman*)]
    \item for any $f\in C^\infty(\mathbb{R}^n \setminus \{0\} \times \mathbb{R}^m)$, the following curvature-dimension inequality holds:
    \begin{equation*}
        \Gamma_2(f) \ge - \frac{C_{n,m,\alpha,\beta}}{|x|^{2-2\alpha}} \Gamma(f) + C_{n,m,\alpha,\beta}' (Lf)^2;
    \end{equation*}
    \item the following gradient estimate holds:
    \begin{align*}
        \left|\nabla_{L}e^{-tL}(\xi,\eta)\right|
        \le \left(\frac{1}{\sqrt{t}} + \frac{1}{|x|^{1-\alpha}}\right)
        \frac{C}{V(\xi,\sqrt{t})}e^{-\frac{d(\xi,\eta)^2}{ct}},
    \end{align*}
    where $\xi = (x,y) \in \mathbb{R}^n \setminus \{0\} \times \mathbb{R}^m$.
\end{enumerate}
\end{lemma}

\begin{remark}
On remote balls the coefficients are uniformly elliptic after the intrinsic rescaling. Hence the required gradient bound may also follow from standard interior gradient estimates for uniformly parabolic equations; see \cite{Lieberman}, for instance.
\end{remark}

\subsection{Poincar\'e inequality}

We establish the following variants of the Poincar\'e inequality.

\begin{lemma}\label{PE}
Let $n,m\ge 1$, $\alpha\in (0,1)$ and $\beta\ge 0$. Let $B(\xi_0, r) = B((x_0,y_0), r)$ be a remote ball. Then there exist $\gamma \ge 1$ (depending only on the structural constants) and $C_P>0$ such that, for all $1\le p < \infty$ and all $f\in C^\infty(\gamma B)$, the following Poincar\'e inequality holds:
\begin{equation}
    \int_B |f-f_B|^p d\xi \le C_P r^p \int_{\gamma B} |\nabla_L f|^p d\xi,
\end{equation}
where the constant $C_P$ depends only on $n,m,\alpha,\beta,p$.
\end{lemma}

\begin{proof}
Let $B=B(\xi_0,r) = B((x_0,y_0),r)$ be a remote ball, and let $f\in C^\infty(\gamma B)$. By Lemma~\ref{le_volume}, there exists $c_2>0$ such that
\begin{equation*}
    B(\xi_0, r) \subset B_n(x_0,c_2 r |x_0|^\alpha) \times B_m(y_0,c_2 r |x_0|^\beta)=: Q.
\end{equation*}
Define the rescaling map
\begin{equation*}
    \Lambda: (x', y') \mapsto \left( x_0 + r|x_0|^\alpha x', y_0 + r|x_0|^\beta y' \right),\quad \forall x'\in \mathbb{R}^n,\quad y'\in \mathbb{R}^m,
\end{equation*}
and the function $F(x',y') = f\left( \Lambda(x',y') \right)$. Clearly, $\Lambda$ maps $Q_0:=B_n(0,c_2) \times B_m(0,c_2)$ onto $Q$. By the standard Poincar\'e inequality on $\mathbb{R}^{n+m}$ for the convex set $Q_0$,
\begin{align*}
    \int_{Q_0} |F(x',y') - F_{Q_0}|^p dx'dy' \le C(n,m,c_2,p) \int_{Q_0} |\nabla F(x',y')|^p dx'dy'.
\end{align*}
After a change of variables,
\begin{equation*}
    \int_{Q} |f(\xi) - f_{Q}|^p d\xi \le C(n,m,c_2,p) r^p \int_Q \left[|x_0|^{2\alpha} |\nabla_xf(\xi)|^2 + |x_0|^{2\beta} |\nabla_y f(\xi)|^2 \right]^{\frac{p}{2}} d\xi.
\end{equation*}
Since $B$ is remote, for all $(x,y)\in Q$,
$|x|\ge (1-c_2\varepsilon_0)|x_0|$.  Shrinking
$\varepsilon_0$ if necessary, we therefore have
$|x_0|\le (1-c_2\varepsilon_0)^{-1}|x|\le 2|x|$, and hence
\begin{align*}
    \int_{Q} |f(\xi) - f_{Q}|^p d\xi &\le C(n,m,c_2,p) r^p \int_Q \left[|x|^{2\alpha} |\nabla_xf(\xi)|^2 + |x|^{2\beta} |\nabla_y f(\xi)|^2 \right]^{\frac{p}{2}} d\xi \\
    &= C(n,m,c_2,p) r^p \int_Q |\nabla_L f(\xi)|^p d\xi.
\end{align*}
Moreover, Lemma~\ref{le_volume} gives $Q \subset \frac{c_2}{c_1} B=: \gamma B$, where $\gamma \ge 1$. It follows that
\begin{align*}
    \int_B |f-f_B|^p d\xi &\le 2^p \int_B |f-f_Q|^p d\xi \\
    &\le C(n,m,c_2,p) r^p \int_Q |\nabla_L f(\xi)|^p d\xi \\
    &\le C(n,m,c_2,p) r^p \int_{\gamma B} |\nabla_L f(\xi)|^p d\xi,
\end{align*}
which completes the proof.
\end{proof}

We shall also use the following local form of the standard telescoping
estimate; see \cite{HK}.

\begin{lemma}\label{le_tele}
Let $n,m\ge 1$, $\alpha\in (0,1)$ and $\beta \ge 0$. Let $B=B(\xi_0,r) = B((x_0,y_0),r)$ be a ball such that $3B$ is remote. Then there exists a constant $C>0$, depending only on the doubling constant and on $C_P$, such that for all $f\in C^{1}(2\gamma B)$ and for all $\xi,\eta\in B$,
\begin{equation}\label{eq:remote-lemma-312}
|f(\xi)-f(\eta)|
\le
C\, d(\xi,\eta)\Big(M_2(|\nabla_L f|\mathbf{1}_{2\gamma B})(\xi)+M_2(|\nabla_L f|\mathbf{1}_{2\gamma B})(\eta)\Big),
\end{equation}
where
\[
M_2 h(\xi):=\sup_{B\ni \xi}\left(\frac{1}{|B|}\int_B |h|^2 \right)^{1/2}.
\]
\end{lemma}

\begin{proof}
Put $g=|\nabla_Lf|\mathbf1_{2\gamma B}$.
For $\zeta\in B$ and $0<s\le r$, the balls
$B(\zeta,2^{-j}s)$ are remote: $3B$ is remote and
$|x_\zeta|\ge |x_0|/2$ imply $s\le\varepsilon_0|x_\zeta|^{1-\alpha}$.
Their $\gamma$-dilates lie in $2\gamma B$. Telescoping their means,
using Lemma~\ref{PE} and doubling, gives
\[
 |f(\zeta)-f_{B(\zeta,s)}|
 \le C\sum_{j\ge0}2^{-j}s\,M_2g(\zeta)
 \le Cs\,M_2g(\zeta).
\]
If $\rho=d(\xi,\eta)<r/8$, compare the means on
$B(\xi,2\rho)$ and $B(\eta,\rho)$ using their inclusion in
$B(\xi,2\rho)$ and Lemma~\ref{PE}. If $\rho\ge r/8$,
compare both point values with $f_B$ through $2B$.
The resulting bounds are respectively
$C\rho(M_2g(\xi)+M_2g(\eta))$ and
$Cr(M_2g(\xi)+M_2g(\eta))$; the latter is at most eight times the former.
This proves the assertion. See \cite[Appendix~A]{HeGrushin} for the
expanded telescoping argument.
\end{proof}

\subsection{Modified Bessel functions}

Throughout, $I_\nu$ and $K_\nu$ denote the standard modified Bessel functions. We shall use several classical properties of these functions, including their power-series expansions, recurrence relations, Wronskian identity, integral representations, and small- and large-argument asymptotics; see Watson~\cite[Chapter~III]{Watson} and Abramowitz--Stegun~\cite[Chapter~9]{AS}. For $\nu>-1$ and $t>0$, we also introduce the normalized functions
\begin{equation*}
    \mathcal I_\nu(t):=t^{-\nu}I_\nu(t),
    \qquad
    \mathcal K_\nu(t):=t^\nu K_\nu(t).
\end{equation*}
We further set
\begin{equation*}
    S_\nu(t):=\frac{I_{\nu+1}(t)}{I_\nu(t)}.
\end{equation*}
The functions $I_\nu$ and $K_\nu$ form a fundamental system of solutions to the modified Bessel equation
\begin{equation*}
    t^2F''(t)+tF'(t)-\bigl(t^2+\nu^2\bigr)F(t)=0.
\end{equation*}
For each fixed $\nu>-1$, one has
\begin{equation*}
    I_\nu(t)\sim_\nu
    \begin{cases}
        t^\nu, & 0<t\leq 1,\\
        t^{-1/2}e^t, & t\geq 1.
    \end{cases}
\end{equation*}
Since $K_{-\nu}=K_\nu$, for each fixed $\nu\in\mathbb R$ one has
\begin{equation*}
    K_\nu(t)\sim_\nu
    \begin{cases}
        t^{-|\nu|}, & 0<t\leq 1,\quad \nu\neq 0,\\
        1+|\log t|, & 0<t\leq 1,\quad \nu=0,\\
        t^{-1/2}e^{-t}, & t\geq 1.
    \end{cases}
\end{equation*}
We record the uniform estimates needed below, with short proofs. Expanded
derivations appear in \cite[Appendix~B]{HeGrushin}.

\begin{lemma}\label{lem:Bessel_regular_branch}
Let $\nu_*>-1$. For every integer $N\geq 0$, there exists a constant
$C_{N,\nu_*}>0$ such that, for every $\nu\geq\nu_*$,
\begin{equation*}
    \sup_{0\leq t\leq 1}
    \left|
        \frac{d^k}{dt^k}\mathcal I_\nu(t)
    \right|
    \leq C_{N,\nu_*},
    \qquad 0\leq k\leq N.
\end{equation*}
Here $\mathcal I_\nu$ is understood at $t=0$ through its smooth extension. Moreover,
\begin{equation*}
    \frac{d}{dt}\mathcal I_\nu(t)
    =t\mathcal I_{\nu+1}(t).
\end{equation*}
\end{lemma}

\begin{proof}
The power series is
\[
 \mathcal I_\nu(t)=\frac{2^{-\nu}}{\Gamma(\nu+1)}
 \sum_{q\ge0}\frac{(t^2/4)^q}{q!(\nu+1)_q}.
\]
For $\nu\ge\nu_*>-1$, its coefficients and every fixed number of
derivatives are uniformly summable on $[0,1]$. The derivative identity
follows from $(t^{-\nu}I_\nu)'=t^{-\nu}I_{\nu+1}$.
\end{proof}

\begin{lemma}\label{lem_S_derivative_bound}
Let $\nu_*>-1$. For every integer $N\geq 0$, there exist constants
$C_{N,\nu_*},A_N>0$ such that, for every $\nu\geq\nu_*$,
\begin{equation*}
    \left|S_\nu^{(k)}(t)\right|
    \leq
    C_{N,\nu_*}(1+\nu+t)^{A_N},
    \qquad
    0\leq k\leq N,\quad t>0.
\end{equation*}
Consequently, for every multi-index $\gamma$ satisfying $|\gamma|\leq N$,
\begin{equation*}
    \left|
        \partial_Y^\gamma\bigl(|Y|S_\nu(|Y|)\bigr)
    \right|
    \leq
    C_{N,\nu_*}(1+\nu+|Y|)^{A_N},
    \qquad
    Y\in\mathbb R^m.
\end{equation*}
\end{lemma}

\begin{proof}
Write
\[
 E_\nu(s)=\sum_{q\ge0}\frac{(s/4)^q}{q!(\nu+1)_q},
 \qquad
 S_\nu(t)=\frac{t}{2(\nu+1)}
                 \frac{E_{\nu+1}(t^2)}{E_\nu(t^2)}.
\]
The power series gives $E_\nu\ge1$ on $[0,1]$ and
\begin{equation}\label{eq_A_derivatives_small}
 \sup_{\nu\ge\nu_*}\sup_{0\le s\le1}|E_\nu^{(k)}(s)|
 \le C_{N,\nu_*},\qquad k\le N.
\end{equation}
This proves the small-argument estimates, including those in $Y$, since
$|Y|S_\nu(|Y|)$ is a smooth function of $|Y|^2$ with uniformly bounded
derivatives on the unit ball. For $t\ge1$, use the standard ratio bound
$0<S_\nu(t)\le C_{\nu_*}$ \cite{KB} and the Riccati equation
\[
 S_\nu'=1-S_\nu^2-\frac{2\nu+1}{t}S_\nu.
\]
Differentiating this equation proves the polynomial bounds by induction.
The usual radial derivative formula gives the assertion for $|Y|\ge1$.
\end{proof}

\begin{lemma}\label{lem_bessel_inputs_high_frequency}
Let $\nu_*>-1$. There exist constants $c,C>0$, depending only on $\nu_*$, such that, for every $\nu\geq\nu_*$,
\begin{equation}\label{eq_S_lower_large}
    \frac{I_{\nu+1}(t)}{I_\nu(t)}
    \geq c,
    \qquad
    t\geq C(1+\nu).
\end{equation}
Moreover, the function $\mathcal K_\nu(t)=t^\nu K_\nu(t)$ is decreasing on $(0,\infty)$, and
\begin{equation}\label{eq_K_normalized_exp_decay}
    \frac{\mathcal K_\nu(t_2)}{\mathcal K_\nu(t_1)}
    \leq
    C\exp\bigl\{-c(t_2-t_1)\bigr\},
    \qquad
    t_2\geq t_1\geq C(1+\nu).
\end{equation}
In addition, for every $\nu\geq 0$ and $t>0$,
\begin{equation}\label{eq:product:Bessel}
    I_\nu(t)K_\nu(t)
    \leq
    \frac{C}{t+\nu}.
\end{equation}
\end{lemma}

\begin{proof}
The Amos-type ratio bound \cite[Corollary~1]{KB},
\[
 S_\nu(t)\ge
 \frac{t}{\nu+\tfrac12+\sqrt{(\nu+\tfrac32)^2+t^2}},
\]
gives \eqref{eq_S_lower_large}. The identity
$(t^\nu K_\nu)'=-t^\nu K_{\nu-1}$ proves monotonicity. Moreover, the
integral representation
$K_\nu(t)=\int_0^\infty e^{-t\cosh s}\cosh(\nu s)\,ds$
gives $-K_\nu'/K_\nu\ge1$, whence
\[
 \frac{K_{\nu-1}(t)}{K_\nu(t)}\ge1-\frac{\nu}{t}\ge\frac12
 \quad\text{if }t\ge C(1+\nu).
\]
Integration of
$(\log\mathcal K_\nu)'=-K_{\nu-1}/K_\nu$
proves \eqref{eq_K_normalized_exp_decay}.
Finally, the Wronskian gives
\[
 I_\nu K_\nu
 \left(\frac{I_\nu'}{I_\nu}-\frac{K_\nu'}{K_\nu}\right)=\frac1t.
\]
For $\nu\ge0$, the factor in parentheses is at least $1$ and,
when $\nu>0$, at least $2\nu/t$. Hence
$I_\nu(t)K_\nu(t)\le\min\{t^{-1},(2\nu)^{-1}\}
\le C/(t+\nu)$; for $\nu=0$, use the first bound.
\end{proof}

The product estimate \eqref{eq:product:Bessel} is used only for nonnegative
orders; the negative-order even branch is handled separately.

We shall also use the following elementary integral estimates repeatedly.
Here $I(q,s,\nu,\varepsilon,t)$ and $II(q,s,\nu,t)$ always denote the
complete integral expressions displayed below, not the modified Bessel
functions themselves.

\begin{lemma}\label{lem:Bessel:integral}
Let $\nu>-1$, let $q,s\in\mathbb R$, and let
$\varepsilon\in\{0,1\}$. The following two estimates hold:
\begin{enumerate}[label=(\roman*)]
    \item if $s>|\nu|$, $q\ge -\nu - \varepsilon$, and $q+s\le 2$, then
    \begin{equation*}
        \sup_{t>0}I(q,s,\nu,\varepsilon,t):=\sup_{t>0} t^q I_{\nu+\varepsilon}(t) \int_t^\infty u^{s-1} K_{|\nu|}(u) du \le C,
    \end{equation*}
    \item if $s+\nu>0$, $q+s+\nu-|\nu|\ge 0$, and $q+s\le 2$, then
    \begin{equation*}
        \sup_{t>0}II(q,s,\nu,t):= \sup_{t>0} t^q K_{|\nu|}(t) \int_0^t u^{s-1} I_{\nu}(u) du \le C.
    \end{equation*}
    In the logarithmic case $\nu=0$, assume additionally that $q+s>0$.
\end{enumerate}
\end{lemma}

\begin{proof}
We first prove (i). Suppose $0<t\leq1$ and $\nu\ne0$. The small-argument
asymptotics of the modified Bessel functions give
\begin{align*}
    I(q,s,\nu,\varepsilon,t)
    &\leq C t^{q+\nu+\varepsilon}
      \left(
        \int_t^1u^{s-1-|\nu|}\,du
        +\int_1^\infty u^{s-\frac32}e^{-u}\,du
      \right)  \\
    &\leq C t^{q+\nu+\varepsilon}\leq C.
\end{align*}
Here $s>|\nu|$ makes the first integral uniformly bounded, and
$q+\nu+\varepsilon\geq0$ controls the remaining power of $t$.
For $\nu=0$, replace the first integrand by
$u^{s-1}(1+|\log u|)$, which is integrable because $s>0$.
If $t>1$,
the large-argument asymptotics instead yield
\begin{equation*}
    I(q,s,\nu,\varepsilon,t)
    \leq
    C t^{q-\frac12}e^t
      \int_t^\infty u^{s-\frac32}e^{-u}\,du
    \leq C t^{q+s-2}\leq C,
\end{equation*}
where we used
$\int_t^\infty u^\delta e^{-u}\,du\leq C_\delta t^\delta e^{-t}$ for
$t\geq1$ and $\delta\in\mathbb R$.

We next prove (ii). If $0<t\leq1$ and $\nu\ne0$, then
\begin{equation*}
    II(q,s,\nu,t)
    \leq C t^{q-|\nu|}\int_0^t u^{s-1+\nu}\,du
    \leq C t^{q+s+\nu-|\nu|}\leq C.
\end{equation*}
The integral is finite because $s+\nu>0$. When $\nu=0$, the factor
$K_0(t)$ contributes $1+|\log t|$, and the strict inequality $q+s>0$
absorbs this logarithm. If $t>1$, the large-argument asymptotics give
\begin{align*}
   II(q,s,\nu,t)
   &\leq C t^{q-\frac12}e^{-t}
      \left(1+\int_1^t u^{s-\frac32}e^u\,du\right)\\
   &\leq C+C t^{q+s-2}\leq C.
\end{align*}
This proves the lemma.
\end{proof}

\noindent\textbf{Notation.}
Throughout the article, for $n,m\geq 1$, $\alpha\in(0,1)$, and $\beta\geq 0$, we set
\begin{equation*}
    \mathfrak a:=\beta+1-\alpha,
    \qquad
    \mathfrak b:=\frac{n+2\alpha-2}{2},
    \qquad
    \mathcal Q
    :=\frac{n+m(\beta+1-\alpha)}{1-\alpha}
    =\frac{n+m\mathfrak a}{1-\alpha}.
\end{equation*}
The symbols $c$ and $C$ denote positive constants that may vary from line to line.

\part{Riesz transform for \texorpdfstring{$n\ge 2$}{n>=2}}\label{part1}

In this part we prove the full-range Riesz transform estimate when the degenerate variable has dimension $n\geq2$. By the abstract criterion recalled in Section~\ref{sec3}, it is enough to prove reverse H\"older inequalities for gradients of Grushin harmonic functions. Remote balls are controlled by standard interior elliptic estimates after intrinsic rescaling, since the coefficients are essentially non-degenerate there. The only new difficulty is the anchored case, where the ball meets the singular set.

After rescaling an anchored ball to a fixed cylinder, we represent a localized harmonic function as the sum of a Poisson term and a Green term. The singularity at $\rho=0$ is encoded in modified Bessel functions after Fourier transform in $y$ and spherical harmonic decomposition in the angular variable. Unlike the one-dimensional analysis of Part~\ref{part2}, the Bessel order now depends on the angular frequency $\ell$ and tends to infinity with it. Thus the proof requires Bessel estimates which are uniform simultaneously in the order, the Fourier frequency, and the distance to the singular set.

\section{Reduction}\label{sec3}

From Robinson--Sikora \cite{RS}, it is well known that the Grushin space
satisfies the volume doubling condition \eqref{D} and the Gaussian upper bound
for the heat kernel \eqref{UE}; see also Theorem~\ref{RS_Grushin}. Moreover,
by \cite[Theorem~1.1]{RS2}, when either $n\geq2$, or $n=1$ and
$0<\alpha<1/2$, there exist fixed constants $C>0$ and $\lambda\geq1$ such
that
\begin{equation}\label{P2}\tag{$\mathrm P_2$}
    \left(\fint_B|f-f_B|^2\,d\xi\right)^{1/2}
    \leq
    Cr(B)
    \left(\fint_{\lambda B}|\nabla_Lf|^2\,d\xi\right)^{1/2}
\end{equation}
for every Grushin ball $B$ and every $f\in C^\infty(\lambda B)$.

By \cite{CJKS}, it is known that, under suitable conditions, \eqref{R_p} and \eqref{eq:RH:Jiang} are equivalent. More precisely,

\begin{theorem}[Coulhon--Jiang--Koskela--Sikora]
\cite[Theorem~1.9]{CJKS}\label{thm:CJKS}
Let $X$ be a non-compact Dirichlet metric measure space endowed with a
carr\'e du champ. Assume that $X$ satisfies \eqref{D} and \eqref{UE}, and
supports the fixed-enlargement Poincar\'e inequality \eqref{P2}. Then, for
every $2<p<\infty$, \eqref{eq:RH:Jiang} and \eqref{R_p} are equivalent.
\end{theorem}

\begin{remark}\label{rem:CJKS-enlarged-Poincare}
We emphasize the following points.
\begin{enumerate}[label=(\roman*)]
\item The proof of \cite[Theorem~1.9]{CJKS} is written with a same-ball
Poincar\'e hypothesis, but it applies to \eqref{P2}. Indeed, for $p>2$,
Cauchy--Schwarz and monotonicity of normalized $L^p$ norms give the
corresponding enlarged-ball $L^p$ estimate. One then carries the fixed
dilate $\lambda B$ through the Poincar\'e and Caccioppoli steps in the proof
of \cite[Theorem~1.9]{CJKS}; volume doubling absorbs every resulting fixed
enlargement. Thus the Robinson--Sikora form \eqref{P2} is sufficient, and no
second Poincar\'e notation is needed.

\item By \cite[Theorem~1.1]{RS2}, it is known that \eqref{P2} holds on Grushin spaces if either $n\ge 2$, or $n=1$ and $0<\alpha<1/2$. Thus, in these two cases, the Riesz transform problem is reduced to proving \eqref{eq:RH:Jiang}.
\item In the special case $n=1$ and $1/2\le \alpha<1$, the global condition \eqref{P2} fails, and hence Theorem~\ref{thm:CJKS} cannot be applied. In \cite{Jiang}, Jiang refines the above theorem by relaxing \eqref{P2} to a so-called Poincar\'e inequality on the ends, which requires \eqref{P2} only for balls outside a compact set. This refinement makes the equivalence stable under gluing operations, for example on manifolds or Dirichlet spaces with ends. However, in our strongly degenerate case, \eqref{P2} holds only for balls outside the singular hyperplane $\{x=0\}$, which is not compact. Therefore, Jiang's theorem does not apply either. We will address this issue later in Part~\ref{part4}.

\end{enumerate}
\end{remark}

The following two lemmas further reduce the problem to establishing reverse H\"older estimates on anchored balls.

\begin{lemma}\label{reduce_RH}
Let $n\ge 1$, $m\ge 1$, $\alpha\in (0,1)$ and $\beta \ge 0$. Let $1<p\le \infty$. The $L^p$-reverse H\"older inequality \eqref{eq:RH:Jiang} holds provided that it holds for all anchored balls of the form $B((0,y_0), r)$ and for all remote balls $B((x_0,y_0), r)$ with $r\le \varepsilon_0 |x_0|^{1-\alpha}$.
\end{lemma}

\begin{proof}
Suppose that $B=B(\xi_0,r) = B((x_0,y_0), r)$ is neither anchored nor remote. That is, $x_0 \ne 0$ and $r>\varepsilon_0 |x_0|^{1-\alpha}$. Clearly, $B((x_0,y_0), r) \subset B((0,y_0), C_1 r)\subset B((0,y_0), \lambda C_1 r) \subset B((x_0,y_0), C_2 r)$ for some $C_2 \ge C_1 \ge 1$. Therefore, by the $L^p$-reverse H\"older inequality for anchored balls and \eqref{D},
\begin{align*}
    \left( \fint_B |\nabla_L u|^p d\xi \right)^{\frac{1}{p}}
    &\le C \left( \fint_{B((0,y_0), C_1 r)} |\nabla_L u|^p d\xi \right)^{\frac{1}{p}} \\
    &\le \frac{C}{r} \fint_{B((0,y_0), \lambda C_1 r)} |u| d\xi \\
    &\le \frac{C}{r} \fint_{C_2 B} |u| d\xi
\end{align*}
as required.
\end{proof}

\begin{lemma}\label{RH_remote}
Let $n\ge 1$, $m\ge 1$, $\alpha\in (0,1)$ and $\beta \ge 0$. The $L^p$-reverse H\"older inequality holds for remote balls $B((x_0,y_0), r)$ with $r\le \varepsilon_0 |x_0|^{1-\alpha}$ for all $1<p\le \infty$.
\end{lemma}

\begin{proof}
Let $B=B((x_0,y_0),r)$ be remote, and suppose that $u$ is $L$-harmonic
in $4B$. It is enough to treat real-valued $u$, since the resulting estimate
may be applied separately to the real and imaginary parts. After decreasing
$\varepsilon_0$ if necessary,
Lemma~\ref{le_volume}, applied at scale $4r$, gives
\begin{equation}\label{eq:remote-x-comparability}
    |x|\simeq |x_0|
    \qquad\text{on }4B.
\end{equation}
This does not require $4B$ itself to be remote with the same value of
$\varepsilon_0$.

Set $\mathcal L=-L$. By Lemma~\ref{Grushin_gradient}(i) and
\eqref{eq:remote-x-comparability}, there are structural constants
$N_0<\infty$ and $C>0$ such that on $4B$
\begin{equation*}
    \Gamma_2(f)
    \geq
    -K_B\Gamma(f)+\frac1{N_0}(\mathcal Lf)^2,
    \qquad
    K_B:=\frac{C}{|x_0|^{2(1-\alpha)}}.
\end{equation*}
Moreover,
\begin{equation*}
    r\sqrt{K_B}
    \leq Cr|x_0|^{\alpha-1}
    \leq C\varepsilon_0.
\end{equation*}

The curvature--dimension argument is local. By
\eqref{eq:remote-x-comparability}, the intrinsically rescaled diffusion
coefficients and reference density have uniform smooth bounds on a fixed
neighbourhood of $4B$. They may be extended smoothly outside a slightly
larger coordinate box to a complete weighted manifold, Euclidean outside
a compact set, with a structural finite-dimensional curvature lower
bound. The local Cheng--Yau estimate for smooth diffusion operators in
\cite[Theorem~5.1]{LiXD} is therefore applicable. It implies that every
positive $\mathcal L$-harmonic function $w$ in $2B$ satisfies
\begin{equation}\label{eq:local-Cheng-Yau}
    \sup_B|\nabla_L\log w|
    \leq
    C\left(\frac1r+\sqrt{K_B}\right)
    \leq\frac Cr.
\end{equation}
For comparison, the corresponding direct harmonic estimate in the
$\operatorname{RCD}^*(K,N)$ setting is
\cite[Theorem~1.6]{ZhangZhu}.

Put $M:=\|u\|_{L^\infty(2B)}$ and, for $\delta>0$, set
$w_\delta:=u+M+\delta$. Then $w_\delta$ is positive and
$\mathcal L$-harmonic in $2B$, and $0<w_\delta\leq2M+\delta$ there.
Consequently, \eqref{eq:local-Cheng-Yau} gives
\begin{equation*}
    \sup_B|\nabla_Lu|
    =
    \sup_B|\nabla_Lw_\delta|
    \leq
    \frac Cr(2M+\delta).
\end{equation*}
Letting $\delta\downarrow0$ and using the local mean-value estimate
\cite[Proposition~2.1]{CJKS}, available under \eqref{D} and \eqref{UE},
we obtain
\begin{equation}\label{meanvalue}
    \|u\|_{L^\infty(2B)}
    \leq
    C\fint_{4B}|u|\,d\xi.
\end{equation}
Thus
\begin{equation*}
    \sup_B|\nabla_Lu|
    \leq
    \frac Cr\fint_{4B}|u|\,d\xi.
\end{equation*}
This is $(\mathrm{RH}_\infty)$ on remote balls, and the finite-$p$
estimates follow immediately.
\end{proof}

Combining Lemma~\ref{reduce_RH} and Lemma~\ref{RH_remote}, the $L^p$-boundedness of the Riesz transform is now reduced to proving the following result.

\begin{theorem}\label{thm:RH}
Let $n\ge 2$, $m\ge 1$, $\alpha\in (0,1)$ and $\beta \ge 0$. There is a
fixed structural constant $\Lambda_0>1$ with the following property.  If
$B=B\big((0,y_0), r\big)$ is an anchored Grushin ball and $u$ is
$L$-harmonic in $\Lambda_0B$, then, for every $2<p\le\infty$,
\begin{equation*}
    \left( \fint_B |\nabla_L u|^p d\xi \right)^{\frac{1}{p}}
    \le  \frac{C}{r} \fint_{\Lambda_0B} |u| d\xi,
\end{equation*}
with the usual supremum interpretation when $p=\infty$.  In particular,
the reverse H\"older inequality \eqref{eq:RH:Jiang} holds.
\end{theorem}

\begin{proof}[Reduction of Theorem~\ref{thm:RH}]
For $R>0$, consider the Grushin box
\begin{equation*}
Q_R := B_n(0,R) \times B_m \left(0, R^{\mathfrak a} \right).
\end{equation*}
By Lemma~\ref{le_volume},
\begin{equation*}
B\big((0,y_0), r\big) \sim B_n\left(0, r^{\frac{1}{1-\alpha}}\right) \times B_m\left(y_0, r^{\frac{\mathfrak a}{1-\alpha}}\right).
\end{equation*}
Define
\begin{equation*}
v(\xi) = v(x,y) := u\left( r^{\frac{1}{1-\alpha}} x, y_0 + r^{\frac{\mathfrak a}{1-\alpha}} y \right).
\end{equation*}
Choose the structural constant $\Lambda_0$ so large that the image of
$Q_4$ under this rescaling is contained in $\Lambda_0B$.  Then $v$ is
$L$-harmonic in $Q_4$. After the change of variables
\begin{equation*}
x \mapsto r^{\frac{1}{1-\alpha}}x, \qquad y \mapsto y_0+ r^{\frac{\mathfrak a}{1-\alpha}}y,
\end{equation*}
it is enough to prove
\begin{equation}\label{eq_RHv}
\left( \fint_{Q_1}|\nabla_L v|^p d\xi \right)^{\frac{1}{p}} \le C \fint_{Q_4} |v| d\xi.
\end{equation}
Moreover, by \eqref{D} and \eqref{UE}, \cite[Proposition~2.1]{CJKS} gives
\begin{equation*}
\| v \|_{L^\infty(Q_3)} \le C \fint_{Q_4} |v| d\xi.
\end{equation*}
Therefore, \eqref{eq_RHv} follows for all $2<p\le \infty$ once
\begin{equation}\label{eq_RH_show}
\sup_{Q_1} |\nabla_L v| \le C \| v \|_{L^\infty(Q_3)}
\end{equation}
is established.

Passing to polar coordinates,
\begin{equation*}
x = \rho \theta, \quad \rho = |x|,\quad \theta \in \mathbb{S}^{n-1},
\end{equation*}
a straightforward computation shows that
\begin{equation*}
L = - \rho^{2\alpha} \partial_\rho^2 - (n-1+2\alpha) \rho^{2\alpha-1} \partial_\rho + \rho^{2\alpha -2} \Delta_{\mathbb{S}^{n-1}} + \rho^{2\beta} \Delta_y.
\end{equation*}
Choose $\chi \in C_c^\infty\left( B_m\big(0, 3^{\mathfrak a}\big) \right)$ such that $\chi = 1$ on $B_m\big(0, 2^{\mathfrak a}\big)$, and define
\begin{equation*}
\omega(\rho,\theta,y) = \chi(y) v(\rho,\theta,y).
\end{equation*}
Then $\omega$ has locally finite form energy, is smooth away from $x=0$,
and is compactly supported in the $y$-variable. Moreover, $\omega = v$ on $Q_2$, and
\begin{equation*}
\sup_{Q_1}|\nabla_L v| = \sup_{Q_1}|\nabla_L \omega|.
\end{equation*}
Consider the cylinder
\begin{equation*}
\mathcal{C}:= (0,2)_\rho \times \mathbb{S}^{n-1}_\theta \times \mathbb{R}^m_y.
\end{equation*}
The function $\omega$ satisfies
\begin{equation*}
L \omega = f:= [L,\chi]v = |x|^{2\beta} v \Delta_y \chi - 2|x|^{2\beta} \nabla_y \chi \cdot \nabla_y v.
\end{equation*}
Clearly, $f$ is supported in $\{2^{\mathfrak a}\le |y|\le 3^{\mathfrak a}\}$. Let
\begin{equation*}
h(\theta,y) := \omega(2,\theta,y).
\end{equation*}
Let $L_{\mathcal C,D}$ denote the self-adjoint operator associated with the
restriction of the Friedrichs form to functions with zero trace at
$\rho=2$.  Here $h$, being the trace of $\omega$ on the regular outer
boundary, belongs to the natural form trace space, while
$f=[L,\chi]v$ belongs to the dual of the zero-trace form domain and is
supported in the stated $y$-annulus.  The following identities are first
obtained for smooth data in these spaces and then extended by density.  In
this sense, the Friedrichs representation gives
\begin{equation*}
\omega = \mathcal{T}h + \mathcal{S}f,
\end{equation*}
where $\mathcal{T}h$ is the solution to the boundary value problem
\begin{align}\label{BVP}
\begin{cases}
Lg = 0, & \textrm{on}\quad \mathcal{C},\\
g = h, & \textrm{when}\quad \rho=2,\\
g\quad \textrm{has finite energy at}\quad \rho=0,
\end{cases}
\end{align}
and $\mathcal{S}f$ is the solution to the Dirichlet problem
\begin{align}\label{DP}
\begin{cases}
Lg = f, & \textrm{on}\quad \mathcal{C},\\
g = 0, & \textrm{when}\quad \rho=2,\\
g\quad \textrm{has finite energy at}\quad \rho=0.
\end{cases}
\end{align}
Indeed, the remainder
$r:=\omega-\mathcal T h-\mathcal S f$ has zero trace at $\rho=2$ and
satisfies $\mathcal q_{\mathcal C}(r,\varphi)=0$ for every zero-trace form
test function $\varphi$.  Taking $\varphi=r$ and using uniqueness for the
Dirichlet Friedrichs form gives $r=0$.
Accordingly, \eqref{eq_RH_show}, and hence Theorem~\ref{thm:RH}, will follow once the two estimates
\begin{align*}
\sup_{Q_1} |\nabla_L \mathcal{T}h| \le C \| v \|_{L^\infty(Q_3)},
\end{align*}
and
\begin{align*}
\sup_{Q_1} |\nabla_L \mathcal{S}f| \le C \| v \|_{L^\infty(Q_3)}
\end{align*}
are proved.
\end{proof}

\section{Kernel construction}

We now construct the Poisson and Green kernels needed for the anchored estimate. The singularity at $\rho=0$ is encoded in a family of modified Bessel equations after Fourier transform in $y$ and spherical harmonic decomposition in the angular variable. The
finite-energy condition selects the regular $I$-branch and excludes the singular $K$-branch. The remainder of the section records the resulting kernels and the estimates required for the proof of the reverse H\"older inequality.

Throughout Part~\ref{part1}, for each $\ell \ge 0$, we use the notation
\begin{equation*}
\lambda_\ell:= \ell(\ell+n-2),\qquad \nu_\ell:= \frac{\sqrt{\mathfrak b^2 + \lambda_\ell}}{\mathfrak a}, \qquad \gamma_\ell:= - \mathfrak b + \sqrt{\mathfrak b^2 + \lambda_\ell}, \qquad \delta_\ell := \mathfrak b+\sqrt{\mathfrak b^2+\lambda_\ell}.
\end{equation*}

\subsection{The Poisson multiplier}

Let $\{\Theta_{\ell,\mu}\}_{\substack{\ell\ge 0\\ 1\le \mu\le d_\ell}}$ be the spherical harmonics on $\mathbb{S}^{n-1}$, that is,
\begin{equation*}
    \Delta_{\mathbb{S}^{n-1}} \Theta_{\ell,\mu} = \lambda_{\ell} \Theta_{\ell,\mu},
\end{equation*}
where $d_\ell \sim (1+\ell)^{n-2}$ denotes the dimension of the space of spherical harmonics of degree $\ell$; see \cite[Chapter~IV]{SW2}. Consider the equation
\begin{equation*}
    - \rho^{2\alpha} \partial_\rho^2 g - (n-1+2\alpha) \rho^{2\alpha-1} \partial_\rho g + \rho^{2\alpha -2} \Delta_{\mathbb{S}^{n-1}}g + \rho^{2\beta} \Delta_y g = 0.
\end{equation*}
Taking the Fourier transform in the $y$-variable and then expanding in spherical harmonics gives
\begin{align*}
    0 &= - \rho^{2\alpha} \partial_\rho^2 \hat{g} - (n-1+2\alpha) \rho^{2\alpha-1} \partial_\rho \hat{g} + \rho^{2\alpha -2} \Delta_{\mathbb{S}^{n-1}} \hat{g} + \rho^{2\beta} |Y|^2 \hat{g}\\
    &= \sum_{\ell=0}^\infty \sum_{\mu=1}^{d_\ell} \left(  - \rho^{2\alpha} \partial_\rho^2 a_{\ell,\mu} - (n-1+2\alpha) \rho^{2\alpha-1} \partial_\rho a_{\ell,\mu} + \rho^{2\alpha -2} \lambda_{\ell} a_{\ell,\mu} + \rho^{2\beta} |Y|^2 a_{\ell,\mu}  \right) \Theta_{\ell,\mu}(\theta),
\end{align*}
where
\begin{equation*}
    \hat{g}(\rho,\theta,Y) = \int_{\mathbb{R}^m} e^{-iy \cdot Y} g(\rho,\theta,y)\, dy,
\end{equation*}
and we use the inverse convention
\begin{equation*}
 g(\rho,\theta,y)=(2\pi)^{-m}\int_{\mathbb R^m}
 e^{iy\cdot Y}\widehat g(\rho,\theta,Y)\,dY.
\end{equation*}
Accordingly, every inverse Fourier integral in the kernel identities below
carries the factor $(2\pi)^{-m}$; in estimates this harmless fixed factor
is absorbed into the constants.
The spherical coefficients are
\begin{equation*}
    a_{\ell,\mu}(\rho,Y) = \int_{\mathbb{S}^{n-1}} \hat{g}(\rho,\theta,Y)\, \overline{\Theta_{\ell,\mu}(\theta)}\, d\theta,
\end{equation*}
and
\begin{equation}\label{eq_spherical_decom}
    \hat{g}(\rho,\theta,Y) = \sum_{\ell=0}^\infty \sum_{\mu=1}^{d_\ell} a_{\ell,\mu}(\rho,Y) \Theta_{\ell,\mu}(\theta).
\end{equation}
By orthogonality,
\begin{equation}\label{eq_equation1}
    \rho^{2} \partial_\rho^2 a_{\ell,\mu} + (n-1+2\alpha) \rho \partial_\rho a_{\ell,\mu} -  \lambda_{\ell} a_{\ell,\mu} - \rho^{2\beta-2\alpha+2} |Y|^2 a_{\ell,\mu} = 0
\end{equation}
for each $\ell,\mu$. Introduce
\begin{equation*}
    t:= \frac{|Y|}{\mathfrak a} \rho^{\mathfrak a}
    \quad \textrm{and}\quad
    U_{\ell,\mu,Y}(t):= \rho^{\mathfrak b} a_{\ell,\mu}(\rho,Y).
\end{equation*}
A direct computation shows that \eqref{eq_equation1} is equivalent to
\begin{equation}\label{eq_equation2}
    t^2 U_{\ell,\mu,Y}'' + t U_{\ell,\mu,Y}' - \left( t^2 + \nu_\ell^2 \right)U_{\ell,\mu,Y} = 0.
\end{equation}
The solutions of the modified Bessel equation \eqref{eq_equation2} are linear combinations of $I_{\nu_\ell}(t)$ and $K_{\nu_\ell}(t)$, where $I_{\nu_\ell}$ and $K_{\nu_\ell}$ denote the modified Bessel functions. Therefore,
\begin{equation}\label{a_l,mu}
    a_{\ell,\mu}(\rho,Y) = C_1(\ell,\mu,Y) \rho^{-\mathfrak b} I_{\nu_\ell}\left( \frac{|Y| \rho^{\mathfrak a}}{\mathfrak a} \right) + C_2(\ell,\mu,Y) \rho^{-\mathfrak b} K_{\nu_\ell}\left( \frac{|Y| \rho^{\mathfrak a}}{\mathfrak a} \right).
\end{equation}
However, since the Friedrichs representation imposes a finite energy condition at $\rho=0$, the following claim holds.

\textit{Claim.} The finite-energy solution of the singular Sturm--Liouville equation \eqref{eq_equation1} excludes the $K_{\nu_\ell}$-branch, whose energy near $\rho=0$ is infinite.

\begin{proof}[Proof of Claim]
Let $g$ be the corresponding finite-energy solution. Then its weighted radial derivative satisfies, for sufficiently small $\varepsilon>0$,
\begin{equation*}
\int_0^\varepsilon \int_{\mathbb S^{n-1}} \int_{\mathbb R^m} \rho^{2\alpha}|\partial_\rho g(\rho,\theta,y)|^2 \rho^{n-1}\, dy\, d\theta\, d\rho < \infty.
\end{equation*}
Since $\partial_\rho$ commutes with both the Fourier transform in the $y$-variable and the spherical harmonic projection in $\theta$, it follows from \eqref{eq_spherical_decom} that
\begin{equation}\label{eq_finite_energy}
    \partial_\rho \hat{g}(\rho,\theta,Y) =
\sum_{\ell=0}^{\infty}\sum_{\mu=1}^{d_\ell}\partial_\rho a_{\ell,\mu}(\rho,Y) \Theta_{\ell,\mu}(\theta).
\end{equation}
By Plancherel's theorem and the orthogonality of the spherical harmonics,
\begin{align*}
&\int_0^\varepsilon \int_{\mathbb S^{n-1}} \int_{\mathbb R^m} \rho^{2\alpha}|\partial_\rho \hat{g}(\rho,\theta,Y)|^2 \rho^{n-1}\, dY\, d\theta\, d\rho\\
&\qquad = \sum_{\ell=0}^{\infty} \sum_{\mu=1}^{d_\ell} \int_{\mathbb R^m} \int_0^\varepsilon \rho^{2\alpha} |\partial_\rho a_{\ell,\mu}(\rho,Y)|^2 \rho^{n-1}\, d\rho\, dY.
\end{align*}
Since the left-hand side is finite, every term on the right-hand side is finite as well. Hence, for each $(\ell,\mu)$ and for a.e. $Y\in \mathbb{R}^m$,
\begin{equation}\label{eq_radial_gradient}
\int_0^\varepsilon \rho^{2\alpha} |\partial_\rho a_{\ell,\mu} (\rho,Y)|^2 \rho^{n-1}\, d\rho < \infty.
\end{equation}
Next, since $K_{\nu_\ell}(t) \sim t^{-\nu_\ell}$ for $t\le 1$, the $K_{\nu_\ell}$-branch in \eqref{a_l,mu} behaves like $\rho^{\widetilde{\gamma}_\ell}$ as $\rho \to 0$, where
\begin{equation*}
    \widetilde{\gamma}_\ell = -\mathfrak b - \sqrt{\mathfrak b^2 + \lambda_\ell}.
\end{equation*}
Therefore, its contribution to \eqref{eq_radial_gradient} is
\begin{equation*}
    \int_0^\varepsilon \rho^{2(\alpha + \widetilde{\gamma}_\ell -1 ) } \rho^{n-1}\, d\rho = \infty,\quad \forall \ell \ge 0,
\end{equation*}
since $n\ge 2$. Consequently, the entire $K_{\nu_\ell}$-branch must be excluded from the solution \eqref{a_l,mu}, that is,
\begin{equation*}
    C_2(\ell,\mu,Y) = 0,\quad \forall \ell,\mu,
    \quad\text{for a.e. }Y.
\end{equation*}
The explicit multiplier formula below then provides the continuous
representative at the remaining Fourier parameters.
This proves the \textit{Claim}.
\end{proof}

Define
\begin{equation*}
    \phi_{\ell, Y}(\rho) = \phi_{\ell}(\rho,Y):= \rho^{-\mathfrak b} I_{\nu_\ell}\left( \frac{|Y| \rho^{\mathfrak a}}{\mathfrak a} \right)
    \qquad \textrm{and}\qquad
    M_{\ell, Y}(\rho) = M_{\ell}(\rho,Y):= \frac{\phi_{\ell}(\rho,Y)}{\phi_{\ell}(2,Y)}.
\end{equation*}
The normalized ratio $M_{\ell,Y}$, and all normalized Green multipliers
below, are understood through their smooth extension at $Y=0$. Such an
extension follows either from the power series or from the form resolvent,
which is analytic in the parameter $|Y|^2$ near zero.
It follows from \eqref{BVP} and the \textit{Claim} that
\begin{align*}
    \widehat{\mathcal{T}h}(\rho,\theta,Y) = \sum_{\ell=0}^\infty \sum_{\mu=1}^{d_\ell} C_1(\ell,\mu,Y) \phi_{\ell}(\rho,Y)\, \Theta_{\ell,\mu}(\theta).
\end{align*}
Moreover, the boundary condition gives
\begin{align*}
    \widehat{\mathcal{T}h}(2,\theta,Y) = \hat{h}(\theta,Y) = \sum_{\ell=0}^\infty \sum_{\mu=1}^{d_\ell} b_{\ell,\mu}(Y) \Theta_{\ell,\mu}(\theta),
\end{align*}
where
\begin{equation*}
    b_{\ell,\mu}(Y) = \int_{\mathbb{S}^{n-1}} \hat{h}(\theta,Y)\, \overline{\Theta_{\ell,\mu}(\theta)}\, d\theta.
\end{equation*}
Orthogonality therefore implies
\begin{equation*}
    C_1(\ell,\mu,Y) = \frac{b_{\ell,\mu}(Y)}{\phi_{\ell}(2,Y)}.
\end{equation*}
Taking the inverse Fourier transform yields
\begin{align*}
    \mathcal{T}h(\rho,\theta,y) = \iint_{\mathbb{S}^{n-1}\times \mathbb{R}^m} \mathcal{K}_{\rho}(\theta,y; \theta', y') h(\theta', y')\, d\theta'\, dy',
\end{align*}
where
\begin{equation}\label{eq_kernel_of_T}
    \mathcal{K}_{\rho}(\theta,y; \theta', y') = (2\pi)^{-m} \sum_{\ell=0}^\infty \Pi_{\ell}(\theta,\theta') \int_{\mathbb{R}^m} e^{i(y-y')\cdot Y} M_{\ell}(\rho,Y)\, dY,
\end{equation}
and
\begin{equation*}
    \Pi_{\ell}(\theta,\theta') = \sum_{\mu=1}^{d_\ell} \Theta_{\ell,\mu}(\theta) \overline{\Theta_{\ell,\mu}(\theta')}.
\end{equation*}

\subsection{The Green kernel}\label{ssec:Green:n>2}

Next, we construct the Green kernel. Let $\mathcal{G}(Z,Z')$ (where $Z = (\rho,\theta,y)$ and $Z' = (\sigma,\theta',y')$) denote the Green kernel associated with $L$ on $\mathcal{C}$ under the Dirichlet boundary condition at $\rho=2$ and the local finite-energy condition at $\rho=0$. In the same spirit, by taking the Fourier transform in the $y$-$y'$ variables and the spherical harmonic decomposition in the $\theta$-$\theta'$ variables, the Green kernel takes the form
\begin{equation}\label{eq_Green's_kernel}
    \mathcal{G}(Z,Z') = (2\pi)^{-m}\sum_{\ell=0}^\infty \Pi_{\ell}(\theta, \theta') \int_{\mathbb{R}^m} e^{i(y-y') \cdot Y} G_{\ell, Y}(\rho, \sigma) \, dY,
\end{equation}
where $G_{\ell,Y}$ is the one-dimensional Green kernel associated with the operator
\begin{equation*}
    \mathcal{L}_{\ell,Y} := - \rho^{2\alpha} \partial_{\rho}^2 - (n-1+2\alpha) \rho^{2\alpha-1} \partial_\rho + \lambda_\ell \rho^{2\alpha-2} + |Y|^2 \rho^{2\beta}.
\end{equation*}
Equivalently,
\begin{equation}
    \mathcal{S}f(Z) = \iiint_{\mathcal{C}} \mathcal{G}(Z,Z') f(Z') \, \sigma^{n-1} \, d\sigma \, d\theta' \, dy'.
\end{equation}
Recall the function
\begin{equation*}
    \phi_{\ell}(\rho,Y):= \rho^{-\mathfrak b} I_{\nu_\ell}\!\left( \frac{|Y| \rho^{\mathfrak a}}{\mathfrak a} \right).
\end{equation*}
This is the solution to \eqref{eq_equation1} with finite energy at $\rho=0$. To construct the Green kernel, another independent solution satisfying the Dirichlet boundary condition at $\rho=2$ is required. Define
\begin{equation*}
    \psi_{\ell}(\rho, Y):= \rho^{-\mathfrak b} \left[ K_{\nu_\ell} \!\left( \frac{|Y| \rho^{\mathfrak a}}{\mathfrak a} \right) - \frac{K_{\nu_\ell} \!\left( \frac{|Y| 2^{\mathfrak a}}{\mathfrak a} \right)}{I_{\nu_\ell} \!\left( \frac{|Y| 2^{\mathfrak a}}{\mathfrak a} \right)} \, I_{\nu_\ell} \!\left( \frac{|Y| \rho^{\mathfrak a}}{\mathfrak a} \right) \right].
\end{equation*}
Note that $\psi_{\ell}$ solves \eqref{eq_equation1} and satisfies $\psi_{\ell}(2)=0$. Hence,
\begin{equation*}
    G_{\ell,Y}(\rho,\sigma) = \begin{cases}
        A_\ell(\sigma, Y) \, \phi_{\ell}(\rho, Y), & 0 < \rho < \sigma,\\[4pt]
        B_\ell(\sigma, Y) \, \psi_{\ell}(\rho, Y), & \sigma < \rho < 2.
    \end{cases}
\end{equation*}
Continuity of the Green kernel at $\rho=\sigma$ yields the first equation:
\begin{equation*}
    A_\ell(\rho,Y) \, \phi_{\ell}(\rho, Y) = B_\ell(\rho,Y) \, \psi_{\ell}(\rho, Y).
\end{equation*}
The derivative jump condition, namely $\mathcal{L}_{\ell,Y}G_{\ell,Y}(\rho,\sigma) = \frac{\delta_{\sigma}(\rho)}{\sigma^{n-1}}$, gives the second equation:
\begin{equation*}
    \sigma^{n+2\alpha-1} \bigl[ A_\ell(\sigma,Y) \, \phi_{\ell}'(\sigma,Y) - B_\ell(\sigma,Y) \, \psi_{\ell}'(\sigma,Y) \bigr] = 1.
\end{equation*}
Solving these equations gives
\begin{equation*}
    A_\ell(\sigma,Y) = \frac{\psi_{\ell}(\sigma,Y)}{W_{\ell,Y}}, \qquad B_\ell(\sigma,Y) = \frac{\phi_{\ell}(\sigma,Y)}{W_{\ell,Y}},
\end{equation*}
where
\begin{equation*}
    W_{\ell,Y} = t^{n+2\alpha-1} \bigl[ \phi_{\ell}'(t,Y) \, \psi_\ell(t,Y) - \phi_{\ell}(t,Y) \, \psi'_\ell(t,Y) \bigr]
\end{equation*}
is the weighted Wronskian. Using the identity
\begin{equation*}
    I_{\nu_\ell}'(t) \, K_{\nu_\ell}(t) - I_{\nu_\ell}(t) \, K_{\nu_\ell}'(t) = \frac{1}{t},
\end{equation*}
it is readily verified that $W_{\ell,Y} = \mathfrak a = \beta+1-\alpha$.

Therefore,
\begin{align}\label{eq_kernel_Green}
    G_{\ell,Y}(\rho,\sigma) = \mathfrak a^{-1}
    \begin{cases}
        \phi_{\ell}(\rho,Y) \, \psi_{\ell}(\sigma,Y), & 0 < \rho < \sigma,\\[4pt]
        \phi_{\ell}(\sigma,Y) \, \psi_{\ell}(\rho,Y), & \sigma < \rho < 2.
    \end{cases}
\end{align}

\section{Proof of Theorem~\ref{thm:RH}}

\begin{theorem}\label{lem:Poisson_kernel}
Let $h\in L^\infty(\mathbb{S}^{n-1} \times \mathbb{R}^m)$ belong to the
natural outer trace space of the cylinder Friedrichs form, and let
$\mathcal{T}h$ be the finite-energy solution of the homogeneous problem
\eqref{BVP}. Then
\begin{equation*}
    \sup_{0<\rho\le 1} \bigl\| \nabla_L \mathcal{T}h \bigr\|_{L^\infty(\mathbb{S}^{n-1} \times \mathbb{R}^m)} \le C \|h\|_{L^\infty(\mathbb{S}^{n-1} \times \mathbb{R}^m)}.
\end{equation*}
\end{theorem}

\begin{proof}
The proof is based on a multiplier-to-kernel argument. First, we estimate the three multiplier families corresponding to the radial, angular, and base components. Note that
\begin{equation*}
    \nabla_L = \bigl( \rho^\alpha \partial_\rho, \, \rho^{\alpha-1} \nabla_{\mathbb{S}^{n-1}}, \, \rho^{\beta} \nabla_y \bigr).
\end{equation*}
By the explicit formula \eqref{eq_kernel_of_T}, it is enough to estimate the kernels associated with the following three multipliers:
\begin{align*}
m_{\ell,\rho}^{\mathrm{rad}}(Y)&:= \rho^{\alpha} \partial_\rho M_{\ell}(\rho, Y),\\
m_{\ell,\rho}^{\mathrm{ang}}(Y)&:= \rho^{\alpha-1} \sqrt{\lambda_\ell} \, M_{\ell}(\rho, Y),\\
m_{\ell,\rho}^{\mathrm{bas},j}(Y)&:= i Y_j \rho^{\beta} M_{\ell}(\rho, Y), \quad 1 \le j \le m,
\end{align*}
where, in particular, $m_{0,\rho}^{\mathrm{ang}}$ is understood to be zero.

The following estimate will be used, whose proof is postponed to Proposition~\ref{lem_detailed_poisson_multiplier} below: for every integer $N \ge 0$,
\begin{equation}\label{eq_detailed_poisson_multiplier_goal_copy}
    \sum_{|\gamma|\le N}
    \bigl\|
        \partial_Y^\gamma m_{\ell,\rho}^{(*)}
    \bigr\|_{L^1(\mathbb R^m_Y)}
    \le
    C_N e^{-c_N \ell},
    \quad
    *\in \{ \mathrm{rad}, \mathrm{ang}, \mathrm{bas}, j\},
\end{equation}
uniformly for $0<\rho\le 1$ and $\ell\ge 0$.

Next, we use the identity
\begin{equation*}
    (I+\Delta_Y)^N e^{i Y \cdot (y-y')} = (1 + |y-y'|^2)^N e^{iY\cdot(y-y')}.
\end{equation*}
Together with \eqref{eq_detailed_poisson_multiplier_goal_copy}, integration by parts yields
\begin{align*}
    \biggl| \int_{\mathbb{R}^m} e^{iY\cdot(y-y')} m_{\ell,\rho}^{(*)}(Y) \, dY \biggr|
    &= (1+|y-y'|^2)^{-N} \biggl| \int_{\mathbb{R}^m} e^{iY\cdot(y-y')} (I+\Delta_Y)^N m_{\ell,\rho}^{(*)}(Y) \, dY \biggr|\\[4pt]
    &\le C_N (1+|y-y'|^2)^{-N} \sum_{|\gamma|\le 2N} \bigl\|
        \partial_Y^\gamma m_{\ell,\rho}^{(*)}
    \bigr\|_{L^1(\mathbb R^m_Y)}\\[4pt]
    &\le C_N (1+|y-y'|^2)^{-N} e^{-c_N\ell}.
\end{align*}
Choosing $N>\frac{m}{2}$, we obtain
\begin{equation}\label{eq_multiplier_final}
    \sup_{ \substack{0<\rho\le 1\\ y\in \mathbb{R}^{m}} } \biggl\| \int_{\mathbb{R}^m} e^{iY\cdot(y-y')} m_{\ell,\rho}^{(*)}(Y) \, dY \biggr\|_{L^1_{y'}(\mathbb{R}^m)} \le C_N e^{-c_N \ell}, \quad
    *\in \{ \mathrm{rad}, \mathrm{ang}, \mathrm{bas}, j\}.
\end{equation}
On the other hand, the standard spherical projection estimates \cite[Chapter~IV]{SW2} (see also \cite[Theorem~2.9, Section~2.3]{AHan}) yield
\begin{align}\label{eq_sph_pro1}
    \sup_{\theta \in \mathbb{S}^{n-1}} \int_{\mathbb{S}^{n-1}} |\Pi_{\ell}(\theta, \theta')| \, d\theta' \le C (1+\ell)^{\frac{n-2}{2}},
\end{align}
and, for $\ell\ge1$,
\begin{align}\label{eq_sph_pro2}
    \sup_{\theta \in \mathbb{S}^{n-1}} \int_{\mathbb{S}^{n-1}} \bigl| \lambda_\ell^{-1/2} \nabla_{\theta} \Pi_\ell(\theta, \theta') \bigr| \, d\theta' \le  C (1+\ell)^{\frac{n-2}{2}}.
\end{align}
Indeed, \eqref{eq_sph_pro1} follows from Cauchy--Schwarz and the rotational invariance of the spherical projection:
\begin{align*}
    \int_{\mathbb{S}^{n-1}} |\Pi_{\ell}(\theta, \theta')| \, d\theta'
    &\le C \biggl( \int_{\mathbb{S}^{n-1}} |\Pi_\ell(\theta,\theta')|^2 \, d\theta' \biggr)^{\frac{1}{2}}
    = C \, \Pi_\ell(\theta, \theta)^{\frac{1}{2}}
    = C \, d_\ell^{\frac{1}{2}}
    \le C (1+\ell)^{\frac{n-2}{2}}.
\end{align*}
Similarly, for \eqref{eq_sph_pro2},
\begin{align*}
    \int_{\mathbb{S}^{n-1}} \bigl| \lambda_\ell^{-1/2} \nabla_{\theta} \Pi_\ell(\theta, \theta') \bigr| \, d\theta'
    &\le C \lambda_\ell^{-1/2} \biggl( \int_{\mathbb{S}^{n-1}} |\nabla_\theta \Pi_\ell(\theta,\theta')|^2 \, d\theta' \biggr)^{\frac{1}{2}}\\
    &= C \lambda_\ell^{-1/2} \biggl( \sum_{\mu=1}^{d_\ell} |\nabla_\theta \Theta_{\ell,\mu}(\theta)|^2 \biggr)^{\frac{1}{2}}\\
    &\le C \, d_\ell^{\frac{1}{2}}
    \le C (1+\ell)^{\frac{n-2}{2}},
\end{align*}
where rotational invariance and integration over the sphere give
\[
 \sum_{\mu=1}^{d_\ell}|\nabla_\theta\Theta_{\ell,\mu}(\theta)|^2
 =\frac{\lambda_\ell d_\ell}{|\mathbb S^{n-1}|}.
\]
Combining \eqref{eq_sph_pro1}, \eqref{eq_sph_pro2}, and \eqref{eq_multiplier_final}, it follows that
\begin{align*}
    \sup_{\substack{0<\rho\le 1\\ y\in \mathbb{R}^{m}\\ \theta \in \mathbb{S}^{n-1}}} \iint_{\mathbb{S}^{n-1} \times \mathbb{R}^m} \bigl| \nabla_L \mathcal{K}_\rho (\theta,y; \theta', y') \bigr| \, d\theta' \, dy' \le C \sum_{\ell=0}^\infty (1+\ell)^{\frac{n-2}{2}} e^{-c\ell} \le C,
\end{align*}
which completes the proof of Theorem~\ref{lem:Poisson_kernel}.
\end{proof}

Before proving the key estimate \eqref{eq_detailed_poisson_multiplier_goal_copy}, further introduce the notation
\begin{equation*}
\mathfrak A_\rho:=\frac{\rho^{\mathfrak a}}{\mathfrak a},
\qquad
\mathfrak A_2:=\frac{2^{\mathfrak a}}{\mathfrak a},
\qquad
\tau_\rho(Y):= |Y| \mathfrak A_\rho,
\qquad
\tau_2(Y):=|Y|\mathfrak A_2,
\end{equation*}
and define
\begin{equation*}
H_{\ell,\rho}(r)
:=
\frac{\mathcal I_{\nu_\ell}(\mathfrak A_\rho r)}
{\mathcal I_{\nu_\ell}(\mathfrak A_2 r)},
\qquad
\mathcal I_\nu(t):=t^{-\nu}I_\nu(t).
\end{equation*}
For simplicity, the dependence of $\tau_\rho$ and $\tau_2$ on $Y$ will often be suppressed.

All Poisson multiplier estimates reduce to the behavior of a single normalized Bessel ratio. The main difficulty here is that the order $\nu_\ell$ is not fixed, but grows linearly with the spherical harmonic frequency $\ell$. The next proposition records the uniform estimates needed later. Its proof separates the transition region $r\lesssim 1+\ell$ from the large-frequency region $r\gg1+\ell$.

\begin{proposition}\label{lem_transition_ratio_derivatives}
For every integer $N \ge 0$ and every $0\le k\le N$, there exist constants $C_N,A_N,c_N>0$ such that
\begin{align}\label{eq_transition_H_derivatives}
    |H_{\ell,\rho}^{(k)}(r)|
    \le \begin{cases}
        C_N(1+\ell)^{A_N}, & r \le C^*(1+\ell),\\[4pt]
        C_N(1+r)^{A_N}e^{-c_N r}, & r > C^*(1+\ell),
    \end{cases}
\end{align}
where $C^*>1$ is a sufficiently large constant.

Consequently, for every multi-index $|\gamma|\le N$,
\begin{align}\label{eq_transition_Y_derivatives}
\bigl|\partial_Y^\gamma H_{\ell,\rho}(|Y|)\bigr| \le \begin{cases}
    C_N (1+\ell)^{A_N}, & |Y|\le C^* (1+\ell),\\[4pt]
    C_N (1+|Y|)^{A_N}e^{-c_N|Y|}, & |Y|>C^*(1+\ell).
\end{cases}
\end{align}
\end{proposition}

\begin{proof}
Using the identity
\begin{equation*}
    tI_\nu'(t)-\nu I_\nu(t)=tI_{\nu+1}(t),
\end{equation*}
it follows that $\mathcal I_\nu'(t)=t^{-\nu}I_{\nu+1}(t)$. Hence,
\begin{equation*}
    \frac{d}{dt}\log\mathcal I_\nu(t) = S_\nu(t)
    \qquad\text{and}\qquad
    \log H_{\ell,\rho}(r)
    =
    -\int_{\mathfrak A_\rho r}^{\mathfrak A_2 r}S_{\nu_\ell}(t) \, dt.
\end{equation*}
On the other hand, note that $\log H_{\ell,\rho}(r) = \log\mathcal I_{\nu_\ell}(\mathfrak A_\rho r) - \log\mathcal I_{\nu_\ell}(\mathfrak A_2 r)$.
This implies that for $j\ge 1$,
\begin{equation}\label{eq_log_H_derivative_formula}
    \frac{d^j}{dr^j}\log H_{\ell,\rho}(r)
    =
    \mathfrak A_\rho^j S_{\nu_\ell}^{(j-1)}(\mathfrak A_\rho r)
    -
    \mathfrak A_2^j S_{\nu_\ell}^{(j-1)}(\mathfrak A_2 r).
\end{equation}

\textit{Case 1: $r\le C^*(1+\ell)$.} Since $0<\mathfrak A_\rho\le \mathfrak A_2$ and $\nu_\ell\sim 1+\ell$, Lemma~\ref{lem_S_derivative_bound} yields
\begin{equation}\label{eq_log_H}
    \Bigl|
        \frac{d^j}{dr^j}\log H_{\ell,\rho}(r)
    \Bigr|
    \le
    C_j(1+\ell)^{A_j}.
\end{equation}
Moreover, as $t^{-\nu}I_\nu(t)$ is increasing on $(0,\infty)$, $H_{\ell,\rho}(r)$ is positive and uniformly bounded by $1$, i.e., $0<H_{\ell,\rho}(r)\le 1$.
Applying Fa\`a di Bruno's formula to $H_{\ell,\rho}=\exp\bigl(\log H_{\ell,\rho}\bigr)$ gives, for each integer $k\ge 1$,
\begin{equation}\label{eq_faa_di_bruno_H}
    H_{\ell,\rho}^{(k)}(r)
    =
    H_{\ell,\rho}(r)
    \, P_k\Bigl(
        \frac{d}{dr}\log H_{\ell,\rho}(r),
        \dots,
        \frac{d^k}{dr^k}\log H_{\ell,\rho}(r)
    \Bigr),
\end{equation}
where $P_k$ is a polynomial depending only on $k$. Combine \eqref{eq_faa_di_bruno_H} with \eqref{eq_log_H} to get
\begin{equation}\label{eq_H_k}
    |H_{\ell,\rho}^{(k)}(r)|
    \le
    C_N (1+\ell)^{A_N} H_{\ell,\rho}(r)
    \le
    C_N (1+\ell)^{A_N}.
\end{equation}
This proves \eqref{eq_transition_H_derivatives} in the range $r\le C^*(1+\ell)$. 

It remains to pass from one-dimensional derivatives in $r$ to derivatives in $Y$. If $\tau_2 \ge 1$, that is, $|Y|\ge \mathfrak A_2^{-1}$, then the chain rule yields
\begin{align*}
\bigl|\partial_Y^\gamma H_{\ell,\rho}(|Y|)\bigr|
&\le \sum_{k=0}^{|\gamma|} |Y|^{k-|\gamma|} \bigl|H_{\ell,\rho}^{(k)}(|Y|)\bigr|
\le C_N (1+\ell)^{A_N}, \qquad |Y|\le C^* (1+\ell),
\end{align*}
for all $|\gamma|\le N$. Next, if $\tau_2 \le 1$, that is, $|Y|\le \mathfrak A_2^{-1}$, first observe that $H_{\ell,\rho}(|Y|)$ is in fact a smooth function of $|Y|^2$. Indeed, recalling the power series expansion
\begin{equation*}
    I_\nu(t)
    =
    \frac{1}{\Gamma(\nu+1)}
    \Bigl(\frac{t}{2}\Bigr)^{\!\nu}
    \sum_{q=0}^{\infty}
    \frac{1}{q!(\nu+1)_q}
    \Bigl(\frac{t^2}{4}\Bigr)^{\!q},
\end{equation*}
we find
\begin{equation*}
    \mathcal I_\nu(t)
    :=
    t^{-\nu}I_\nu(t)
    =
    \frac{1}{2^\nu\Gamma(\nu+1)}
     E_\nu(t^2), \qquad  E_\nu(s)
    =
    \sum_{q=0}^{\infty}
    \frac{1}{q!(\nu+1)_q}
    \Bigl(\frac{s}{4}\Bigr)^{\!q}.
\end{equation*}
Note that $E_\nu(s) \ge 1$ for all $s>0$. Also,
\begin{equation*}
    H_{\ell,\rho}(r)
    =
    \frac{
        E_{\nu_\ell}(\mathfrak A_\rho^2 r^2)
    }{
        E_{\nu_\ell}(\mathfrak A_2^2 r^2)
    }.
\end{equation*}
Equivalently,
\begin{equation*}
    H_{\ell,\rho}(|Y|)
    =
    h_{\ell,\rho}(|Y|^2), \qquad h_{\ell,\rho}(s) :=  \frac{
        E_{\nu_\ell}(\mathfrak A_\rho^2 s)
    }{
        E_{\nu_\ell}(\mathfrak A_2^2 s)
    }.
\end{equation*}
By \eqref{eq_A_derivatives_small}, for every fixed $N$, $\sup_{0\le s\le 1}\bigl| E_\nu^{(k)}(s) \bigr| \le C_N$ ($0\le k\le N$),
uniformly in $\nu\ge \nu_*> -1$. The quotient rule therefore gives
\begin{equation}\label{eq_h_deri}
    |h_{\ell,\rho}^{(k)}(s)|\le C_k, \quad \forall\, s\le \mathfrak A_2^{-2}.
\end{equation}
Hence, for every multi-index $\gamma$,
\begin{equation*}
    \bigl| \partial_Y^{\gamma} H_{\ell,\rho}(|Y|)\bigr|
    = \bigl|\partial_Y^\gamma
    \bigl[
        h_{\ell,\rho}(|Y|^2)
    \bigr]\bigr|
    \le \sum_{k=0}^{|\gamma|} C_{\gamma, k} \bigl|P_{\gamma,k}(|Y|)\bigr| \, \bigl|h_{\ell,\rho}^{(k)}\bigl(|Y|^2\bigr)\bigr|,
\end{equation*}
where $P_{\gamma,k}$ is a polynomial depending only on $\gamma$ and $k$. Since $|Y|\le C^*(1+\ell)$, and $|Y|\le \mathfrak A_2^{-1}$, \eqref{eq_h_deri} implies
\begin{equation*}
    \bigl|
        \partial_Y^\gamma H_{\ell,\rho}(|Y|)
    \bigr|
    \le
    C_N (1+\ell)^{A_N},\quad \forall\, |\gamma|\le N,
\end{equation*}
as required.

\textit{Case 2: $r>C^*(1+\ell)$.} By a standard Amos-type estimate \cite[Corollary~1]{KB} (see Lemma~\ref{lem_bessel_inputs_high_frequency}), there exist $c_0, C_1>0$ such that
\begin{equation}\label{eq_amos_lower_bound}
    S_\nu(t)=\frac{I_{\nu+1}(t)}{I_\nu(t)}\geq c_0,\quad \text{if}\quad t\ge C_1(1+\nu).
\end{equation}
Also, since $0<\rho\le 1$, $\frac{\mathfrak A_\rho}{\mathfrak A_2} = \frac{\rho^{\mathfrak a}}{2^{\mathfrak a}} \le 2^{-\mathfrak a}$.
Now, if $C^*$ is chosen sufficiently large, then for every
$t\in(2^{- \mathfrak a} \mathfrak A_2 r,\mathfrak A_2 r]$,
\begin{equation*}
    t\ge 2^{- \mathfrak a} \mathfrak A_2 r\ge \mathfrak a^{-1} C^*(1+\ell) \ge C_1(1+\nu_\ell).
\end{equation*}
Using \eqref{eq_amos_lower_bound}, we deduce, for some $c>0$,
\begin{align*}
\log H_{\ell,\rho}(r) &= -\int_{\mathfrak A_\rho r}^{\mathfrak A_2 r}S_{\nu_\ell}(t) \, dt
\le -\int_{2^{- \mathfrak a} \mathfrak A_2 r}^{\mathfrak A_2 r}S_{\nu_\ell}(t) \, dt \\
&\le -c_0(1-2^{- \mathfrak a})\mathfrak A_2 r \le -cr.
\end{align*}
It follows that
\begin{equation}\label{eq_H_large_decay}
    0<H_{\ell,\rho}(r)\le e^{-cr},
    \qquad
    r>C^* (1+\ell).
\end{equation}
Next, we estimate the derivatives of $H_{\ell,\rho}$. Combining \eqref{eq_log_H_derivative_formula} with Lemma~\ref{lem_S_derivative_bound}, 
yields, for $j\ge 1$,
\begin{equation}\label{eq_log_H_derivative_polynomial}
    \Bigl|
        \frac{d^j}{dr^j}\log H_{\ell,\rho}(r)
    \Bigr|
    \le
    C_j(1+r)^{A_j},
    \qquad
    r>C^*(1+\ell).
\end{equation}
Therefore, by \eqref{eq_H_large_decay}, \eqref{eq_log_H_derivative_polynomial}, and \eqref{eq_faa_di_bruno_H}, Fa\`a di Bruno's formula gives
\begin{equation*}
    |H_{\ell,\rho}^{(k)}(r)| \le C_N (1+r)^{A_N}e^{-c_N r},\qquad 0\le k \le N,
\end{equation*}
which proves \eqref{eq_transition_H_derivatives} in the range $r>C^*(1+\ell)$. Estimate \eqref{eq_transition_Y_derivatives} then follows immediately from the chain rule for derivatives of radial functions, since $|Y|>C^*(1+\ell)\ge C^*$. This completes the proof of Proposition~\ref{lem_transition_ratio_derivatives}.
\end{proof}

\begin{remark}\label{re:H}
The proof of Proposition~\ref{lem_transition_ratio_derivatives} uses the growth of the order only to keep the constants uniform in $\ell$. If the order is fixed, the same argument applies to every $\mu>-1$. More precisely, for
\begin{equation*}
    H_{\mu,\rho}(r):=
\frac{\mathcal I_\mu(\mathfrak A_\rho r)}{\mathcal I_\mu(\mathfrak A_2 r)},
\end{equation*}
one has, for every integer $N\ge 0$,
\begin{equation*}
    |H_{\mu,\rho}^{(k)}(r)|
\le
\begin{cases}
C_{N,\mu}, & r\le C^*,\\[4pt]
C_{N,\mu}(1+r)^{A_N}e^{-c_N r}, & r>C^*,
\end{cases}
\end{equation*}
for $0\le k\le N$, uniformly in $0<\rho\leq1$. The corresponding
$Y$-derivative estimates for $H_{\mu,\rho}(|Y|)$ follow in the same way. In
Part~\ref{part2} we apply this fixed-order version with $\mu=\nu$ and with
$\mu=-\nu\in(-1,0)$. 
\end{remark}

We now translate the normalized Bessel-ratio estimates into estimates for the
actual Poisson multipliers. The three multiplier families correspond respectively
to the radial derivative, the angular derivative, and the $y$-derivatives in
$\nabla_L$. The factor $(\rho/2)^{\gamma_\ell}$ is crucial: for $\rho\leq1$ it
converts polynomial growth in $\ell$ into exponential summability in $\ell$.

\begin{proposition}\label{lem_detailed_poisson_multiplier}
For every integer $N \ge 0$, there exist constants $C_N, c_N > 0$ such that
\begin{equation}\label{eq_detailed_poisson_multiplier_goal}
    \sum_{|\gamma|\le N}
    \bigl\|
        \partial_Y^\gamma m_{\ell,\rho}^{(*)}
    \bigr\|_{L^1(\mathbb R^m_Y)}
    \le
    C_N e^{-c_N\ell},
    \quad
    *\in \{\mathrm{rad}, \mathrm{ang}, \mathrm{bas}, j\},
\end{equation}
uniformly for $0<\rho\le 1$ and $\ell\ge 0$.
\end{proposition}

\begin{proof}
It suffices to prove that for every multi-index $|\gamma|\le N$, there exist constants $A_N, C_N, c_N>0$ such that
\begin{equation}\label{eq_pointwise_split_bound_Y}
    \bigl|
        \partial_Y^\gamma m_{\ell,\rho}^{(*)}(Y)
    \bigr|
    \le
    C_N
    (1+\ell+|Y|)^{A_N}
    \bigl[
        e^{-c_N \ell}\mathbf 1_{|Y|\le C^* (1+\ell)}
        +
        e^{-c_N |Y|}\mathbf 1_{|Y|>C^* (1+\ell)}
    \bigr],
\end{equation}
uniformly for $0<\rho\le 1$. Once \eqref{eq_pointwise_split_bound_Y} is established, \eqref{eq_detailed_poisson_multiplier_goal} follows immediately. Indeed, \eqref{eq_pointwise_split_bound_Y} suggests
\begin{align*}
    \int_{|Y|\le C^* (1+\ell)}
    (1+\ell+|Y|)^{A_N}e^{-c_N \ell} \, dY
    &\le
    C_N (1+\ell)^{A_N + m}e^{-c_N \ell}
    \le
    C_N e^{-c_N'\ell},
\end{align*}
and
\begin{align*}
    \int_{|Y|> C^* (1+\ell)}
    (1+\ell+|Y|)^{A_N}e^{-c_N |Y|} \, dY
    &\le
    C_N e^{-c_N''\ell}.
\end{align*}
Therefore, $\bigl\| \partial_Y^\gamma m_{\ell,\rho}^{(*)} \bigr\|_{L^1_Y} \le C_N e^{-c_N \ell}$.
Summing over $|\gamma|\le N$ proves \eqref{eq_detailed_poisson_multiplier_goal}.

It remains to prove \eqref{eq_pointwise_split_bound_Y}. Note that for $\ell\ge 0$,
\begin{equation}\label{eq_M_H}
    M_{\ell,\rho}(Y) = \Bigl( \frac{\rho}{2} \Bigr)^{\!-\mathfrak b} \frac{I_{\nu_\ell}(\tau_\rho)}{I_{\nu_\ell}(\tau_2)} = \Bigl( \frac{\rho}{2} \Bigr)^{\!\gamma_\ell} H_{\ell,\rho}(|Y|).
\end{equation}
Moreover, for $\ell\ge 1$,
\begin{align*}
    \rho^\alpha \partial_\rho M_{\ell, Y}(\rho)
    = \gamma_\ell \rho^{\alpha-1} \Bigl(\frac{\rho}{2}\Bigr)^{\!\gamma_\ell} H_{\ell,\rho}(|Y|) + \rho^{\alpha} \Bigl(\frac{\rho}{2}\Bigr)^{\!\gamma_\ell} \partial_\rho H_{\ell,|Y|}(\rho).
\end{align*}
The identity
\begin{align*}
    \frac{\partial_\rho H_{\ell, |Y|}(\rho)}{H_{\ell, |Y|}(\rho)}
    = \partial_\rho \log H_{\ell, |Y|}(\rho)
    = |Y| \rho^{\mathfrak a-1} \frac{\mathcal{I}'_{\nu_\ell}(\mathfrak A_\rho |Y|)}{\mathcal{I}_{\nu_\ell}(\mathfrak A_\rho |Y|)}
    = |Y| \rho^{\mathfrak a-1} S_{\nu_\ell}(\mathfrak A_\rho |Y|),
\end{align*}
gives
\begin{align}\label{eq_rad_M_H}
    \rho^\alpha \partial_\rho M_{\ell, Y}(\rho)
    = \gamma_\ell 2^{\alpha-1} \Bigl(\frac{\rho}{2}\Bigr)^{\!\gamma_\ell+\alpha-1} H_{\ell, \rho}(|Y|)
    + \mathfrak a 2^{\alpha -1} \Bigl(\frac{\rho}{2}\Bigr)^{\!\gamma_\ell+\alpha - 1} |\mathfrak A_\rho Y| \, S_{\nu_\ell}(\mathfrak A_\rho |Y|) H_{\ell, \rho}(|Y|).
\end{align}
We split the argument into two cases.

\textit{Case 1: $|Y|\le C^* (1+\ell)$.} 
By Proposition~\ref{lem_transition_ratio_derivatives} and \eqref{eq_M_H},
\begin{equation}\label{eq_M_deri}
    \bigl| \partial_Y^\gamma M_{\ell,\rho}(Y) \bigr| \le C_N \Bigl( \frac{\rho}{2} \Bigr)^{\!\gamma_\ell} (1+\ell)^{A_N} \le C_N (1+\ell)^{A_N} e^{-c \ell},\quad \forall\, |\gamma|\le N,
\end{equation}
where the last inequality follows from
\begin{equation*}
     \Bigl(\frac{\rho}{2}\Bigr)^{\!\gamma_\ell} \le e^{- \gamma_\ell \log 2} \le C e^{-c \ell}
\end{equation*}
since $\rho\le1$, $\gamma_0=0$, and $\gamma_\ell\sim1+\ell$ for $\ell\ge1$. Therefore, since $\alpha-1+\gamma_\ell>0$ for $\ell \ge 1$, $m_{0,\rho}^{(\mathrm{ang})}=0$, and $\sqrt{\lambda_\ell}$ grows at most polynomially in $\ell$, we obtain
\begin{equation*}
    \bigl|\partial_Y^\gamma m_{\ell,\rho}^{(\mathrm{ang})}(Y)\bigr|
    \le
    C_N
    (1+\ell+|Y|)^{A_N}
        e^{-c_N \ell}\mathbf 1_{|Y|\le C^* (1+\ell)},\quad |\gamma|\le N.
\end{equation*}
For the base multiplier $m_{\ell,\rho}^{(\mathrm{bas}, j)}(Y) = i Y_j \rho^{\beta} M_{\ell,\rho}(Y)$, with $|Y|\le C^*(1+\ell)$, $\gamma_\ell+\beta\ge0$, and \eqref{eq_M_deri}, Leibniz's rule gives
\begin{equation*}
    \bigl|\partial_Y^\gamma m_{\ell,\rho}^{(\mathrm{bas},j)}(Y)\bigr|
    \le
    C_N
    (1+\ell+|Y|)^{A_N}
        e^{-c_N \ell}\mathbf 1_{|Y|\le C^* (1+\ell)},\quad |\gamma|\le N,\quad 1\le j\le m.
\end{equation*}
Next, for the radial multiplier, that is, for $*=\mathrm{rad}$, by \eqref{eq_rad_M_H},
and $\gamma_\ell +\alpha-1>0$, Proposition~\ref{lem_transition_ratio_derivatives} together with Lemma~\ref{lem_S_derivative_bound} yields
\begin{equation*}
    \bigl| \partial_Y^{\gamma} m_{\ell,\rho}^{(\mathrm{rad})}(Y) \bigr|
    \le C_{\alpha,N} (1+\ell + |Y|)^{A_N} e^{-c_N \ell} \mathbf 1_{|Y|\le C^* (1+\ell)},\quad |\gamma|\le N,\qquad \ell \ge 1.
\end{equation*}
The case $\ell = 0$ is handled similarly. Indeed, since $\gamma_0=0$, $M_{0,\rho}(Y) = H_{0,\rho}(|Y|)$ and hence
\begin{equation*}
    m_{0,\rho}^{(\mathrm{rad})}(Y) = \rho^\alpha \partial_\rho M_{0, Y}(\rho)
    = 2^\beta \Bigl(\frac{\rho}{2}\Bigr)^{\!\beta}|Y| S_{\nu_0}(\mathfrak A_\rho |Y|) H_{0, \rho}(|Y|).
\end{equation*}
Clearly,
\begin{equation*}
    |Y|S_{\nu_0}(\mathfrak A_\rho |Y|)\le C|Y|\le C(1+|Y|).
\end{equation*}
Moreover, by Lemma~\ref{lem_S_derivative_bound} and the chain rule, for $|\gamma|\ge 1$,
\begin{equation*}
    |\partial_Y^{\gamma}(|Y|S_{\nu_0}(\mathfrak A_\rho |Y|))|
    \le C_N \mathfrak A_\rho^{|\gamma|-1} (1+\nu_0+\mathfrak A_\rho |Y|)^{A_N}
    \le C_N (1+|Y|)^{A_N}.
\end{equation*}
The required $Y$-derivative estimate for $m_{0,\rho}^{(\mathrm{rad})}$ follows at once. 

\textit{Case 2: $|Y|>C^* (1+\ell)$.} With \eqref{eq_M_H}, \eqref{eq_rad_M_H} in mind, using Proposition~\ref{lem_transition_ratio_derivatives}, Lemma~\ref{lem_S_derivative_bound}, the same argument as in \textit{Case 1} implies 
\begin{equation*}
    \bigl|
        \partial_Y^\gamma m_{\ell,\rho}^{(*)}(Y)
    \bigr|
    \le
    C_N
    (1+\ell+|Y|)^{A_N}
        e^{-c_N |Y|}\mathbf 1_{|Y|> C^* (1+\ell)},\quad |\gamma|\le N,\qquad * \in \{\mathrm{rad}, \mathrm{ang}, \mathrm{bas}, j\}.
\end{equation*}
for all $\ell \ge 0$. This proves \eqref{eq_pointwise_split_bound_Y}, and hence Proposition~\ref{lem_detailed_poisson_multiplier}.
\end{proof}

\begin{remark}\label{re:multiplier}
The proof of Proposition~\ref{lem_detailed_poisson_multiplier} has two independent
components: the fixed-order Bessel estimates for the multiplier in the Fourier
variable $Y$, and the summation over spherical harmonics. In the one-dimensional
case of Part~\ref{part2} the second component disappears. Therefore, once Proposition~\ref{lem_transition_ratio_derivatives} is replaced by its fixed-order version from the preceding Remark~\ref{re:H}, the same multiplier-to-kernel argument applies
verbatim. The only new feature in the odd half-line branch is the explicit factor
$\rho^{1-2\alpha}$, whose intrinsic radial derivative contributes the loss $\rho^{-\alpha}$.
\end{remark}

We next turn from the Poisson operator to the Green operator. The Poisson term is
controlled by differentiating normalized boundary multipliers. For the Green term,
the main issue is row integrability of the one-dimensional Green kernel and of its
intrinsic derivatives.

The next estimate is purely mode-wise and will later be
combined with integration by parts in the Fourier variable and summation in $\ell$.

\begin{proposition}\label{prop:Green:Shur}
For $\ell\ge 1$ and $0<\rho\le 2$, we have
\begin{equation}\label{eq_basic_R_schur_improved}
    \int_0^2 G_{\ell,Y}(\rho,\sigma) \, \sigma^{n-1} \, d\sigma \le \begin{cases}
        C \rho^{1-\alpha} (1+\ell)^{-2}, & |Y|\le C^*(1+\ell),\\[4pt]
        C \rho^{1-\alpha} (1+\ell+\tau_\rho)^{-2}, & |Y|> C^*(1+\ell).
    \end{cases}
\end{equation}
In particular,
\begin{equation}\label{eq_basic_R_schur_ell=0}
    \int_0^2 G_{0,Y}(\rho,\sigma) \, \sigma^{n-1} \, d\sigma \le \begin{cases}
        C, & |Y|\le C^*,\\[4pt]
        C (1+\tau_\rho)^{-2}, & |Y|> C^*.
    \end{cases}
\end{equation}
Moreover, for $\ell \ge 0$,
\begin{align}\label{eq_basic_R_schur}
\sup_{0<\rho<2} \rho^{2\beta} \int_0^2 G_{\ell,Y}(\rho,\sigma) \, \sigma^{n-1} \, d\sigma \le \begin{cases}
        C (1+\ell)^{-2}, & |Y|\le C^*(1+\ell),\\[4pt]
        C |Y|^{-2}, & |Y|> C^*(1+\ell).
    \end{cases}
\end{align}
\end{proposition}

\begin{proof}
\textit{Case 1: $|Y|\le C^*(1+\ell)$.} 
The one-dimensional Green kernel takes the form (cf.~\eqref{eq_kernel_Green}):
\begin{align}\label{eq_kernel_Green_copy}
G_{\ell,Y}(\rho,\sigma) = \mathfrak a^{-1}
\begin{cases}
\phi_{\ell}(\rho,Y)\,\psi_{\ell}(\sigma,Y), &0<\rho<\sigma,\\[4pt]
\phi_{\ell}(\sigma,Y)\,\psi_{\ell}(\rho,Y), &\sigma < \rho<2,
\end{cases}
\end{align}
where $\phi_{\ell,Y}(\rho) = \rho^{-\mathfrak b}I_{\nu_\ell}(\tau_\rho)$ and
\begin{equation*}
\psi_{\ell,Y}(\rho) = \rho^{-\mathfrak b} \bigl[ K_{\nu_\ell}(\tau_\rho) - \eta_{\ell,Y} I_{\nu_\ell}(\tau_\rho)\bigr],\qquad \eta_{\ell,Y} := \frac{K_{\nu_\ell}(\tau_2)}{I_{\nu_\ell}(\tau_2)}.
\end{equation*}

Suppose first $\ell\ge 1$. With the notation
\begin{equation*}
\gamma_\ell:= -\mathfrak b+\mathfrak a\nu_\ell = -\mathfrak b+\sqrt{\mathfrak b^2+\lambda_\ell},\quad \delta_\ell := \mathfrak b+\mathfrak a\nu_\ell = \mathfrak b+\sqrt{\mathfrak b^2+\lambda_\ell},
\end{equation*}
both $\gamma_\ell$ and $\delta_\ell$ are of size $1+\ell$. Since $t^{-\nu}I_\nu(t)$ is increasing on $(0,\infty)$, for
$0<\sigma\leq\rho$,
\begin{equation}\label{eq_phi_ratio}
    \frac{\phi_{\ell,Y}(\sigma)}{\phi_{\ell,Y}(\rho)}
    =
    \Bigl(\frac{\sigma}{\rho}\Bigr)^{\!-\mathfrak b}
    \frac{I_{\nu_\ell}(\tau_\sigma)}{I_{\nu_\ell}(\tau_\rho)}
    \le
    \Bigl(\frac{\sigma}{\rho}\Bigr)^{\!\gamma_\ell}.
\end{equation}
On the other hand, the upper estimate
\begin{equation}\label{eq_psi_comparable_K}
    \psi_{\ell,Y}(\rho)
    \le
    \rho^{-\mathfrak b}K_{\nu_\ell}(\tau_\rho)
\end{equation}
holds for $0<\rho<2$.  Moreover, there exists $0<c<1$, independent of
$\ell$ and $Y$, such that
\begin{equation}\label{eq_psi_lower_K}
    c\rho^{-\mathfrak b}K_{\nu_\ell}(\tau_\rho)
    \le \psi_{\ell,Y}(\rho),
    \qquad 0<\rho\le1,\quad \ell\ge0.
\end{equation}
Indeed, as $t^\nu K_\nu(t)$ is decreasing on $(0,\infty)$, the definition of $\eta_{\ell,Y}$ gives
\begin{equation*}
\frac{\eta_{\ell,Y}I_{\nu_\ell}(\tau_\rho)}{K_{\nu_\ell}(\tau_\rho)} =
\frac{I_{\nu_\ell}(\tau_\rho)}{I_{\nu_\ell}(\tau_2)}
\frac{K_{\nu_\ell}(\tau_2)}{K_{\nu_\ell}(\tau_\rho)}\le
\Bigl(\frac{\rho}{2}\Bigr)^{\!2\mathfrak a\nu_\ell}\le
2^{-2\mathfrak a\nu_0} < 1,
\qquad 0<\rho\le1.
\end{equation*}
This establishes \eqref{eq_psi_comparable_K}--\eqref{eq_psi_lower_K}.
Consequently, if $0<\rho\le1$ and $\rho\le \sigma \le 2$, then
\begin{equation}\label{eq_psi_ratio}
    \frac{\psi_{\ell,Y}(\sigma)}{\psi_{\ell,Y}(\rho)}
    \leq
    C
    \Bigl(\frac{\rho}{\sigma}\Bigr)^{\!\delta_\ell}.
\end{equation}
By Lemma~\ref{lem_bessel_inputs_high_frequency} and $\nu_\ell \ge \nu_0 >0$,
\begin{equation}\label{eq_IK_product_bound}
I_{\nu_\ell} (t) K_{\nu_\ell} (t) \le \frac{C}{1+\nu_\ell}, \qquad \ell \ge 0, \qquad t>0.
\end{equation}
Combining \eqref{eq_psi_comparable_K} and \eqref{eq_IK_product_bound},
\begin{equation}\label{eq_phi_psi_product_bound}
    \phi_{\ell,Y}(\rho)\psi_{\ell,Y}(\rho)
    \le
    C\rho^{-2\mathfrak b}(1+\ell)^{-1},
    \qquad
    0<\rho\leq1,
    \qquad
    \ell\ge 1.
\end{equation}
Fix $0<\rho\le 1$. Split
\begin{equation*}
\begin{aligned}
    \int_0^2 G_{\ell,Y}(\rho,\sigma) \, \sigma^{n-1} \, d\sigma
    &=
    \mathfrak a^{-1} \psi_{\ell,Y}(\rho)
    \int_0^\rho
        \phi_{\ell,Y}(\sigma)\,\sigma^{n-1}\,d\sigma          \\
    &\quad+
    \mathfrak a^{-1} \phi_{\ell,Y}(\rho)
    \int_\rho^2
        \psi_{\ell,Y}(\sigma)\,\sigma^{n-1}\,d\sigma          \\
    &=: I_1+I_2.
\end{aligned}
\end{equation*}
For $I_1$, \eqref{eq_phi_ratio} together with \eqref{eq_phi_psi_product_bound} implies
\begin{align}\label{eq_I1_bound_schur}
    I_1 &\le \mathfrak a^{-1} \phi_{\ell,Y}(\rho)\psi_{\ell,Y}(\rho) \int_0^\rho \Bigl(\frac{\sigma}{\rho}\Bigr)^{\!\gamma_\ell} \sigma^{n-1} \, d\sigma\\ \nonumber
    &= \frac{\mathfrak a^{-1}}{n+\gamma_\ell} \, \rho^n \phi_{\ell,Y}(\rho)\psi_{\ell,Y}(\rho)\\ \nonumber
    &\le C (1+\ell)^{-2}\rho^{n-2\mathfrak b}\\ \nonumber
    &\le C (1+\ell)^{-2}\rho^{1-\alpha}.
\end{align}
For $I_2$, using \eqref{eq_psi_ratio} and \eqref{eq_phi_psi_product_bound},
\begin{align*}
    I_2 &\le C\phi_{\ell,Y}(\rho)\psi_{\ell,Y}(\rho)
    \int_\rho^2 \Bigl(\frac{\rho}{\sigma}\Bigr)^{\!\delta_\ell}\sigma^{n-1} \, d\sigma \\
    &\le C (1+\ell)^{-1} \rho^{\delta_\ell - 2 \mathfrak b} \int_\rho^2 \sigma^{n-\delta_\ell - 1} \, d\sigma \\
    &\le C \begin{cases}
        (1+\ell)^{-2} \rho^{n-2\mathfrak b}, & \delta_\ell >n,\\
        (1+\ell)^{-1} \rho^{n-2\mathfrak b} \log(2/\rho), & \delta_\ell = n,\\
        (1+\ell)^{-1} |n-\delta_\ell|^{-1}
        \rho^{\delta_\ell-2\mathfrak b}, &\delta_\ell<n.
    \end{cases}
\end{align*}
For $\delta_\ell \ge n$, it is clear that $I_2 \le C (1+\ell)^{-2} \rho^{1-\alpha}$ since $n-2\mathfrak b=2-2\alpha$ and $\rho\le 1$ (use $\log(2/\rho)\le C_\varepsilon\rho^{-\varepsilon}$ for $0<\varepsilon<1-\alpha$). If $\delta_\ell <n$, first note that for $n\ge 2$, 
\begin{equation}\label{eq_delta_lower}
    \delta_\ell \ge \delta_1\geq n+\alpha-1.
\end{equation}
It follows that
\begin{equation*}
 \rho^{\delta_\ell-2\mathfrak b}
 =\rho^{n-2\mathfrak b}\rho^{\delta_\ell-n}
 \le \rho^{n+\alpha-1-2\mathfrak b}
 =\rho^{1-\alpha}.
\end{equation*}
The coefficients in $\ell$ are harmless as there are only finitely many
$\ell\ge 1$ such that $\delta_\ell < n$. In summary,
\begin{equation*}
    I_2 \le C (1+\ell)^{-2} \rho^{1-\alpha},\qquad 0<\rho \le 1.
\end{equation*}
Combining this with \eqref{eq_I1_bound_schur} proves \eqref{eq_basic_R_schur_improved} and \eqref{eq_basic_R_schur} with $\ell \ge 1$.

It remains to treat the case $\ell=0$. Observe that \eqref{eq_psi_comparable_K} and \eqref{eq_IK_product_bound} remain valid. From the explicit formula \eqref{eq_kernel_Green_copy} and the monotonicity of $t^{-\nu}I_{\nu}(t)$ and $t^\nu K_\nu(t)$,
\begin{equation*}
    G_{0,Y}(\rho,\sigma) \le C \begin{cases}
        \rho^{\gamma_0} \sigma^{-\delta_0}, & 0<\rho \le \sigma <2,\\[4pt]
        \rho^{-\delta_0} \sigma^{\gamma_0}, & 0<\sigma <\rho<2.
    \end{cases}
\end{equation*}
Therefore, since $\gamma_0 = 0$ and $\delta_0 = n-2(1-\alpha) <n$, for all $0<\rho \le 1$,
\begin{align*}
    \int_0^2 G_{0,Y}(\rho,\sigma)\,\sigma^{n-1}\,d\sigma
    &\le C \rho^{-\delta_0} \int_0^\rho \sigma^{n-1}\,d\sigma\\
    &\quad+ C \int_\rho^2 \sigma^{n-\delta_0 -1}\,d\sigma \\
    &\le C.
\end{align*}
This proves \eqref{eq_basic_R_schur_ell=0} and \eqref{eq_basic_R_schur} with $\ell=0$.

\textit{Case 2: $|Y|> C^*(1+\ell)$.} Set for simplicity $k_{\ell,Y}(\rho):=\rho^{-\mathfrak b}K_{\nu_\ell}(\tau_\rho)$. 
By \eqref{eq_psi_comparable_K}, it is enough to estimate $G_{\ell,Y}(\rho,\sigma)$ with $\psi_{\ell,Y}$ replaced by $k_{\ell,Y}$. Fix $0<\rho<2$. From \eqref{eq_kernel_Green_copy},
\begin{align}\label{eq_weighted_split}
    \rho^{2\beta}
    \int_0^2
        G_{\ell,Y}(\rho,\sigma)
    \, \sigma^{n-1} \, d\sigma                                      \nonumber
    \le
    C\rho^{2\beta}
    \phi_{\ell,Y}(\rho)k_{\ell,Y}(\rho)
    \biggl[\int_0^\rho \frac{\phi_{\ell,Y}(\sigma)}{\phi_{\ell,Y}(\rho)} \, \sigma^{n-1} \, d\sigma + \int_\rho^2 \frac{k_{\ell,Y}(\sigma)}{k_{\ell,Y}(\rho)} \, \sigma^{n-1} \, d\sigma \biggr].
\end{align}
By Lemma~\ref{lem_bessel_inputs_high_frequency}, for $\nu_\ell\ge \nu_0>0$,
\begin{equation}\label{eq_phi_k_product}
    \phi_{\ell,Y}(\rho)k_{\ell,Y}(\rho)
    =
    \rho^{-2\mathfrak b}
    I_{\nu_\ell}(\tau_\rho)K_{\nu_\ell}(\tau_\rho)
    \le
    C\rho^{-2\mathfrak b}(1 + \ell + \tau_\rho)^{-1}.
\end{equation}
Put $r=\sigma/\rho\in(0,1]$. Then
\begin{equation*}
\frac{\phi_{\ell,Y}(\sigma)}{\phi_{\ell,Y}(\rho)} = r^{\gamma_\ell} \frac{\mathcal I_{\nu_\ell}(\tau_\rho r^{\mathfrak a})}{\mathcal I_{\nu_\ell}(\tau_\rho)}.
\end{equation*}
Assume first that $\tau_\rho \le  C_0(1+\ell)$ with $C_0>0$ to be determined later. Since $\mathcal I_\nu$ is increasing,
\begin{align*}
    \int_0^\rho
        \frac{\phi_{\ell,Y}(\sigma)}
             {\phi_{\ell,Y}(\rho)}
    \, \sigma^{n-1} \, d\sigma
    &\le C\rho^n \int_0^1 r^{n-1+\gamma_\ell}\,dr
    \le C\frac{\rho^n}{1+\ell}
    \le C\frac{\rho^n}{1+ \ell + \tau_\rho}.
\end{align*}
Next assume that $\tau_\rho>C_0(1+\ell)$. Choose $C_0>1$ sufficiently large so that $2^{- \mathfrak a} C_0 (1+\ell) \ge C(1+\ell)$, where $C>0$ is the constant from \eqref{eq_S_lower_large}. Lemma~\ref{lem_bessel_inputs_high_frequency} \eqref{eq_S_lower_large} then gives for $0<r<1/2$ (thus $r^{\mathfrak a}\le 2^{- \mathfrak a}$),
\begin{align*}
\frac{\mathcal I_{\nu_\ell}(\tau_\rho r^{\mathfrak a})}{\mathcal I_{\nu_\ell}(\tau_\rho)}
&=
\exp\Bigl\{-\int_{\tau_\rho r^{\mathfrak a}}^{\tau_\rho}\frac{I_{\nu_\ell+1}(t)}{I_{\nu_\ell}(t)}\,dt\Bigr\} \\
&\le
\exp\Bigl\{-\int_{\tau_\rho 2^{-\mathfrak a}}^{\tau_\rho}\frac{I_{\nu_\ell+1}(t)}{I_{\nu_\ell}(t)}\,dt\Bigr\}
\le C e^{-c \tau_\rho},
\end{align*}
and for $1/2\le r \le 1$ (thus $\tau_\rho r^{\mathfrak a} \ge 2^{- \mathfrak a} C_0 (1+\ell) \ge C(1+\ell)$),
\begin{equation*}
\frac{\mathcal I_{\nu_\ell}(\tau_\rho r^{\mathfrak a})}{\mathcal I_{\nu_\ell}(\tau_\rho)}
=\exp\Bigl\{-\int_{\tau_\rho r^{\mathfrak a}}^{\tau_\rho}\frac{I_{\nu_\ell+1}(t)}{I_{\nu_\ell}(t)}\,dt\Bigr\}
\le C e^{-c \tau_\rho(1-r^{\mathfrak a})}
\le C e^{-c \tau_\rho(1-r)}.
\end{equation*}
Hence
\begin{align*}
   \int_0^\rho \frac{\phi_{\ell,Y}(\sigma)}{\phi_{\ell,Y}(\rho)}\,\sigma^{n-1}\,d\sigma
   &\le
    C\rho^n \int_0^{1/2} r^{n-1+\gamma_\ell} e^{-c\tau_\rho}\,dr\\
    &\quad+ C\rho^n \int_{1/2}^1 r^{n-1+\gamma_\ell}e^{-c\tau_\rho(1-r)}\,dr\\
    &\le C\frac{\rho^n}{1+\ell+\tau_\rho}.
\end{align*}
Thus, in all cases,
\begin{equation}\label{eq_I_ratio_integral_final}
    \int_0^\rho
        \frac{\phi_{\ell,Y}(\sigma)}
             {\phi_{\ell,Y}(\rho)}
    \, \sigma^{n-1} \, d\sigma
    \le
    C\frac{\rho^n}{1+\ell+\tau_\rho}.
\end{equation}

For the $K$-branch, note that now $r=\sigma/\rho \ge 1$. Then
\begin{equation}\label{eq_K_ratio}
    \frac{k_{\ell,Y}(\sigma)}
         {k_{\ell,Y}(\rho)}
    =
    r^{-\delta_\ell}
    \frac{\mathcal K_{\nu_\ell}(\tau_\rho r^{\mathfrak a})}
         {\mathcal K_{\nu_\ell}(\tau_\rho)},
    \qquad
    \mathcal K_\nu(t):=t^\nu K_\nu(t),
\end{equation}
where
\begin{equation*}
    \delta_\ell=\mathfrak b+\mathfrak a\nu_\ell =
    \mathfrak b+\sqrt{\mathfrak b^2+\lambda_\ell}
    \sim 1+\ell.
\end{equation*}
Moreover, for $\ell\ge 1$, \eqref{eq_delta_lower} implies $\delta_\ell\ge n+\alpha-1$.

\textit{Case 2a: $\delta_\ell > n$.} If $\tau_\rho\le C_0(1+\ell)$, then $\mathcal K_\nu$ is decreasing and hence
\begin{align}\label{eq_formula3}
\int_\rho^2 \frac{k_{\ell,Y}(\sigma)}{k_{\ell,Y}(\rho)}\,\sigma^{n-1}\,d\sigma
&\le \rho^n
    \int_1^{2/\rho} r^{n-1-\delta_\ell}\,dr
    \le C\frac{\rho^n}{1+\ell}
    \le C\frac{\rho^n}{1+ \ell + \tau_\rho}.
\end{align}
If instead $\tau_\rho>C_0 (1+\ell)$, where $C_0$ is again chosen large enough so that Lemma~\ref{lem_bessel_inputs_high_frequency} applies, then \eqref{eq_K_normalized_exp_decay} yields
\begin{equation*}
    \frac{\mathcal K_{\nu_\ell}(\tau_\rho r^{\mathfrak a})}
         {\mathcal K_{\nu_\ell}(\tau_\rho)}
    \le
    C e^{-c\tau_\rho(r^{\mathfrak a}-1)}
\end{equation*}
for $r\ge 1$.  Since $\tau_\rho>C_0(1+\ell)$, the substitution
$s=r^{\mathfrak a}$ gives
\begin{equation*}
 \int_1^\infty e^{-c\tau_\rho(r^{\mathfrak a}-1)}\,dr
 =\frac1{\mathfrak a}\int_1^\infty
 e^{-c\tau_\rho(s-1)}s^{1/\mathfrak a-1}\,ds
 \le \frac{C}{\tau_\rho}.
\end{equation*}
Therefore,
\begin{equation}\label{eq_K_ratio_large_ell}
    \int_\rho^2\frac{k_{\ell,Y}(\sigma)}{k_{\ell,Y}(\rho)}\,\sigma^{n-1}\,d\sigma
    \le C\rho^n  \int_1^\infty e^{-c\tau_\rho(r^{\mathfrak a}-1)}\,dr
    \le C\frac{\rho^n}{1+\ell+\tau_\rho}.
\end{equation}
Combining this with \eqref{eq_phi_k_product},
\begin{align*}
    \rho^{2\beta}
    \int_0^2
        G_{\ell,Y}(\rho,\sigma)
    \, \sigma^{n-1} \, d\sigma
    \le \frac{C \rho^{n-2 \mathfrak b+2\beta} }{(1+\ell + \tau_\rho)^2}
    = \frac{C \rho^{2 \mathfrak a} }{(1+\ell + \tau_\rho)^2}
    \le C |Y|^{-2},
\end{align*}
as $\rho^{2 \mathfrak a} = \mathfrak a^2 |Y|^{-2} \tau_\rho^2$. This proves \eqref{eq_basic_R_schur} with $\delta_\ell >n$.

\textit{Case 2b: $\delta_\ell \le n$.} On the region $\sigma<\rho$, the change of variables $u=\tau_\sigma$ gives
\begin{align}\label{eq_finite_left_scaling}
\rho^{2\beta}
    k_{\ell,Y}(\rho)
    \int_0^\rho
        \phi_{\ell,Y}(\sigma) \, \sigma^{n-1} \, d\sigma
    =
    C |Y|^{-2}
    \tau_\rho^{\frac{2\beta-\mathfrak b}{\mathfrak a}}
    K_{\nu_\ell}(\tau_\rho)
    \int_0^{\tau_\rho}
        u^{\frac{n-\mathfrak b}{\mathfrak a}-1}
        I_{\nu_\ell}(u) \, du.
\end{align}
Since only finitely many values of $\ell$ are involved here, the standard small- and large-argument asymptotic formulas for the modified Bessel functions suffice, and there is no need to keep track of the $\ell$-dependence. We show that the function of $\tau_\rho$ on the right-hand side of \eqref{eq_finite_left_scaling} is uniformly bounded on $(0,\infty)$. By Lemma~\ref{lem:Bessel:integral}~(ii), this function is precisely $II(q,s,\nu_\ell,\tau_\rho)$ with $q=\frac{2\beta-\mathfrak b}{\mathfrak a}$, $s=\frac{n-\mathfrak b}{\mathfrak a}$, and $\nu_\ell \ge \nu_0>0$. In particular,
\begin{equation*}
    s+\nu_\ell=\frac{n-\mathfrak b}{\mathfrak a}+\nu_\ell>0.
\end{equation*}
Moreover,
\begin{align}\label{check}
    q+s+\nu-|\nu| = q+s = \frac{2\beta - 2\mathfrak b+n}{\mathfrak a} = 2 \le 2,
\end{align}
which implies that $II(q,s,\nu_\ell,\tau_\rho)$ is uniformly bounded for all $\tau_\rho >0$. Hence \eqref{eq_finite_left_scaling} is $O(|Y|^{-2})$, as required.

On the other hand,
consider the region $\sigma \ge \rho$. After the same change of variables,
\begin{align*}
\rho^{2\beta}
    \phi_{\ell,Y}(\rho)
    \int_\rho^2
        k_{\ell,Y}(\sigma)\,\sigma^{n-1}\,d\sigma
    \le
    C |Y|^{-2}
    \tau_\rho^{\frac{2\beta-\mathfrak b}{\mathfrak a}}
    I_{\nu_\ell}(\tau_\rho)
    \int_{\tau_\rho}^\infty
        u^{\frac{n-\mathfrak b}{\mathfrak a}-1}
        K_{\nu_\ell}(u) \, du
    = C|Y|^{-2}I(q,s,\nu_\ell,0,\tau_\rho),
\end{align*}
with the same $q,s$ as above. If $\delta_\ell<n$, then $s>\nu_\ell=|\nu_\ell|$. Moreover, \eqref{check} and
\begin{align*}
q+\nu_\ell = \frac{2\beta-\mathfrak b + \sqrt{\mathfrak b^2 + \lambda_\ell}}{\mathfrak a} \ge \frac{2\beta}{\mathfrak a}\ge 0,
\end{align*}
verify all hypotheses of Lemma~\ref{lem:Bessel:integral}~(i), and hence
$I(q,s,\nu_\ell,0,t)$ is uniformly bounded.

It remains to consider the endpoint $\delta_\ell=n$, for which $s=\nu_\ell$ and the strict hypothesis in Lemma~\ref{lem:Bessel:integral}~(i) is unavailable. If $0<t\le1$, the small-argument Bessel estimates give
\begin{equation*}
 I(q,s,\nu_\ell,0,t)
 \le C t^{q+\nu_\ell}\log(2/t)
 = C t^2\log(2/t)\le C,
\end{equation*}
because $q+\nu_\ell=q+s=2$. If $t\ge1$, the usual large-argument estimates give
\begin{equation*}
 I(q,s,\nu_\ell,0,t)\le C t^{q+s-2}=C.
\end{equation*}
This closes the proof of \eqref{eq_basic_R_schur} for $\delta_\ell\le n$, including the endpoint $\delta_\ell=n$.

To complete the argument, it remains to derive a pointwise estimate for the $K$-branch when $\delta_\ell \le n$. Clearly, \eqref{eq_K_ratio_large_ell} remains valid when $\tau_\rho > C_0(1+\ell)$. For the case $\tau_\rho \le C_0(1+\ell)$, however, \eqref{eq_delta_lower} only gives
\begin{align*}
    \int_\rho^2 \frac{k_{\ell,Y}(\sigma)}{k_{\ell,Y}(\rho)} \, \sigma^{n-1} \, d\sigma
    &\le C \rho^n \int_1^{2/\rho} r^{n-\delta_\ell-1}\,dr\\
    &\le C \begin{cases}
        \rho^n \log(2/\rho), & \delta_\ell = n,\\
        |n-\delta_\ell|^{-1} \rho^{\delta_\ell}, & \delta_\ell < n.
    \end{cases}
\end{align*}
For $\delta_\ell=n$, it is plain that the integral is bounded by $C(1+\ell+\tau_\rho)^{-1} \rho^{n+\alpha-1}$, where one can bound the logarithmic term by $C_\alpha \rho^{-(1-\alpha)}$. As for $\delta_\ell<n$, it satisfies upper bound
\begin{align}\label{equation2}
    \begin{cases}
        (1+\ell+\tau_\rho)^{-1} \rho^{n+\alpha-1}, & \ell \ge 1,\\
        (1+\tau_\rho)^{-1} \rho^{2\mathfrak b}, & \ell=0.
    \end{cases}
\end{align}
Combining \eqref{eq_phi_k_product}, \eqref{eq_I_ratio_integral_final}, \eqref{eq_K_ratio_large_ell}, and \eqref{equation2}, we obtain, for all $0<\rho \le 1$,
\begin{align*}
    \int_0^2
        G_{\ell,Y}(\rho,\sigma)
    \, \sigma^{n-1} \, d\sigma
    &\le C \rho^{-2 \mathfrak b} (1+\ell+\tau_\rho)^{-1} \begin{cases}
        (1+\ell+\tau_\rho)^{-1} \rho^{n+\alpha-1}, & \ell \ge 1,\\
        (1+\tau_\rho)^{-1} \rho^{2\mathfrak b}, & \ell=0.
    \end{cases} \\
    &\le C \begin{cases}
        (1+\ell+\tau_\rho)^{-2} \rho^{1-\alpha}, & \ell \ge 1,\\
        (1+\tau_\rho)^{-2}, & \ell=0.
    \end{cases} 
\end{align*}
It remains only to justify the stated range $1\le\rho<2$.  Put
\begin{equation*}
 w_{\ell,Y}(\rho):=\int_0^2G_{\ell,Y}(\rho,\sigma)
                         \sigma^{n-1}\,d\sigma.
\end{equation*}
Then $\mathcal L_{\ell,Y}w_{\ell,Y}=1$, $w_{\ell,Y}(2)=0$, and the
positivity of the Green kernel gives $w_{\ell,Y}\geq0$. The estimates just
proved give the required boundary bounds at $\rho=1$.
The maximum principle on $[1,2]$ now applies with the following barriers:
\begin{equation*}
 C(1+\ell)^{-2}\quad(\ell\ge1,\ |Y|\le C^*(1+\ell)),
 \qquad C(4-\rho^2)\quad(\ell=0,\ |Y|\le C^*),
\end{equation*}
and
\begin{equation*}
 C|Y|^{-2}\quad(|Y|>C^*(1+\ell)).
\end{equation*}

Indeed, on $[1,2]$ the angular potential is bounded below by
$c\lambda_\ell$ when $\ell\ge1$, the last barrier is controlled by the
potential $|Y|^2\rho^{2\beta}$, and a direct calculation gives
\begin{equation*}
 \mathcal L_{0,Y}(4-\rho^2)
 =2(n+2\alpha)\rho^{2\alpha}
   +|Y|^2\rho^{2\beta}(4-\rho^2)
 \ge c>0.
\end{equation*}
After increasing $C$, each displayed barrier is therefore a supersolution
of $\mathcal L_{\ell,Y}b=1$ and dominates $w_{\ell,Y}$ at $\rho=1$;
at $\rho=2$ it is non-negative while $w_{\ell,Y}(2)=0$. The comparison
principle consequently gives the desired bounds throughout $[1,2]$. Since
$\rho^{1-\alpha}\simeq1$ and $\tau_\rho\simeq|Y|$ there, these barriers
give all three asserted estimates for $1\le\rho<2$.  This completes the
proof of \eqref{eq_basic_R_schur_improved},
\eqref{eq_basic_R_schur_ell=0}, and Proposition~\ref{prop:Green:Shur}.

\end{proof}

Define operators
\begin{equation*}
    D_r:=\rho^\alpha\partial_\rho,
    \qquad
    D_\theta:=\rho^{\alpha-1}\sqrt{\lambda_\ell},
    \qquad
    D_{y,j}:=\rho^\beta Y_j.
\end{equation*}
We also have the following first order derivative estimates for the Green kernel.

\begin{proposition}\label{lem_basic_D_R_schur_explicit}
For $\ell \ge 0$,
\begin{align}\label{eq_basic_D_R_schur}
    \sup_{0<\rho\leq1} \int_0^2 |D_* G_{\ell,Y}(\rho,\sigma)|\,\sigma^{n-1}\,d\sigma \le \begin{cases}
        C, & |Y|\le C^*(1+\ell),\\[4pt]
        C|Y|, & |Y|>C^*(1+\ell),
    \end{cases}
\end{align}
where $* \in \{r, \theta, y,j\}$.
\end{proposition}

\begin{proof}
Suppose first $\ell\ge 1$. We claim that for $0<\rho\leq1$,
\begin{equation}\label{eq_log_derivative_phi_psi}
    |\rho^\alpha\phi_{\ell,Y}'(\rho)|
    \leq
    C(1+\ell+\tau_\rho) \rho^{\alpha-1}\phi_{\ell,Y}(\rho),
\end{equation}
and
\begin{equation}\label{eq_log_derivative_psi}
    |\rho^\alpha\psi_{\ell,Y}'(\rho)|
    \leq
    C(1+\ell+\tau_\rho) \rho^{\alpha-1}\psi_{\ell,Y}(\rho).
\end{equation}
To verify \eqref{eq_log_derivative_phi_psi}, use the identity $tI_\nu'(t)-\nu I_\nu(t)=tI_{\nu+1}(t)$,
which gives
\begin{align*}
\rho^\alpha\frac{\phi_{\ell,Y}'(\rho)}{\phi_{\ell,Y}(\rho)}
&=\rho^{\alpha-1}\Bigl[-\mathfrak b+\mathfrak a \tau_\rho \frac{I_{\nu_\ell}'(\tau_\rho)}{I_{\nu_\ell}(\tau_\rho)}\Bigr] =\rho^{\alpha-1}\Bigl[-\mathfrak b+\mathfrak a\nu_\ell+\mathfrak a \tau_\rho \frac{I_{\nu_\ell+1}(\tau_\rho)}{I_{\nu_\ell}(\tau_\rho)}\Bigr].
\end{align*}
Since $-\mathfrak b+\mathfrak a\nu_\ell =\gamma_\ell \sim 1+\ell$, and $0<\frac{I_{\nu+1}(t)}{I_\nu(t)}< 1$,
\begin{equation*}
    \Bigl|\rho^\alpha\frac{\phi_{\ell,Y}'(\rho)}{\phi_{\ell,Y}(\rho)}\Bigr|
    \le
    C\rho^{\alpha-1}\bigl(1+\ell + \tau_\rho\bigr).
\end{equation*}
This proves \eqref{eq_log_derivative_phi_psi}.

To prove \eqref{eq_log_derivative_psi}, write $F_{\ell,Y}(t):=K_{\nu_\ell}(t)-\eta_{\ell,Y}I_{\nu_\ell}(t)$ so that $\psi_{\ell,Y}(\rho)=\rho^{-\mathfrak b}F_{\ell,Y}(\tau_\rho)$.
For $0<\rho\le1$, \eqref{eq_psi_comparable_K} and
\eqref{eq_psi_lower_K} imply that $K_{\nu_\ell}(\tau_\rho)$ and
$F_{\ell,Y}(\tau_\rho)$ are comparable, i.e.,
\begin{equation}\label{eq_F_comparable_K}
    cK_{\nu_\ell}(\tau_\rho)
    \le
    F_{\ell,Y}(\tau_\rho)
    \leq
    K_{\nu_\ell}(\tau_\rho),
\end{equation}
with $0<c<1$ independent of $\ell,Y,\rho$ in this range. Moreover, by Lemma~\ref{lem_S_derivative_bound}, for all $t>0$ and $\nu \ge \nu_0>0$,
\begin{equation}\label{eq_I_log_derivative_standard}
    \bigl|t\frac{I_\nu'(t)}{I_\nu(t)}\bigr|
    =
    \bigl| \nu + t S_{\nu}(t) \bigr|
    \le
    C(1+\nu+t).
\end{equation}
We also claim that
\begin{equation}\label{eq_K_log_derivative_standard}
    \Bigl|t\frac{K_\nu'(t)}{K_\nu(t)}\Bigr|
    \leq
    C(1+\nu+t),
    \quad
    t>0,\quad \nu\ge \nu_0>0.
\end{equation}
Indeed, since $(t^\nu K_\nu(t))'= -t^\nu K_{\nu-1}(t)$,
and $K_{\nu}(t)$ is increasing as a function of $\nu$ on $(0,\infty)$, for each $t>0$ and $\nu\ge1$,
\begin{equation*}
\Bigl| t\frac{K_\nu'(t)}{K_\nu(t)} \Bigr| \le \nu + t \frac{K_{\nu-1}(t)}{K_\nu(t)} \le 1+\nu+t.
\end{equation*}
If instead $\nu_0\le\nu<1$, then $1-\nu \le 1-\nu_0$, and therefore
\begin{align*}
    t\frac{K_{\nu-1}(t)}{K_{\nu}(t)}
    &=t \frac{K_{1-\nu}(t)}{K_{\nu}(t)}
    = t\frac{K_{1-\nu}(t)}{K_{1-\nu_0}(t)} \frac{K_{1-\nu_0}(t)}{K_{\nu_0}(t)} \frac{K_{\nu_0}(t)}{K_{\nu}(t)} \\
    &\le t\frac{K_{1-\nu_0}(t)}{K_{\nu_0}(t)}
    \le C_{\nu_0} \begin{cases}
        t^{2\nu_0}, & t\le 1,\\
        t, & t\ge 1.
    \end{cases}
\end{align*}
This proves \eqref{eq_K_log_derivative_standard}.

Combining \eqref{eq_I_log_derivative_standard}, \eqref{eq_K_log_derivative_standard}, and \eqref{eq_F_comparable_K},
\begin{equation*}
    |\tau_\rho F_{\ell,Y}'(\tau_\rho)|
    \le
    C(1+\nu_\ell+\tau_\rho) |F_{\ell,Y}(\tau_\rho)|.
\end{equation*}
Consequently,
\begin{equation}\label{eq:normal2}
\begin{aligned}
    \rho^\alpha
    \Bigl|
        \frac{\psi_{\ell,Y}'(\rho)}{\psi_{\ell,Y}(\rho)}
    \Bigr|
    &\le
    \rho^{\alpha-1}
    \bigl[
        |\mathfrak b|
        +
        \mathfrak a
        \bigl|
            \tau_\rho\frac{F_{\ell,Y}'(\tau_\rho)}{F_{\ell,Y}(\tau_\rho)}
        \bigr|
    \bigr]                                                \\
    &\le
    C\rho^{\alpha-1}(1+\nu_\ell+\tau_\rho)                    \\
    &\le C(1+\ell+\tau_\rho) \rho^{\alpha-1}.
\end{aligned}
\end{equation}
This establishes \eqref{eq_log_derivative_psi}.

We now estimate \eqref{eq_basic_D_R_schur}. By \eqref{eq_kernel_Green}, for
$0<\rho\le 1$,
\begin{align}\label{D:bdd:G}
    \int_0^2
        |D_r G_{\ell,Y}(\rho,\sigma)|
    \, \sigma^{n-1} \, d\sigma
    &\le
    C (1+\ell+\tau_\rho) \rho^{\alpha-1}
    \int_0^2
        G_{\ell,Y}(\rho,\sigma)
    \, \sigma^{n-1} \, d\sigma.
\end{align}
Indeed, when $\sigma\ge \rho$, the $\rho$-derivative falls on $\phi_{\ell,Y}(\rho)$, so \eqref{eq_log_derivative_phi_psi} applies; when $\sigma\le \rho$, the $\rho$-derivative falls on $\psi_{\ell,Y}(\rho)$, so \eqref{eq_log_derivative_psi} applies. It follows by Proposition~\ref{prop:Green:Shur} that if $|Y|\le C^*(1+\ell)$ (thus $\tau_\rho = \mathfrak a^{-1} |Y| \rho^{\mathfrak a} \le C(1+\ell)$), then for $\ell\ge 1$,
\begin{align*}
    \int_0^2 &|D_* G_{\ell,Y}(\rho,\sigma)| \, \sigma^{n-1} \, d\sigma \\
    &\le  \begin{cases}
        C (1+\ell)\rho^{\alpha-1} (1+\ell)^{-2}\rho^{1-\alpha}, & *=r,\\
    C\sqrt{\lambda_\ell} \rho^{\alpha-1} (1+\ell)^{-2}\rho^{1-\alpha}, & *=\theta,\\
    C|Y_j|\rho^\beta (1+\ell)^{-2}\rho^{1-\alpha}, & *=y,j.
    \end{cases}\\
    &\le C.
\end{align*}
This proves \eqref{eq_basic_D_R_schur} for $\ell\ge 1$ and $|Y|\le C^*(1+\ell)$.

Suppose next $|Y|>C^*(1+\ell)$. Again by Proposition~\ref{prop:Green:Shur}, for $\ell \ge 1$,
\begin{align*}
    \int_0^2 &|D_* G_{\ell,Y}(\rho,\sigma)| \, \sigma^{n-1} \, d\sigma \\
    &\le  \begin{cases}
        C (1+\ell+\tau_\rho) \rho^{\alpha-1} \frac{\rho^{1-\alpha}}{(1+\ell+\tau_\rho)^2}, & *=r,\\
    C (1+\ell) \rho^{\alpha-1} \frac{\rho^{1-\alpha}}{(1+\ell+\tau_\rho)^2}, & *=\theta,\\
    C|Y|\rho^\beta
    \frac{\rho^{1-\alpha}}{(1+\ell+\tau_\rho)^2}, & *=y,j.
    \end{cases}\\
    &\le C\le C|Y|,
\end{align*}
where in the base component we used
$|Y|\rho^{\beta+1-\alpha}=\mathfrak a\tau_\rho$.

To complete, assume $\ell=0$. The angular component, namely the $D_\theta$ part, is absent. The base component $D_{y,j}$ can be treated exactly as above, using \eqref{eq_basic_R_schur_ell=0}. For the radial part $D_r$, note that
\begin{align*}
\Bigl|\rho^\alpha\frac{\phi_{0,Y}'(\rho)}{\phi_{0,Y}(\rho)}\Bigr|
=\rho^{\alpha-1}\Bigl|-\mathfrak b+\mathfrak a\nu_0 + \mathfrak a \tau_\rho \frac{I_{\nu_0+1}(\tau_\rho)}{I_{\nu_0}(\tau_\rho)}\Bigr|= \mathfrak a \rho^{\alpha-1} \tau_\rho \frac{I_{\nu_0+1}(\tau_\rho)}{I_{\nu_0}(\tau_\rho)}.
\end{align*}
This, together with \eqref{eq:normal2}, gives
\begin{align*}
    \int_0^2
        |D_r G_{0,Y}(\rho,\sigma)|
    \, \sigma^{n-1} \, d\sigma &\le C (1+\tau_\rho) \rho^{\alpha-1} \int_0^\rho \psi_{0,Y}(\rho) \phi_{0,Y}(\sigma) \, \sigma^{n-1} \, d\sigma\\
    &\quad + C \tau_\rho \rho^{\alpha-1} \frac{I_{\nu_0+1}(\tau_\rho)}{I_{\nu_0}(\tau_\rho)}  \int_\rho^2 \phi_{0,Y}(\rho) \psi_{0,Y}(\sigma) \, \sigma^{n-1} \, d\sigma\\
    &\le C (1+\tau_\rho) \rho^{\alpha-1-\mathfrak b} K_{\nu_0}(\tau_\rho) \int_0^\rho \sigma^{n-\mathfrak b-1} I_{\nu_0}(\tau_\sigma) \, d\sigma\\
    &\quad + C \tau_\rho \rho^{\alpha-1-\mathfrak b} I_{\nu_0+1}(\tau_\rho)  \int_\rho^2 \sigma^{n-\mathfrak b-1} K_{\nu_0}(\tau_\sigma) \, d\sigma\\
    &\le  C \rho^{1-\alpha} (1+\tau_\rho) \tau_\rho^{\frac{2\alpha-2-\mathfrak b}{\mathfrak a}} K_{\nu_0}(\tau_\rho) \int_0^{\tau_\rho} u^{\frac{n-\mathfrak b}{\mathfrak a}-1} I_{\nu_0}(u) du\\
    &\quad + C \rho^{1-\alpha} \tau_\rho^{\frac{2\alpha-2-\mathfrak b}{\mathfrak a}+1} I_{\nu_0+1}(\tau_\rho) \int_{\tau_\rho}^\infty u^{\frac{n-\mathfrak b}{\mathfrak a}-1} K_{\nu_0}(u) du.
\end{align*}
By Lemma~\ref{lem:Bessel:integral}, the above is bounded by 
\begin{equation*}
    C \rho^{1-\alpha} II \left(\frac{2\alpha-2-\mathfrak b}{\mathfrak a},\frac{n-\mathfrak b}{\mathfrak a},\nu_0,\tau_\rho \right) +  C \rho^{1-\alpha} II \left(\frac{2\alpha-2-\mathfrak b}{\mathfrak a}+1,\frac{n-\mathfrak b}{\mathfrak a},\nu_0,\tau_\rho \right)\le C,
\end{equation*}
and
\begin{equation*}
   C \rho^{1-\alpha} I\left(\frac{2\alpha-2-\mathfrak b}{\mathfrak a}+1, \frac{n-\mathfrak b}{\mathfrak a},\nu_0, 1, \tau_\rho\right) \le C,
\end{equation*}
as a consequence of $\frac{2\alpha-2-\mathfrak b}{\mathfrak a}+ \frac{n-\mathfrak b}{\mathfrak a}=0$ and $\rho\le 1$. This completes the proof of Proposition~\ref{lem_basic_D_R_schur_explicit}.

\end{proof}

The following estimate is the technical heart of the Green-kernel analysis. Its
proof avoids differentiating the explicit Bessel formula repeatedly. Instead, it
uses the resolvent identity. The key observation is that $\mathcal L_{\ell,Y}$ depends on
$Y$ only through the potential $|Y|^2\rho^{2\beta}$. Hence every $Y$-derivative
of the Green operator can be written as a finite sum of products of resolvents and
multiplication operators. This turns $Y$-differentiation into row-norm estimates
for the already controlled Green kernel.

\begin{proposition}\label{lem:second_to_last}
For every integer $N>m+2$, we have
\begin{equation}
    \sup_{0<\rho\le 1} \int_{\mathbb{R}^m} \int_0^2 \sum_{|\gamma|=N} \bigl| \partial_Y^\gamma D_* G_{\ell,Y}(\rho,\sigma) \bigr| \, \sigma^{n-1} \, d\sigma \, dY \le C_{N} (1+\ell)^{-N+m+2}
\end{equation}
for all $\ell \ge 0$, where
$D_*\in\{D_r,D_\theta,D_{y,j}:1\leq j\leq m\}$.
\end{proposition}

\begin{proof}
Set
\begin{equation*}
    \mathcal H:=L^2((0,2),\rho^{n-1}\,d\rho).
\end{equation*}
For fixed $\ell$, let $\mathfrak q_{\ell,0}$ be the closed Friedrichs form
on $\mathcal H$ determined by
\begin{equation*}
 \mathfrak q_{\ell,0}[f]
 =\int_0^2
   \left(\rho^{2\alpha}|f'(\rho)|^2
   +\lambda_\ell\rho^{2\alpha-2}|f(\rho)|^2\right)
   \rho^{n-1}\,d\rho,
\end{equation*}
with the Dirichlet condition at $\rho=2$ and the Friedrichs condition at
$\rho=0$. Writing $W$ for multiplication by $\rho^{2\beta}$, put
\begin{equation*}
    \mathfrak q_{\ell,Y}
    :=\mathfrak q_{\ell,0}+|Y|^2\langle W\,\cdot,\cdot\rangle_{\mathcal H}.
\end{equation*}
Because $W$ is bounded on $(0,2)$, all these forms have the same form
domain. Their associated operators are strictly positive, by the Dirichlet
condition at $\rho=2$, and are given by
\begin{equation*}
    \mathcal{L}_{\ell,Y} = - \rho^{2\alpha} \partial_\rho^2 - (n-1+2\alpha) \rho^{2\alpha-1} \partial_\rho + \lambda_\ell \rho^{2\alpha -2} + |Y|^2 \rho^{2\beta}.
\end{equation*}
The crucial point is that $\mathcal{L}_{\ell,Y}$ depends on $Y$ only through
the bounded potential $|Y|^2W$. Bounded-perturbation theory therefore gives
\begin{equation}\label{eq:common-resolvent-domain}
    \mathcal D(\mathcal L_{\ell,Y_1})
    =\mathcal D(\mathcal L_{\ell,Y_2}),
    \qquad Y_1,Y_2\in\mathbb R^m,
\end{equation}
and norm-resolvent continuity in $Y$. Moreover,
\begin{equation}
    \partial_{Y_j} \mathcal{L}_{\ell,Y} = 2Y_j W,
    \qquad
    \partial_{Y_i} \partial_{Y_j} \mathcal{L}_{\ell,Y} = 2 \delta_{ij} W,
\end{equation}
where $\delta_{ij}$ is the Kronecker symbol; all higher $Y$-derivatives
vanish.

Define $R_{\ell,Y}:= \mathcal{L}_{\ell,Y}^{-1}$. Then, for any $u\in \mathcal{H}$ and any $t>0$,
\begin{align}\label{eq:operator:calculus}
    R_{\ell,Y+t e_j}
    \bigl( \mathcal{L}_{\ell, Y+t e_j} - \mathcal{L}_{\ell, Y} \bigr) R_{\ell,Y}u
    &= R_{\ell,Y+t e_j} \mathcal{L}_{\ell, Y+t e_j} R_{\ell,Y}u - R_{\ell,Y+t e_j} \mathcal{L}_{\ell, Y} R_{\ell,Y}u \nonumber\\
    &= R_{\ell,Y+t e_j} \mathcal{L}_{\ell, Y+t e_j} R_{\ell,Y}u - R_{\ell,Y+t e_j}u \nonumber\\
    &= R_{\ell,Y}u - R_{\ell,Y+t e_j}u,
\end{align}
where the last identity uses the fact that
\begin{equation*}
    R_{\ell,Y}u \in \mathcal{D}(\mathcal{L}_{\ell,Y})
    = \mathcal{D}(\mathcal{L}_{\ell,Y+te_j}).
\end{equation*}
In addition,
\begin{equation*}
    \mathcal{L}_{\ell, Y+t e_j} - \mathcal{L}_{\ell, Y} = (2t Y_j + t^2) \rho^{2\beta}.
\end{equation*}
Dividing both sides of \eqref{eq:operator:calculus} by $t$ and using
norm-resolvent continuity as $t\to0$ gives
\begin{equation*}
    \partial_{Y_j}R_{\ell,Y} = - R_{\ell,Y} \bigl( 2Y_j \rho^{2\beta} \bigr) R_{\ell,Y}.
\end{equation*}
Iterating this identity, for every multi-index $\gamma$ with $|\gamma|=N$, the operator $\partial_{Y}^\gamma R_{\ell,Y}$ is a finite sum of terms of the form
\begin{equation}
    R_{\ell,Y} B_1 R_{\ell,Y} B_2 \dots B_k R_{\ell,Y},
\end{equation}
where each $B_i$ is one of the multiplication operators
\begin{equation*}
    2Y_j \rho^{2\beta} \quad \text{or}\quad 2 \delta_{ij}\rho^{2\beta}.
\end{equation*}
Suppose that $q_1$ of the $B_i$'s are of the first type and $q_2$ of the second type. Then
\begin{equation*}
    q_1+2q_2 = N\quad \text{and}\quad k=q_1+q_2.
\end{equation*}
The explicit Bessel formulas show that $G_{\ell,Y}(\rho,\sigma)$ is smooth
in the Cartesian parameter $Y$, including at $Y=0$, where its dependence is
through a convergent power series in $|Y|^2$. Standard
parameter-dependent Sturm--Liouville regularity therefore identifies the
kernel of $\partial_Y^\gamma R_{\ell,Y}$ with
$\partial_Y^\gamma G_{\ell,Y}$ and permits the intrinsic derivatives
$D_r$ and $D_\theta$ to commute with $\partial_Y^\gamma$.

For an integral operator $T$ with kernel $K_T$ acting on $\mathcal H$, define its row norm by
\begin{equation*}
    \| T \|_{\mathrm{row}}:= \sup_{0<\rho <2} \int_0^2
    |K_T(\rho,\sigma)| \, \sigma^{n-1} \, d\sigma.
\end{equation*}
This norm is submultiplicative for integral operators. Multiplication
operators are kept separate: if $M_b$ denotes multiplication by $b$, then
\begin{equation}\label{eq:multiplier-resolvent-row}
    \|M_bT\|_{\mathrm{row}}
    \leq \|b\|_\infty\|T\|_{\mathrm{row}}.
\end{equation}
More precisely, after $D_*$ is applied to the leftmost resolvent in a term
$R_{\ell,Y}B_1R_{\ell,Y}\cdots B_kR_{\ell,Y}$, submultiplicativity gives
\begin{align}\label{eq:row-product-fixed-rho}
 &\int_0^2
 \left|K_{D_*R_{\ell,Y}B_1R_{\ell,Y}\cdots B_kR_{\ell,Y}}
       (\rho,\sigma)\right|\sigma^{n-1}\,d\sigma \notag\\
 &\qquad\leq
 \left(\int_0^2|D_*G_{\ell,Y}(\rho,s)|s^{n-1}\,ds\right)
 \prod_{i=1}^k\|B_iR_{\ell,Y}\|_{\mathrm{row}}.
\end{align}
Therefore, by Proposition~\ref{prop:Green:Shur}, for $|Y|\le C^*(1+\ell)$,
\begin{align*}
    \prod_{i=1}^k\|B_iR_{\ell,Y}\|_{\mathrm{row}}
    &\le C |Y|^{q_1} (1+\ell)^{-2k}\\
    &\le C (1+\ell)^{q_1-2q_1-2q_2}
    = C (1+\ell)^{-N}.
\end{align*}
Consequently, for $|Y|\le C^*(1+\ell)$,
\begin{align}\label{eq_G_low}
    \sup_{0<\rho\le 1}\int_0^2 &\sum_{|\gamma|=N} \bigl| \partial_Y^\gamma D_* G_{\ell,Y}(\rho,\sigma) \bigr| \, \sigma^{n-1} \, d\sigma \nonumber\\
    &\le \sum_{|\gamma|=N} C_\gamma \sup_{0<\rho\le 1} \int_0^2 \bigl| D_* G_{\ell,Y}(\rho,\sigma) \bigr| \, \sigma^{n-1} \, d\sigma \; (1+\ell)^{-N} \nonumber\\
    &\le C_N (1+\ell)^{-N},\quad *\in \{ r, \theta\}.
\end{align}
For $D_{y,j}=Y_j\rho^\beta$, Leibniz's rule gives, componentwise for
every $|\gamma|=N$,
\begin{equation}\label{eq:base-Y-Leibniz}
    \partial_Y^\gamma(D_{y,j}R_{\ell,Y})
    =D_{y,j}\partial_Y^\gamma R_{\ell,Y}
     +\gamma_j\rho^\beta
       \partial_Y^{\gamma-e_j}R_{\ell,Y},
\end{equation}
where the second term is absent when $\gamma_j=0$.
Accordingly,
\begin{align}\label{eq_low_worst}
    \sup_{0<\rho\le 1}\int_0^2 &\sum_{|\gamma|=N} \bigl| \partial_Y^\gamma D_{y,j} G_{\ell,Y}(\rho,\sigma) \bigr| \, \sigma^{n-1} \, d\sigma \nonumber\\
    &\le (1+\ell)^{-N}\sum_{|\gamma|=N} C_\gamma \sup_{0<\rho\le 1} \int_0^2 \bigl| D_{y,j} G_{\ell,Y}(\rho,\sigma) \bigr| \, \sigma^{n-1} \, d\sigma  \nonumber\\
    &\quad+  (1+\ell)^{-N+1}\sum_{|\gamma|=N} C_\gamma \sup_{0<\rho\le 1} \int_0^2 \bigl| \rho^\beta G_{\ell,Y}(\rho,\sigma) \bigr| \, \sigma^{n-1} \, d\sigma \nonumber\\
    &\le C_N (1+\ell)^{-N+1}.
\end{align}

Now consider the region $|Y|\ge C^*(1+\ell)$. Proposition~\ref{prop:Green:Shur} gives
\begin{align*}
    \bigl\| B_1 R_{\ell,Y} B_2 \dots B_k R_{\ell,Y} \bigr\|_{\mathrm{row}}
    &\le \|B_1 R_{\ell,Y}\|_{\mathrm{row}} \dots \|B_k R_{\ell,Y}\|_{\mathrm{row}}
    \le C|Y|^{q_1-2k}.
\end{align*}
Hence, for $|Y|\ge C^*(1+\ell)$,
\begin{align}\label{eq:G-high}
    \sup_{0<\rho\le 1}\int_0^2 &\sum_{|\gamma|=N} \bigl| \partial_Y^\gamma D_* G_{\ell,Y}(\rho,\sigma) \bigr| \, \sigma^{n-1} \, d\sigma \nonumber\\
    &\le \sum_{|\gamma|=N} C_\gamma \sup_{0<\rho\le 1} \int_0^2 \bigl| D_* G_{\ell,Y}(\rho,\sigma) \bigr| \, \sigma^{n-1} \, d\sigma \; |Y|^{q_1-2k} \nonumber\\
    &\le C_N |Y|^{1-N},\quad * \in \{ r, \theta\}.
\end{align}
The case $D_* = D_{y,j}$ is analogous, and the corresponding estimate is bounded by $C_N |Y|^{2-N}$.

Finally, integrate \eqref{eq_G_low} and \eqref{eq:G-high} in the $Y$-variable:
\begin{align*}
\int_{\mathbb{R}^m} \int_0^2 \sum_{|\gamma|=N} \bigl| \partial_Y^\gamma D_* G_{\ell,Y}(\rho,\sigma) \bigr| \, \sigma^{n-1} \, d\sigma \, dY
&= \int_{|Y|\le C^*(1+\ell)} + \int_{|Y|>C^*(1+\ell)}\\
&\le C(1+\ell)^{-N+1+m} + C (1+\ell)^{-N+2+m}\\
&\le C (1+\ell)^{-N+m+2},
\end{align*}
for all $0<\rho\le 1$, as required.
\end{proof}

\begin{corollary}\label{cor_weighted_Y_moment_Green}
For every integer $N>m+3$, we have
\begin{equation}\label{eq_weighted_Y_moment_DG}
    \sup_{0<\rho\le 1}
    \int_{\mathbb R^m}
    \int_0^2
    |Y| \sum_{|\gamma|=N}\bigl|\partial_Y^\gamma D_*G_{\ell,Y}(\rho,\sigma)\bigr|\,\sigma^{n-1}\,d\sigma\,dY\le C_N(1+\ell)^{-N+m+3}
\end{equation}
for all $\ell \ge 0$, where
$D_*\in\{D_r,D_\theta,D_{y,j}:1\leq j\leq m\}$.
\end{corollary}

\begin{proof}
On the low-frequency region $|Y|\le C^*(1+\ell)$, Proposition~\ref{lem:second_to_last}, in particular \eqref{eq_low_worst}, gives, for $|\gamma|=N$,
\begin{equation*}
    \sup_{0<\rho\leq1}
    \int_0^2
    \bigl|
        \partial_Y^\gamma D_*G_{\ell,Y}(\rho,\sigma)
    \bigr|
    \,\sigma^{n-1}\,d\sigma
    \leq
    C_N(1+\ell)^{-N+1},
\end{equation*}
where the exponent $-N+1$ is the worst case, corresponding to the base derivative $D_{y,j}$. Multiplying by $|Y|\le C^*(1+\ell)$ and integrating over the ball $B_{\mathbb{R}^m}(0, C^*(1+\ell))$ gives
\begin{align*}
    \int_{|Y|\le C^*(1+\ell)} |Y| \sup_{0<\rho\leq1} \int_0^2 \sum_{|\gamma|=N} \bigl|\partial_Y^\gamma D_*G_{\ell,Y}(\rho,\sigma)\bigr|\,\sigma^{n-1}\,d\sigma\,dY
    \le C_N(1+\ell)^{-N+m+2}.
\end{align*}
On the high-frequency region $|Y|>C^*(1+\ell)$, Proposition~\ref{lem:second_to_last}, in particular \eqref{eq:G-high}, implies
\begin{equation*}
    \sup_{0<\rho\le 1}
    \int_0^2
    \sum_{|\gamma|=N}
    \bigl|
        \partial_Y^\gamma D_*G_{\ell,Y}(\rho,\sigma)
    \bigr|
    \,\sigma^{n-1}\,d\sigma
    \le
    C_N |Y|^{2-N},
\end{equation*}
again taking the worst case $D_*=D_{y,j}$. Therefore,
\begin{align*}
    \int_{|Y|>C^*(1+\ell)} |Y| \sup_{0<\rho\leq1} \int_0^2 \sum_{|\gamma|=N}\bigl|\partial_Y^\gamma D_*G_{\ell,Y}(\rho,\sigma)\bigr|\,\sigma^{n-1}\,d\sigma\,dY
    &\le C_N  \int_{|Y|>C^*(1+\ell)} |Y|^{3-N}\,dY\\
    &\le C_N(1+\ell)^{-N+m+3},
\end{align*}
provided $N>m+3$. This completes the proof of Corollary~\ref{cor_weighted_Y_moment_Green}.
\end{proof}

The preceding estimates are mode-wise. We now return to physical space by summing the spherical harmonics and inverting the Fourier transform in $y$. The sets $E_1$ and $E_2$ are separated in the $y$-variable. This separation allows us to integrate by parts in $Y$ and use the high-order estimates from Proposition~\ref{lem:second_to_last}. The resulting kernel bounds are the precise estimates needed to control the localized Green term in the proof of the anchored reverse H\"older inequality.

\begin{theorem}\label{lem:Green_kernel}
Let $\mathcal{G}(Z,Z')$ be the Green kernel for $L$ on $\mathcal{C}$ satisfying the finite energy condition at $\rho=0$ and the Dirichlet boundary condition at $\rho=2$, where $Z=(\rho,\theta,y)$ and $Z' = (\sigma,\theta',y')$. Let
\begin{equation*}
    E_1 = \{|y|\le 1\},\quad E_2 = \bigl\{2^{\mathfrak a}\le |y| \le 3^{\mathfrak a} \bigr\}.
\end{equation*}
Then
\begin{equation}\label{1}
    \sup_{ \substack{0<\rho \le 1\\ \theta \in \mathbb{S}^{n-1}\\ y\in E_1}}   \int_0^2 \int_{\mathbb{S}^{n-1}} \int_{E_2} \bigl| \nabla_L \mathcal{G}(Z,Z') \bigr| \, \sigma^{n-1} \, dy' \, d\theta' \, d\sigma \le C.
\end{equation}
Moreover, for each $1\le k \le m$,
\begin{equation}\label{2}
    \sup_{ \substack{0<\rho \le 1\\ \theta \in \mathbb{S}^{n-1}\\ y\in E_1}}   \int_0^2 \int_{\mathbb{S}^{n-1}} \int_{E_2} \bigl| \nabla_L \partial_{y_k'} \mathcal{G}(Z,Z') \bigr| \, \sigma^{n-1} \, dy' \, d\theta' \, d\sigma \le C.
\end{equation}
\end{theorem}

\begin{proof}
We first prove \eqref{1}. By \eqref{eq_Green's_kernel}, it suffices to estimate
\begin{align*}
    I_1&:=\int_0^2 \int_{\mathbb{S}^{n-1}} \int_{E_2} \biggl|\sum_{\ell=0}^\infty \Pi_{\ell}(\theta, \theta') \int_{\mathbb{R}^m} e^{i(y-y') \cdot Y} D_* G_{\ell, Y}(\rho, \sigma) \, dY\biggr| \, \sigma^{n-1} \, dy' \, d\theta' \, d\sigma,\\
    &\qquad *\in\{r, y,j\},
\end{align*}
and
\begin{align*}
    I_2&:=\int_0^2 \int_{\mathbb{S}^{n-1}} \int_{E_2} \biggl|\sum_{\ell=1}^\infty \lambda_\ell^{-1/2}\nabla_\theta\Pi_{\ell}(\theta, \theta') \int_{\mathbb{R}^m} e^{i(y-y') \cdot Y} D_\theta G_{\ell, Y}(\rho, \sigma) \, dY\biggr| \, \sigma^{n-1} \, dy' \, d\theta' \, d\sigma.
\end{align*}
Since $y'\in E_2$ and $y\in E_1$, we have $|y-y'|\ge c$. Observe that for any integer $N\ge 1$,
\begin{equation*}
    e^{i(y-y')\cdot Y} = \Bigl(\frac{(y-y')}{i|y-y'|^2} \cdot \nabla_Y \Bigr)^{\!N} e^{i(y-y')\cdot Y}.
\end{equation*}
Choose $N$ sufficiently large. Integrating by parts $N$ times and applying Proposition~\ref{lem:second_to_last}, we obtain
\begin{align*}
    I_1 &\le  \int_{\mathbb{S}^{n-1}} \int_{E_2} \sum_{\ell=0}^\infty |\Pi_{\ell}(\theta, \theta')| \, |y-y'|^{-N}\int_{\mathbb{R}^m} \int_0^2 \sum_{|\gamma| =N} \bigl|\partial_Y^{\gamma} D_* G_{\ell, Y}(\rho, \sigma)\bigr| \, \sigma^{n-1} \, d\sigma \, dY \, dy' \, d\theta' \\
    &\le C \sum_{\ell=0}^\infty (1+\ell)^{\frac{n-2}{2}} (1+\ell)^{-N+m+2}\le C,
\end{align*}
provided $N>m+2+\frac{n}{2}$. Similarly,
\begin{align*}
    I_2 &\le  \int_{\mathbb{S}^{n-1}} \int_{E_2} \sum_{\ell=1}^\infty \lambda_\ell^{-1/2}|\nabla_\theta\Pi_{\ell}(\theta, \theta')| \, |y-y'|^{-N} \int_{\mathbb{R}^m} \int_0^2 \sum_{|\gamma|=N}\bigl|\partial_Y^\gamma D_\theta G_{\ell, Y}(\rho, \sigma)\bigr| \, \sigma^{n-1} \, d\sigma \, dY \, dy' \, d\theta' \\
    &\le C \sum_{\ell=1}^\infty (1+\ell)^{\frac{n-2}{2}-N+m+2}\le C.
\end{align*}
This proves \eqref{1}.

It remains to prove \eqref{2}. Fix $1\le k \le m$. Since
\begin{equation*}
    \partial_{y_k'}e^{i(y-y')\cdot Y}
    =
    -iY_k e^{i(y-y')\cdot Y},
\end{equation*}
it is enough to estimate oscillatory integrals of the form
\begin{equation*}
    \int_{\mathbb R^m}
        e^{i(y-y')\cdot Y}
        Y_k D_*G_{\ell,Y}(\rho,\sigma) \, dY,
\end{equation*}
where $D_*\in\{D_r,D_\theta,D_{y,j}:1\le j\le m\}$. Again, integrate by parts $N$ times. Since
$|y-y'|\ge c$, it follows that for all $0<\rho \le 1$,
\begin{align*}
    \biggl|\int_{\mathbb R^m} e^{i(y-y')\cdot Y} Y_k D_*G_{\ell,Y}(\rho,\sigma) \, dY \biggr|
    &\le C_N |y-y'|^{-N} \int_{\mathbb R^m} \sum_{|\gamma|=N} \bigl| \partial_Y^\gamma \bigl[ Y_k D_*G_{\ell,Y}(\rho,\sigma)\bigr]\bigr| \, dY\\
    &\le C_N \sum_{|\gamma|=N} \int_{\mathbb R^m}|Y|\,\bigl|\partial_Y^\gamma D_*G_{\ell,Y}(\rho,\sigma)\bigr| \, dY\\
    &\quad+ C_N\sum_{|\gamma|=N-1} \int_{\mathbb R^m}\bigl|\partial_Y^\gamma D_*G_{\ell,Y}(\rho,\sigma)\bigr| \, dY.
\end{align*}
After integrating in $\sigma$ and $y'$, Proposition~\ref{lem:second_to_last}
(with $N-1$ in place of $N$) and
Corollary~\ref{cor_weighted_Y_moment_Green} give
\begin{align*}
    \sup_{0<\rho\leq1}
    \int_0^2
    \int_{E_2}
    \biggl|
    \int_{\mathbb R^m}
        e^{i(y-y')\cdot Y}
        Y_k D_*G_{\ell,Y}(\rho,\sigma) \, dY
    \biggr|
    \,\sigma^{n-1}\,dy'\,d\sigma \le C_N(1+\ell)^{-N+m+3}.
\end{align*}
Using the spherical projection estimates once more, the series in $\ell$ is summable provided
\begin{equation*}
    N>m+3+\frac{n}{2}.
\end{equation*}
Choosing such an $N$ proves \eqref{2}, and hence Theorem~\ref{lem:Green_kernel}.
\end{proof}

We are now in a position to complete the proof of Theorem~\ref{thm:RH}.

\begin{proof}[Proof of Theorem~\ref{thm:RH}]
Recall from the reduction of Theorem~\ref{thm:RH} that
\begin{equation*}
    \omega(\rho,\theta,y)=\chi(y)v(\rho,\theta,y),
\end{equation*}
where $\chi\in C_c^\infty(B_m(0,3^{\mathfrak a}))$ and $\chi=1$ on $B_m(0,2^{\mathfrak a})$. Since $v$ is $L$-harmonic in $Q_4$, we have on the cylinder $\mathcal{C} = (0,2)_\rho \times \mathbb S^{n-1}_\theta \times \mathbb R^m_y$ that $L\omega=f:=[L,\chi]v$. More explicitly,
\begin{equation*}
    f(\sigma,\theta',y')
    =
    \sigma^{2\beta}v(\sigma,\theta',y')\Delta_{y'}\chi(y')
    -
    2\sigma^{2\beta}\nabla_{y'}\chi(y')\cdot
    \nabla_{y'}v(\sigma,\theta',y').
\end{equation*}
Clearly, $\operatorname{supp}(f) \subset E_2$. By the Friedrichs representation, we decompose (see \eqref{DP} and \eqref{BVP}) $\omega=\mathcal T h+\mathcal S f$, where $h(\theta,y)=\omega(2,\theta,y)$. Since $\omega=v$ on $Q_2$, it is enough to prove
\begin{equation*}
    \sup_{Q_1}|\nabla_L\omega|
    \le
    C\|v\|_{L^\infty(Q_3)}.
\end{equation*}
By Theorem~\ref{lem:Poisson_kernel},
\begin{equation*}
    \sup_{Q_1}|\nabla_L\mathcal T h|
    \leq
    C\|h\|_{L^\infty(\mathbb S^{n-1}\times\mathbb R^m)}
    \leq
    C\|v\|_{L^\infty(Q_3)}.
\end{equation*}
It remains to estimate $\nabla_L\mathcal S f$. Write $f=f_1+f_2$, where
\begin{equation*}
    f_1(\sigma,\theta',y')
    :=
    \sigma^{2\beta}v(\sigma,\theta',y')\Delta_{y'}\chi(y')
\end{equation*}
and
\begin{equation*}
    f_2(\sigma,\theta',y')
    :=
    -
    2\sigma^{2\beta}
    \sum_{j=1}^m
        \partial_{y_j'}\chi(y')
        \partial_{y_j'}v(\sigma,\theta',y').
\end{equation*}
We first consider $f_1$. By Theorem~\ref{lem:Green_kernel}, estimate
\eqref{1} yields
\begin{align*}
    \sup_{Q_1}
    |\nabla_L\mathcal S f_1|
    &\le C\|v\|_{L^\infty(Q_3)}
    \sup_{\substack{0<\rho\le 1\\ \theta\in\mathbb S^{n-1}\\ y\in E_1}}
    \int_0^2
    \int_{\mathbb S^{n-1}}
    \int_{E_2} |\nabla_L \mathcal{G}(Z,Z')| \, \sigma^{n-1} \, dy' \, d\theta' \, d\sigma \\
    &\le C\|v\|_{L^\infty(Q_3)}.
\end{align*}
Next, consider $f_2$. For $Z=(\rho,\theta,y)\in Q_1$,
\begin{align*}
    \nabla_L\mathcal S f_2(Z) =
    -2\sum_{j=1}^m
    \int_0^2
    \int_{\mathbb S^{n-1}}
    \int_{E_2}
        \nabla_L\mathcal{G}(Z,Z')
        \sigma^{2\beta}
        \partial_{y_j'}\chi(y')
        \partial_{y_j'}v(Z')
    \,\sigma^{n-1}\, dy' \, d\theta' \, d\sigma .
\end{align*}
Integrating by parts in the $y_j'$-variable,
\begin{align*}
    \nabla_L\mathcal S f_2(Z)
    &= 2\sum_{j=1}^m
    \int_0^2
    \int_{\mathbb S^{n-1}}
    \int_{E_2}
        \partial_{y_j'}
        \bigl[
            \nabla_L\mathcal{G}(Z,Z')
            \sigma^{2\beta}
            \partial_{y_j'}\chi(y')
        \bigr]
        v(Z')
    \,\sigma^{n-1}\, dy' \, d\theta' \, d\sigma\\
    &= 2\sum_{j=1}^m
    \int_0^2
    \int_{\mathbb S^{n-1}}
    \int_{E_2}
        \nabla_L\partial_{y_j'}\mathcal{G}(Z,Z')
        \sigma^{2\beta}
        \partial_{y_j'}\chi(y')
        v(Z')
    \,\sigma^{n-1}\, dy' \, d\theta' \, d\sigma\\
    &\quad -2 \int_0^2 \int_{\mathbb S^{n-1}}
    \int_{E_2}
        \nabla_L\mathcal{G}(Z,Z')
        \sigma^{2\beta}
        \Delta_{y'}\chi(y')
        v(Z')
    \,\sigma^{n-1}\, dy' \, d\theta' \, d\sigma.
\end{align*}
By Theorem~\ref{lem:Green_kernel}, using \eqref{2} for the first term and \eqref{1} for the second, we obtain
\begin{align*}
    \sup_{Q_1}|\nabla_L \mathcal S f_2|
    &\le
    C\|v\|_{L^\infty(Q_3)}
    \sup_{\substack{0<\rho\leq1\\ \theta\in\mathbb S^{n-1}\\ y\in E_1}}
    \int_0^2
    \int_{\mathbb S^{n-1}}
    \int_{E_2}
        |\nabla_L \partial_{y_j'} \mathcal{G}(Z,Z')|
    \,\sigma^{n-1}\, dy' \, d\theta' \, d\sigma\\
    &\quad +
    C\|v\|_{L^\infty(Q_3)}
    \sup_{\substack{0<\rho\leq1\\ \theta\in\mathbb S^{n-1}\\ y\in E_1}}
    \int_0^2
    \int_{\mathbb S^{n-1}}
    \int_{E_2}
        |\nabla_L \mathcal{G}(Z,Z')|
    \,\sigma^{n-1}\, dy' \, d\theta' \, d\sigma\\
    &\le C\|v\|_{L^\infty(Q_3)}.
\end{align*}
Combining the estimates for $\mathcal T h$ and $\mathcal S f$ gives
\begin{equation*}
    \sup_{Q_1}|\nabla_L v| = \sup_{Q_1}|\nabla_L\omega|
    \le
    C\|v\|_{L^\infty(Q_3)},
\end{equation*}
as required. The proof of Theorem~\ref{thm:RH} is complete.
\end{proof}

\begin{proof}[Proof of Theorem~\ref{thm_main1}]
The estimate \eqref{R_p} for $1<p\leq2$ is known from
\cite[Theorem~8.1]{RS}. For $2<p<\infty$, Theorem~\ref{thm:RH}, together
with Lemmas~\ref{reduce_RH} and~\ref{RH_remote}, gives
$(\mathrm{RH}_p)$. Theorem~\ref{thm:CJKS} then yields \eqref{R_p}.
\end{proof}

\begin{corollary}\label{cor:heat-kernel-gradient}
Let $n\geq2$, $m\geq1$, $\alpha\in(0,1)$, and $\beta\geq0$. If
$p_t(\xi,\eta)$ denotes the heat kernel of $e^{-tL}$, then there exist
constants $C,c>0$ such that
\begin{equation}\label{eq:global-heat-kernel-gradient}
    |\nabla_{L,\xi}p_t(\xi,\eta)|
    \leq
    \frac{C}{\sqrt t\,V(\eta,\sqrt t)}
    \exp\left(-\frac{d(\xi,\eta)^2}{ct}\right),
    \qquad t>0,
\end{equation}
for almost every $\xi,\eta\in\mathbb R^{n+m}$.
\end{corollary}

\begin{proof}
The proof of Theorem~\ref{thm_main1} establishes
$(\mathrm{RH}_\infty)$. The Grushin space also satisfies volume doubling
and the Gaussian upper estimate. Moreover, its intrinsic metric is the
length metric associated with the carr\'e du champ $|\nabla_L f|^2$.
Consequently it supports $(\mathrm P_{\infty,\mathrm{loc}})$: the
upper-gradient inequality along intrinsic curves gives, on every fixed
bounded scale, the required oscillation estimate for locally Lipschitz
functions. Therefore the endpoint implication
$(\mathrm{RH}_\infty)\Rightarrow(\mathrm{GLY}_\infty)$ in
\cite[Theorem~1.2]{CJKS} applies and yields
\eqref{eq:global-heat-kernel-gradient}.
\end{proof}

\part{Riesz transform for \texorpdfstring{$n=1$}{n=1} and \texorpdfstring{$0<\alpha<\frac{1}{2}$}{0<alpha<1/2}}\label{part2}

The even and odd half-line problems have, respectively, zero-flux and
zero-trace boundary conditions at the singular set. The odd branch produces
the gradient singularity $|x|^{-\alpha}$.

Throughout Part~\ref{part2}, we assume $n=1$, $m\ge 1$, $\beta\ge 0$, and $\alpha \in \big(0,\frac{1}{2}\big)$, and we use the notation
\begin{align*}
\nu:= \frac{-\mathfrak b}{\mathfrak a} = \frac{1-2\alpha}{2(\beta+1-\alpha)},\qquad \rho = |x|, \qquad \mathfrak A_t:= \frac{t^{\mathfrak a}}{\mathfrak a},\qquad \tau_t := |Y| \mathfrak A_t.
\end{align*}
Note that for $n=1$, $\mathfrak b = \alpha-\frac{1}{2} <0$ (unlike the case $n\ge 2$) and $\nu\in (0,1)$.

\section{Kernel construction for \texorpdfstring{$n=1$}{n=1}}\label{sec:kernel:n=1:alpha<1/2}

By \cite[Theorem~1.1]{RS2}, the Poincar\'e inequality \eqref{P2} still holds in the case $n=1$ and $0<\alpha<1/2$. It follows from Theorem~\ref{thm:CJKS}, Lemma~\ref{reduce_RH}, and Lemma~\ref{RH_remote} that showing the $L^p$-boundedness of the Riesz transform is equivalent to proving the following theorem.

\begin{theorem}\label{thm:RH2}
Let $n = 1$, $m \ge 1$, $\alpha\in (0,\frac{1}{2})$, and $\beta \ge 0$. There is a fixed structural constant $\Lambda_1>1$ such that, for every
anchored ball $B=B((0,y_0),r)$ and every $L$-harmonic function $u$ on $\Lambda_1B$, for every $2<p < \alpha^{-1}$, there exists a constant $C>0$ such that
\begin{equation*}
    \biggl( \fint_B |\nabla_L u|^p \, d\xi \biggr)^{\!\frac{1}{p}} \le  \frac{C}{r} \fint_{\Lambda_1B} |u| \, d\xi.
\end{equation*}
\end{theorem}

\begin{proof}[Reduction of Theorem~\ref{thm:RH2}]
Choose $\Lambda_1$ so that the rescaled box $Q_4$ is contained in
$\Lambda_1B$, as in the reduction of Theorem~\ref{thm:RH}. It suffices to prove
\begin{equation*}
    \biggl( \fint_{Q_1} |\nabla_L v|^p \, d\xi \biggr)^{\!\frac{1}{p}} \le C \fint_{Q_4} |v| \, d\xi
\end{equation*}
for $1<p<\alpha^{-1}$, where $v$ is $L$-harmonic in $Q_4$.

Let $\chi \in C_c^\infty\bigl( B_m( 0, 3^{\mathfrak a} ) \bigr)$ with $\chi = 1$ on $B_m( 0, 2^{\mathfrak a} )$, and set $\omega(x,y) = \chi(y) v(x,y)$. Then
\begin{equation*}
    |\nabla_L v| = |\nabla_L \omega| \qquad \text{on } Q_1.
\end{equation*}
Consider the cylinder
\begin{equation*}
    \widetilde{\mathcal{C}}:= (-2,2)_x \times \mathbb{R}^m_y.
\end{equation*}
A straightforward computation gives
\begin{equation*}
    L\omega = f:= |x|^{2\beta} v \, \Delta_y \chi - 2 |x|^{2\beta} \nabla_y \chi \cdot \nabla_y v,
\end{equation*}
and $f$ is compactly supported in $\{2^{\mathfrak a}\le |y|\le 3^{\mathfrak a}\}$. Let
\begin{equation*}
    h_-(y):= \omega(-2,y), \qquad h_+(y):= \omega(2,y).
\end{equation*}
By the Friedrichs representation,
\begin{equation*}
    \omega = \mathcal{T}h + \mathcal{S}f,
\end{equation*}
where $\mathcal{T}h$ and $\mathcal{S}f$ are the unique solutions of the following problems, respectively:
\begin{equation}\label{eq:pde:n=1}
    \begin{cases}
        Lg = 0, & \text{in } \widetilde{\mathcal{C}},\\
        g(\pm 2,y) = h_\pm(y), \\
        \lim_{x\to 0^-} |x|^{2\alpha}\partial_x g = \lim_{x\to 0^+} |x|^{2\alpha}\partial_x g,\\
        g(0^-,y) = g(0^+,y),
    \end{cases} \qquad
    \begin{cases}
        Lg = f, & \text{in } \widetilde{\mathcal{C}},\\
        g(-2,y)=g(2,y)=0,\\
        \lim_{x\to 0^-} |x|^{2\alpha}\partial_x g = \lim_{x\to 0^+} |x|^{2\alpha}\partial_x g,\\
        g(0^-,y) = g(0^+,y),
    \end{cases}
\end{equation}
where the trace-matching and flux-matching Friedrichs conditions were discussed in Subsection~\ref{ssec:Friedrichs}; see \eqref{eq_subcritical_trace_matching} and \eqref{eq_subcritical_flux_matching}.

Define the reflection operator
\begin{equation*}
    \mathfrak Ru(x,y):= u(-x,y).
\end{equation*}
Since $L$ is invariant under reflection, the Friedrichs form is also reflection-invariant, namely
\begin{equation*}
    \mathfrak q(\mathfrak Ru,\mathfrak Rv) = \mathfrak q(u,v).
\end{equation*}
Equivalently,
\begin{equation*}
    \mathfrak q(\mathfrak Ru,v) = \mathfrak q(u,\mathfrak Rv)
\end{equation*}
for all admissible test functions $v$. Suppose that $g\in \mathcal{F}_{\mathrm{loc}}(\widetilde{\mathcal{C}})$ solves the equation $Lg=F$. Set
\begin{equation*}
    g_e:= \frac{g+\mathfrak Rg}{2},\qquad g_o:=\frac{g-\mathfrak Rg}{2},\qquad F_e:=\frac{F+\mathfrak RF}{2},\qquad F_o:=\frac{F-\mathfrak RF}{2}.
\end{equation*}
Then
\begin{equation*}
    Lg_e=F_e,\qquad L g_o = F_o.
\end{equation*}
Indeed, for every compactly supported Friedrichs test function $v$,
\begin{align*}
    \mathfrak q(g_e,v)
    &= \frac{1}{2} \mathfrak q(g,v) + \frac{1}{2} \mathfrak q(\mathfrak R g,v)\\
    &= \frac{1}{2} \mathfrak q(g,v) + \frac{1}{2} \mathfrak q(g,\mathfrak Rv)\\
    &= \frac{1}{2} \int F v \, dx\,dy  + \frac{1}{2} \int F (\mathfrak Rv) \, dx\,dy\\
    &= \int F_e v \, dx\,dy.
\end{align*}
The proof for $g_o$ is identical. Thus the even and odd components are not merely formal algebraic pieces, but genuine Friedrichs weak solutions of the corresponding even and odd equations.

Now set
\begin{equation*}
    h_{e}:= \frac{h_+ + h_-}{2},\qquad h_{o}:= \frac{h_+ - h_-}{2},\qquad f_e:= \frac{f + \mathfrak R f}{2},\qquad f_o:= \frac{f - \mathfrak R f}{2},
\end{equation*}
where `e' and `o' denote the even and odd parts, respectively. Recalling the endpoint Friedrichs conditions for even and odd solutions discussed in Subsection~\ref{ssec:Friedrichs}, decompose
\begin{align*}
    &\mathcal{T}h(x,y) = \mathcal{T}_e h_e(|x|, y) + \operatorname{sgn}(x)\, \mathcal{T}_o h_o(|x|,y),\\
    &\mathcal{S}f(x,y) = \mathcal{S}_e f_e(|x|,y) + \operatorname{sgn}(x)\, \mathcal{S}_o f_o(|x|,y),
\end{align*}
where $\mathcal{T}_{e}h_{e}$ and $\mathcal{T}_o h_o$ are the unique solutions of the half-line problems
\begin{align}\label{eq:evpdd:Poisson}
\begin{cases}
        Lg = 0, & \text{in } (0,2)\times \mathbb{R}^m,\\
        g(2,y)= h_e(y),\\
        \lim_{x\to 0^+} x^{2\alpha} \partial_x g = 0,
    \end{cases} \qquad
    \begin{cases}
        Lg = 0, & \text{in } (0,2)\times \mathbb{R}^m,\\
        g(2,y)= h_o(y),\\
        g(0^+,y) = 0.
    \end{cases}
\end{align}
Similarly, $\mathcal{S}_e f_e$ and $\mathcal{S}_o f_o$ solve
\begin{align}\label{eq:evpdd:Green}
\begin{cases}
        Lg = f_e, & \text{in } (0,2)\times \mathbb{R}^m,\\
        g(2,y)= 0,\\
        \lim_{x\to 0^+} x^{2\alpha} \partial_x g = 0,
    \end{cases} \qquad
    \begin{cases}
        Lg = f_o, & \text{in } (0,2)\times \mathbb{R}^m,\\
        g(2,y)= 0,\\
        g(0^+,y) = 0.
    \end{cases}
\end{align}
The uniqueness of these Friedrichs problems follows directly from the energy identity.

Accordingly, the problem has been reduced to solving \eqref{eq:evpdd:Poisson} and \eqref{eq:evpdd:Green}, and to estimating
\begin{equation*}
    |\nabla_L \mathcal{T}_e h_e|,\qquad |\nabla_L \mathcal{T}_o h_o|, \qquad |\nabla_L \mathcal{S}_o f_o|, \qquad |\nabla_L \mathcal{S}_e f_e|.
\end{equation*}
\end{proof}

We construct the two half-line kernels below and prove their estimates
in Section~\ref{sec:pf:n=1:alpha<1/2}.

\subsection{The Poisson multiplier for \texorpdfstring{$n=1$}{n=1}}

We first treat equation \eqref{eq:evpdd:Poisson}. Consider the partial differential equation
\begin{equation*}
    Lg = -\partial_x (|x|^{2\alpha}\partial_x g) + |x|^{2\beta} \Delta_y g = 0.
\end{equation*}
Taking the Fourier transform in the $y$-variable, it is enough to consider, for $\rho>0$,
\begin{align*}
    - \rho^{2\alpha} \partial_\rho^2 \hat{g}(\rho,Y) - 2\alpha \rho^{2\alpha-1} \partial_\rho \hat{g}(\rho,Y) + \rho^{2\beta} |Y|^2 \hat{g}(\rho,Y) = 0,
\end{align*}
where
\begin{equation*}
    \hat{g}(\rho,Y) = \int_{\mathbb{R}^m} g(\rho,y) e^{-i y \cdot Y} \, dy.
\end{equation*}
Introduce
\begin{equation*}
    t:= \frac{|Y|}{\mathfrak a} \rho^{\mathfrak a}, \qquad U_Y(t) := \rho^{\mathfrak b} \hat{g}(\rho, Y).
\end{equation*}
Then the equation becomes
\begin{equation*}
    t^2 U_Y''(t) + t U_Y'(t) - (t^2+\nu^2)U_Y(t) = 0.
\end{equation*}
The modified Bessel equation above has two basic solutions:
\begin{equation*}
    \phi_Y(\rho):=\rho^{-\mathfrak b} I_\nu\!\left( \frac{|Y|}{\mathfrak a} \rho^{\mathfrak a} \right), \qquad \varphi_Y(\rho):=\rho^{-\mathfrak b} K_\nu\!\left( \frac{|Y|}{\mathfrak a} \rho^{\mathfrak a} \right).
\end{equation*}
The solution is required to satisfy either the zero-flux or the zero-trace condition. Note that for $0<\nu<1$, as $\rho \to 0$,
\begin{align*}
\phi_Y(\rho) &= \frac{1}{2^\nu \Gamma(\nu+1)} \Bigl(\frac{|Y|}{\mathfrak a}\Bigr)^{\!\nu} \rho^{-2\mathfrak b} + O_{\nu,Y}(\rho^{2\mathfrak a - 2\mathfrak b}), \\
\varphi_Y(\rho) &= 2^{\nu-1} \Gamma(\nu) \Bigl(\frac{|Y|}{\mathfrak a}\Bigr)^{\!-\nu} + 2^{-\nu-1} \Gamma(-\nu) \Bigl(\frac{|Y|}{\mathfrak a}\Bigr)^{\!\nu} \rho^{-2\mathfrak b} + O_{\nu,Y}(\rho^{2\mathfrak a}).
\end{align*}
Set
\begin{equation*}
    c_\nu := 2^{-1} \Gamma(-\nu) \Gamma(\nu+1) = -\frac{\pi}{2\sin{\pi \nu}}.
\end{equation*}
To match the zero-flux condition, define
\begin{align*}
    \Phi_Y(\rho)
    :&= \varphi_Y(\rho) - c_\nu \phi_Y(\rho)\\
    &= \rho^{-\mathfrak b} \Bigl( K_\nu\Bigl(\frac{|Y|}{\mathfrak a}\rho^\mathfrak a \Bigr) + \frac{\pi}{2\sin{\pi \nu}} I_\nu \Bigl(\frac{|Y|}{\mathfrak a}\rho^\mathfrak a \Bigr) \Bigr)\\
    &= \frac{\pi}{2\sin{\pi \nu}} \rho^{-\mathfrak b} I_{-\nu}\Bigl(\frac{|Y|}{\mathfrak a}\rho^\mathfrak a\Bigr)\\
    &= O_{\nu,Y}(1) + O_{\nu,Y}(\rho^{2\mathfrak a}),
\end{align*}
where the second-to-last equality follows from the identity
\begin{equation*}
    K_\nu(t) = \frac{\pi}{2\sin{\pi \nu}} \bigl( I_{-\nu}(t) - I_\nu(t) \bigr), \qquad 0<\nu<1, \quad t>0.
\end{equation*}
In particular,
\begin{equation*}
    \lim_{\rho \to 0} \rho^{2\alpha} \partial_\rho \Phi_Y(\rho) = \lim_{\rho \to 0} \rho^{2\beta+1} = 0.
\end{equation*}
To match the zero-trace condition, we simply take the solution $\phi_Y(\rho)$.

The boundary condition at $\rho =2$ must also be imposed. Define the multipliers
\begin{equation*}
    M_e(\rho,Y):= \frac{\Phi_Y(\rho)}{\Phi_Y(2)} = \frac{\rho^{-\mathfrak b}}{2^{-\mathfrak b}} \frac{I_{-\nu}\!\left(\frac{|Y|}{\mathfrak a}\rho^\mathfrak a\right)}{I_{-\nu}\!\left(\frac{|Y|}{\mathfrak a}2^\mathfrak a\right)}, \qquad M_o(\rho,Y):= \frac{\phi_Y(\rho)} {\phi_Y(2)} = \frac{\rho^{-\mathfrak b}}{2^{-\mathfrak b}} \frac{I_{\nu}\!\left(\frac{|Y|}{\mathfrak a}\rho^\mathfrak a\right)}{I_{\nu}\!\left(\frac{|Y|}{\mathfrak a}2^\mathfrak a\right)}.
\end{equation*}
Their values at $|Y|=0$ are understood by continuity, namely
\begin{equation*}
    M_e(\rho,0)=1 \qquad \text{and} \qquad M_o(\rho,0)= (\rho/2)^{1-2\alpha}.
\end{equation*}
Then $M_e(x,Y) \hat{h}_e(Y)$ and $M_o(x,Y) \hat{h}_o(Y)$ solve the Fourier-transformed equations corresponding to \eqref{eq:evpdd:Poisson}. Therefore, for $x\in (-2,2)$,
\begin{align*}
    \widehat{\mathcal{T}h}(x,Y) = M_e(|x|,Y) \hat{h}_e(Y) + \operatorname{sgn}(x) M_o(|x|,Y) \hat{h}_o(Y),
\end{align*}
and hence
\begin{align}
    \mathcal{T}h(x,y) &= \mathcal{T}_e h_e(x,y) + \mathcal{T}_o h_o(x,y) \nonumber\\
    &=(2\pi)^{-m}\int_{\mathbb{R}^m} \biggl( \int_{\mathbb{R}^m} M_e(|x|, Y) e^{i (y-y')\cdot Y} \, dY \biggr) h_e(y') \, dy' \nonumber\\
    &\quad+ (2\pi)^{-m}\operatorname{sgn}(x) \int_{\mathbb{R}^m} \biggl( \int_{\mathbb{R}^m} M_o(|x|, Y) e^{i (y-y')\cdot Y} \, dY \biggr) h_o(y') \, dy'.
\end{align}

\subsection{The Green kernel for \texorpdfstring{$n=1$}{n=1}}

We next turn to equation \eqref{eq:evpdd:Green}. The construction of the Green kernel is similar to that in the case $n\ge 2$, although here the equation is split into its even and odd parts.

Let $\mathcal{G}(Z,Z')$ (with $Z = (x,y)$ and $Z' = (x',y')$) denote the Green kernel associated with $L$ on $\widetilde{\mathcal{C}}$ under the Dirichlet boundary condition at $|x|=2$ and the local finite-energy condition at $x=0$. As before, taking the Fourier transform in the $y$-$y'$ variables shows that the Green kernel takes the form
\begin{equation}\label{eq_Green_kernel:n=1}
    \mathcal{G}(Z,Z') = (2\pi)^{-m}\int_{\mathbb{R}^m} e^{i(y-y') \cdot Y} G_{Y}(x, x') \, dY,
\end{equation}
where $G_{Y}$ is the one-dimensional Green kernel associated with the operator
\begin{equation*}
    \mathcal{L}_{Y}:=-\partial_x\bigl(|x|^{2\alpha}\partial_x\bigr)+|Y|^2|x|^{2\beta}.
\end{equation*}
Equivalently,
\begin{equation*}
    \mathcal{S}f(Z) = \iint_{\widetilde{\mathcal{C}}} \mathcal{G}(Z,Z') f(Z') \, dx' \, dy'.
\end{equation*}
Similarly, decompose the Green kernel $G_Y(x,x')$ into even and odd parts corresponding to the Fourier-transformed half-line problems in \eqref{eq:evpdd:Green}. First, to impose the vanishing condition at $x=2$, recall the function
\begin{equation}\label{eq:psi}
    \psi_{Y}(\rho):= \rho^{-\mathfrak b} \Bigl[ K_{\nu}\Bigl( \frac{|Y| \rho^{\mathfrak a}}{\mathfrak a} \Bigr) - \eta_Y I_{\nu} \Bigl( \frac{|Y| \rho^{\mathfrak a}}{\mathfrak a} \Bigr) \Bigr], \qquad \eta_Y:= \frac{K_{\nu} \bigl( \frac{|Y| 2^{\mathfrak a}}{\mathfrak a} \bigr)}{I_{\nu} \bigl( \frac{|Y| 2^{\mathfrak a}}{\mathfrak a} \bigr)}.
\end{equation}
Note that $\psi_{Y}(2)=0$. Next, for the zero-flux condition at $x=0$, use the function $\Phi_Y(x)$. Finally, for the zero-trace condition at $x=0$, use the function $\phi_Y(x)$.

Denote by $G_{e,Y}$ and $G_{o,Y}$ the Green operators for the two
Fourier-transformed equations in \eqref{eq:evpdd:Green}.  A calculation
parallel to that in Subsection~\ref{ssec:Green:n>2} yields
\begin{align}\label{eq_kernel_Green:n=1:even}
G_{e,Y}(\rho,\sigma) &= \mathfrak a^{-1} (\eta_Y - c_\nu)^{-1}
\begin{cases}
\Phi_{Y}(\rho) \, \psi_{Y}(\sigma), &0<\rho<\sigma,\\[4pt]
\Phi_{Y}(\sigma) \, \psi_{Y}(\rho), &\sigma < \rho < 2,
\end{cases}\nonumber\\
&=\mathfrak a^{-1} \frac{2\sin{\pi \nu}}{\pi} \frac{I_{\nu}\bigl( \frac{|Y| 2^{\mathfrak a}}{\mathfrak a} \bigr)}{I_{-\nu}\bigl( \frac{|Y| 2^{\mathfrak a}}{\mathfrak a} \bigr)}
\begin{cases}
\Phi_{Y}(\rho) \, \psi_{Y}(\sigma), &0<\rho<\sigma,\\[4pt]
\Phi_{Y}(\sigma) \, \psi_{Y}(\rho), &\sigma < \rho < 2,
\end{cases}
\end{align}
and
\begin{align}\label{eq_kernel_Green:n=1:odd}
G_{o,Y}(\rho,\sigma) = \mathfrak a^{-1}
\begin{cases}
\phi_{Y}(\rho) \, \psi_{Y}(\sigma), &0<\rho<\sigma,\\[4pt]
\phi_{Y}(\sigma) \, \psi_{Y}(\rho), &\sigma < \rho < 2.
\end{cases}
\end{align}
Consequently,
\begin{align}\label{eq:Green:n=1:full}
    \mathcal{S}f(x,y) &= (2\pi)^{-m}\int_0^2 \int_{\mathbb{R}^m} \biggl( \int_{\mathbb{R}^m} G_{e,Y}(|x|,x') e^{i (y-y')\cdot Y} \, dY \biggr) f_e(x',y') \, dy' \, dx' \nonumber\\
    &\quad + (2\pi)^{-m}\operatorname{sgn}(x) \int_0^2 \int_{\mathbb{R}^m} \biggl( \int_{\mathbb{R}^m} G_{o,Y}(|x|,x') e^{i (y-y')\cdot Y} \, dY \biggr) f_o(x',y') \, dy' \, dx'.
\end{align}

\section{Proof of Theorem~\ref{thm:RH2}}\label{sec:pf:n=1:alpha<1/2}

Recall the notation
\begin{equation*}
    \mathcal{I}_\nu (t) := t^{-\nu} I_\nu(t), \qquad S_\nu(t):= \frac{I_{\nu+1}(t)}{I_\nu(t)},\qquad H_{\nu,\rho}(r):= \frac{\mathcal{I}_\nu(\mathfrak A_\rho r)}{\mathcal{I}_\nu(\mathfrak A_2 r)}, \qquad \nu>-1.
\end{equation*}

\begin{lemma}\label{lem:Poisson:n=1}
Let $N\ge 0$, $\mu\in \{-\nu, \nu \}$, and $0<\rho \le 1$. Then for every multi-index $|\gamma|\le N$,
\begin{equation*}
        |\partial_Y^{\gamma} H_{\mu,\rho}(|Y|)| \le \begin{cases}
            C_N, & |Y|\le C^*,\\[4pt]
            C_N (1+|Y|)^{A_N} e^{-c_N |Y|}, & |Y|\ge C^*,
        \end{cases}
\end{equation*}
where $C^*>1$ is a sufficiently large constant.
\end{lemma}

\begin{proof}
By Lemma~\ref{lem_S_derivative_bound}, the estimates for $S_\mu$ imply the corresponding estimates for $H_{\mu,\rho}$. The proof is the same as that of Proposition~\ref{lem_transition_ratio_derivatives}; see also Remark~\ref{re:H}.
\end{proof}

\begin{theorem}\label{prop:Poisson:n=1}
Let $n= 1$, $m\ge 1$, $\alpha\in \big(0,\frac{1}{2} \big)$, and $\beta \ge 0$. Let $0<\rho\le 1$. Then for every $N\ge 0$ and all $|\gamma|\le N$,
\begin{align}\label{eq:prop:Poisson:n=1:even}
    \bigl| \partial_Y^{\gamma} m_{e,\rho}^{(*)}(Y) \bigr| \le \begin{cases}
        C_N, & |Y|\le C^*,\\[4pt]
        C_N (1+|Y|)^{A_N} e^{-c_N |Y|}, & |Y|\ge C^*,
    \end{cases}
\end{align}
while
\begin{align}\label{eq:prop:Poisson:n=1:odd}
    \bigl| \partial_Y^{\gamma} m_{o,\rho}^{(*)}(Y) \bigr| \le \rho^{-\alpha} \begin{cases}
        C_N, & |Y|\le C^*,\\[4pt]
        C_N (1+|Y|)^{A_N} e^{-c_N |Y|}, & |Y|\ge C^*,
    \end{cases}
\end{align}
where $*\in\{\mathrm{rad}, \mathrm{bas},j\}$.
\end{theorem}

\begin{proof}
We first consider the even multiplier $M_{e,\rho}(Y) = M_{e}(\rho,Y)$. Clearly,
\begin{equation*}
    M_e(\rho,Y) = H_{-\nu,\rho}(|Y|).
\end{equation*}
Thus an argument parallel to that in Proposition~\ref{lem_detailed_poisson_multiplier}, with the order $\nu_\ell$ replaced by the fixed order $-\nu$, applies here as well. In particular, note that
\begin{equation*}
    m_{e,\rho}^{\mathrm{rad}}(Y) = \rho^\alpha \partial_\rho H_{-\nu,\rho}(|Y|) = \rho^\beta |Y| S_{-\nu}\bigl(\mathfrak A_\rho |Y| \bigr) H_{-\nu,\rho}(|Y|).
\end{equation*}
Combining this with Lemma~\ref{lem:Poisson:n=1} gives \eqref{eq:prop:Poisson:n=1:even} for $*=\mathrm{rad}$. The case $*=\mathrm{bas},j$ is entirely similar.

Next, consider $M_{o,\rho}(Y)$. Observe that
\begin{equation*}
    M_{o}(\rho,Y) = \Bigl(\frac{\rho}{2}\Bigr)^{\!1-2\alpha} H_{\nu,\rho}(|Y|).
\end{equation*}
Consequently,
\begin{align*}
    m_{o,\rho}^{\mathrm{rad}}(Y) &= 2^{2\alpha-1}(1-2\alpha) \rho^{-\alpha} H_{\nu,\rho}(|Y|) + 2^{2\alpha-1} \rho^{1-\alpha} \partial_\rho H_{\nu,\rho}(|Y|)\\
    &= 2^{2\alpha-1}(1-2\alpha) \rho^{-\alpha} H_{\nu,\rho}(|Y|) + 2^{2\alpha-1} \rho^{\beta+1-2\alpha} |Y| S_{\nu}(\mathfrak A_\rho |Y|) H_{\nu,\rho}(|Y|).
\end{align*}
Since $\rho^{\beta+1-2\alpha}\le \rho^{-\alpha}$, as $\beta+1-\alpha >0$, \eqref{eq:prop:Poisson:n=1:odd} for $*=\mathrm{rad}$ follows from Lemma~\ref{lem:Poisson:n=1}. The case $*=\mathrm{bas},j$ is analogous.
\end{proof}

\begin{remark}\label{re:even:Poisson}
The proof of the even multiplier estimate in Theorem~\ref{prop:Poisson:n=1} does not use the restriction $0<\alpha<1/2$, except through the notation associated with the weakly degenerate regime. The relevant branch is the zero-flux Friedrichs branch, which is expressed by the normalized Bessel ratio with order $-\nu>-1$. The same argument therefore applies to the zero-flux branch in the strongly degenerate case $1/2\leq\alpha<1$, after replacing $-\nu$ by the corresponding non-negative order. This observation will be used in Part~\ref{part4}.
\end{remark}

\begin{proposition}\label{prop:Green:n=1}
Let $n=1$, $m\ge 1$, $\alpha\in \big(0,\frac{1}{2}\big)$, and $\beta\ge 0$. Then for $0<\rho\le 2$,
\begin{align}\label{eq:prop:Green:1}
    \int_0^2 |G_{o,Y}(\rho,\sigma)| \, d\sigma \le C\,\frac{\rho^{2-2\alpha} + \rho^{1-2\alpha}}{1+\tau_\rho}\le C.
\end{align}
Moreover, for $|Y|>C^*$,
\begin{align}\label{eq:prop:Green:2}
    \sup_{0<\rho\le 2} \rho^{2\beta}\int_0^2 |G_{o,Y}(\rho,\sigma)| \, d\sigma \le C|Y|^{-2}.
\end{align}
In addition, for $0<\rho\le 1$,
\begin{equation}\label{eq:prop:Green:3}
     \int_0^2 |D_* G_{o,Y}(\rho,\sigma)| \, d\sigma \le \rho^{-\alpha}\begin{cases}
         C, & |Y|\le C^*,\\[4pt]
         C|Y|, & |Y|>C^*,
     \end{cases}
\end{equation}
where $*\in\{r,y,j\}$.
\end{proposition}

\begin{proof}
The proof parallels those of Propositions~\ref{prop:Green:Shur}, \ref{lem_basic_D_R_schur_explicit}, and~\ref{lem:second_to_last}, and Corollary~\ref{cor_weighted_Y_moment_Green}. Recall that
\begin{align*}
G_{o,Y}(\rho,\sigma) = \mathfrak a^{-1}
\begin{cases}
\phi_{Y}(\rho) \, \psi_{Y}(\sigma), &0<\rho<\sigma,\\[4pt]
\phi_{Y}(\sigma) \, \psi_{Y}(\rho), &\sigma < \rho < 2.
\end{cases}
\end{align*}
It follows from the argument in Proposition~\ref{prop:Green:Shur} that
\begin{equation*}
    \frac{\phi_Y(\sigma)}{\phi_Y(\rho)}\le \Bigl(\frac{\sigma}{\rho}\Bigr)^{\!-2\mathfrak b}\quad \text{if}\quad \sigma\le \rho,\qquad \frac{\psi_Y(\sigma)}{\psi_Y(\rho)}\le C \quad \text{if}\quad \rho \le \sigma\le 2.
\end{equation*}
Moreover, Lemma~\ref{lem_bessel_inputs_high_frequency} and \eqref{eq:product:Bessel} give $\phi_Y(\rho) \psi_Y(\rho) \le C\rho^{-2\mathfrak b}(1+\tau_\rho)^{-1}$.
Hence, for $0<\rho\le 2$, a straightforward computation shows
\begin{align*}
    \int_0^2 G_{o,Y}(\rho,\sigma) \, d\sigma
    &\le C \phi_Y(\rho) \psi_Y(\rho)  \left(\int_0^\rho \frac{\phi_Y(\sigma)}{\phi_Y(\rho)} \, d\sigma  +  \int_\rho^2 \frac{\psi_Y(\sigma)}{\psi_Y(\rho)} \, d\sigma \right)\le C\,\frac{\rho^{2-2\alpha} + \rho^{1-2\alpha}}{1+\tau_\rho}.
\end{align*}
Since $1-2\alpha >0$, this proves \eqref{eq:prop:Green:1}. In particular, if $\rho \le 1$, then $\int_0^2 G_{o,Y}(\rho,\sigma) \, d\sigma \le C\rho^{1-2\alpha}(1+\tau_\rho)^{-1}$.

Estimate \eqref{eq:prop:Green:2} follows from the same argument as in \textit{Case 2} of the proof of Proposition~\ref{prop:Green:Shur}, using the asymptotic expansions of the Bessel functions. Indeed, for $0<\rho\le 2$,
\begin{align*}
     \rho^{2\beta}\int_0^2 |G_{o,Y}(\rho,\sigma)| \, d\sigma &\le \rho^{2\beta} \int_0^\rho \phi_Y(\sigma) \psi_Y(\rho) d\sigma + \rho^{2\beta} \int_\rho^2 \phi_Y(\rho) \psi_Y(\sigma) d\sigma\\
     &\le C \rho^{2\beta -\mathfrak b} K_\nu(\tau_\rho) \int_0^\rho \sigma^{-\mathfrak b} I_\nu(\tau_\sigma) d\sigma + C \rho^{2\beta -\mathfrak b} I_\nu(\tau_\rho) \int_\rho^\infty \sigma^{-\mathfrak b} K_\nu(\tau_\sigma) d\sigma\\
     &=C|Y|^{-2} \tau_\rho^{\frac{2\beta-\mathfrak b}{\mathfrak a}} K_\nu(\tau_\rho) \int_0^{\tau_\rho} u^{\frac{1-\mathfrak b}{\mathfrak a}-1} I_\nu(u) du + C|Y|^{-2} \tau_\rho^{\frac{2\beta-\mathfrak b}{\mathfrak a}} I_\nu(\tau_\rho) \int_{\tau_\rho}^\infty u^{\frac{1-\mathfrak b}{\mathfrak a}-1} K_\nu(u) du\\
     &= C|Y|^{-2} \left( II\left(\frac{2\beta-\mathfrak b}{\mathfrak a}, \frac{1-\mathfrak b}{\mathfrak a},\nu, \tau_\rho \right) + I\left(\frac{2\beta-\mathfrak b}{\mathfrak a}, \frac{1-\mathfrak b}{\mathfrak a},\nu,0, \tau_\rho \right) \right) \le C |Y|^{-2},
\end{align*}
where the last inequality follows by Lemma~\ref{lem:Bessel:integral}.

We next prove the gradient estimate \eqref{eq:prop:Green:3}. By Proposition~\ref{lem_basic_D_R_schur_explicit} (see~\eqref{D:bdd:G}), for all $0<\rho\le 1$,
\begin{align*}
    \int_0^2 |D_r G_{o,Y}(\rho,\sigma)| \, d\sigma
    &\le C(1+\tau_\rho) \rho^{\alpha-1} \int_0^2 G_{o,Y}(\rho,\sigma) \, d\sigma \le C \rho^{-\alpha}.
\end{align*}
This, together with \eqref{eq:prop:Green:1}, yields
\begin{align*}
    \int_0^2 |D_r G_{o,Y}(\rho,\sigma)| \, d\sigma \le C \rho^{-\alpha},\qquad 0<\rho\le 1.
\end{align*}
Finally, consider the base derivative $D_{y,j}$. Clearly,
\begin{align*}
    \int_0^2 |D_{y,j}G_{o,Y}(\rho,\sigma)| \, d\sigma
    &\le C|Y| \rho^{\beta} \int_0^2 G_{o,Y}(\rho,\sigma) \, d\sigma \le C|Y| \rho^{\beta+1-2\alpha}\le C |Y| \rho^{-\alpha}.
\end{align*}
This proves \eqref{eq:prop:Green:3}.
\end{proof}

\begin{proposition}\label{prop:Green:n=1:even}
Let $n=1$, $m\ge 1$, $\alpha\in (0,\frac12)$, and $\beta\ge 0$. Then for $|Y|\le C^*$,
\begin{align}\label{eq:prop:Green:even:1}
\sup_{0<\rho\le 2} \int_0^2 |G_{e,Y}(\rho,\sigma)| \, d\sigma \le C.
\end{align}
Moreover, for $|Y|>C^*$,
\begin{align}\label{eq:prop:Green:even:2}
    \sup_{0<\rho\le 2} \rho^{2\beta}\int_0^2 |G_{e,Y}(\rho,\sigma)| \, d\sigma \le C|Y|^{-2}.
\end{align}
In addition, for $0<\rho\le 1$,
\begin{equation}\label{eq:prop:Green:even:3}
     \int_0^2 |D_* G_{e,Y}(\rho,\sigma)| \, d\sigma \le \begin{cases}
         C, & |Y|\le C^*,\\[4pt]
         C|Y|, & |Y|>C^*,
     \end{cases}
\end{equation}
where $*\in\{r,y,j\}$.
\end{proposition}

\begin{proof}
Recall that
\begin{align}\label{eq:normal}
G_{e,Y}(\rho,\sigma) = \mathfrak a^{-1} \frac{2\sin{\pi \nu}}{\pi} \frac{I_{\nu}\!\bigl( \frac{|Y| 2^{\mathfrak a}}{\mathfrak a} \bigr)}{I_{-\nu}\!\bigl( \frac{|Y| 2^{\mathfrak a}}{\mathfrak a} \bigr)}
\begin{cases}
\Phi_{Y}(\rho) \, \psi_{Y}(\sigma), &0<\rho<\sigma,\\[4pt]
\Phi_{Y}(\sigma) \, \psi_{Y}(\rho), &\sigma < \rho < 2.
\end{cases}
\end{align}
To prove \eqref{eq:prop:Green:even:1}, define
\begin{equation*}
    \mathcal{U}_Y(\rho):= \frac{\rho^{- \mathfrak b} I_{-\nu}(\tau_\rho)}{2^{-\mathfrak b} I_{-\nu}(\tau_2)},\qquad \mathcal{V}_Y(\rho):= 2^{-\mathfrak b} \rho^{-\mathfrak b} \bigl[ I_\nu(\tau_2) K_\nu(\tau_\rho) - K_\nu(\tau_2) I_\nu(\tau_\rho) \bigr],
\end{equation*}
so that $G_{e,Y}(\rho,\sigma)$ can be written in the normalized form
\begin{equation*}
G_{e,Y}(\rho,\sigma) = \mathfrak a^{-1}
\begin{cases}
\mathcal{U}_{Y}(\rho) \, \mathcal{V}_{Y}(\sigma), &0<\rho<\sigma,\\[4pt]
\mathcal{U}_{Y}(\sigma) \, \mathcal{V}_{Y}(\rho), &\sigma < \rho < 2.
\end{cases}
\end{equation*}
Since $\mathfrak b = - \mathfrak a \nu$, it is clear that $\mathcal{U}_Y(\rho) = H_{-\nu,\rho}(|Y|)$. Moreover, using the identity $K_\nu(t) = \frac{\pi}{2 \sin{\pi \nu}} \bigl( I_{-\nu}(t) - I_\nu(t) \bigr)$,
one obtains
\begin{equation*}
    \mathcal{V}_Y(\rho) = C_{\alpha,\beta,\nu} \, \mathcal{I}_\nu(\tau_2) \, \mathcal{I}_{-\nu}(\tau_2) \Bigl[ H_{-\nu,\rho}(|Y|) - \Bigl(\frac{\rho}{2}\Bigr)^{\!1-2\alpha} H_{\nu,\rho}(|Y|) \Bigr].
\end{equation*}
If $|Y|\le C^*$, then $\tau_2\le C$, and therefore $\mathcal{I}_\nu(\tau_2) \mathcal{I}_{-\nu}(\tau_2)$ is bounded by a constant depending only on $\alpha,\beta,\nu$. Hence Lemma~\ref{lem:Poisson:n=1} implies
\begin{equation*}
    |\mathcal{U}_Y(\rho)|+|\mathcal{V}_Y(\rho)|\le C.
\end{equation*}
Furthermore, the identity $\rho^{\alpha} \partial_\rho H_{\pm \nu, \rho}(|Y|) = \rho^{\beta} |Y| S_{\pm \nu}(\mathfrak A_\rho |Y|) H_{\pm \nu, \rho}(|Y|)$, together with
Lemma~\ref{lem_S_derivative_bound} gives
\begin{equation*}
    |D_r \mathcal{U}_Y(\rho)|\le C
    \qquad\text{and}\qquad
    |D_r \mathcal{V}_Y(\rho)|\le C \rho^{-\alpha}.
\end{equation*}
Immediately, for all $0<\rho\le 1$,
\begin{align*}
    \int_0^2 |D_r G_{e,Y}(\rho,\sigma)| \, d\sigma \le C \int_0^\rho \rho^{-\alpha}\, d\sigma + C\int_\rho^2 \, d\sigma \le C \rho^{1-\alpha} + C\le C.
\end{align*}
The estimate for the base derivative $D_{y,j}$ is straightforward. Therefore \eqref{eq:prop:Green:even:1} and \eqref{eq:prop:Green:even:3} for $|Y|\le C^*$ follow directly.

To prove \eqref{eq:prop:Green:even:2}, first note that for $|Y|\ge C^*$, the factor $\frac{I_\nu(\tau_2)}{I_{-\nu}(\tau_2)}\le C$ since $\tau_2 \ge \mathfrak a^{-1} 2^{\mathfrak a}$. Next, write
\begin{align*}
    \rho^{2\beta}\int_0^2 |G_{e,Y}(\rho,\sigma)| \, d\sigma &\le C \rho^{2\beta} \psi_Y(\rho) \int_0^\rho \Phi_Y(\sigma) \,d\sigma + C \rho^{2\beta} \Phi_Y(\rho) \int_\rho^2 \psi_Y(\sigma)\, d\sigma\\
    &\le C \rho^{2\beta - \mathfrak b} K_\nu(\tau_\rho) \int_0^\rho \sigma^{-\mathfrak b} I_{-\nu}(\tau_\sigma) \, d\sigma + C \rho^{2\beta - \mathfrak b} I_{-\nu}(\tau_\rho) \int_\rho^\infty \sigma^{-\mathfrak b} K_{\nu}(\tau_\sigma) \, d\sigma.
\end{align*}
After a change of variables, it equals
\begin{align*}
    &C |Y|^{-2} \tau_\rho^{\frac{2\beta-\mathfrak b}{\mathfrak a}} K_\nu(\tau_\rho) \int_0^{\tau_\rho} u^{\frac{1-\mathfrak b}{\mathfrak a}-1}  I_{-\nu}(u) \, du + C |Y|^{-2} \tau_\rho^{\frac{2\beta-\mathfrak b}{\mathfrak a}} I_{-\nu}(\tau_\rho) \int_{\tau_\rho}^\infty u^{\frac{1-\mathfrak b}{\mathfrak a}-1}  K_{\nu}(u) \,  du \\
    &= C |Y|^{-2} II\left(\frac{2\beta-\mathfrak b}{\mathfrak a}, \frac{1-\mathfrak b}{\mathfrak a}, -\nu, \tau_\rho \right) + C |Y|^{-2} I\left(\frac{2\beta-\mathfrak b}{\mathfrak a}, \frac{1-\mathfrak b}{\mathfrak a}, -\nu, 0, \tau_\rho \right)\le C|Y|^{-2},
\end{align*}
where the last inequality follows by Lemma~\ref{lem:Bessel:integral} and $\frac{2\beta-\mathfrak b}{\mathfrak a}+ \frac{1-\mathfrak b}{\mathfrak a}=2$ and $\nu\in (0,1)$. This proves \eqref{eq:prop:Green:even:2}.

It remains to prove \eqref{eq:prop:Green:even:3} for $|Y|\ge C^*$. Use the expression \eqref{eq:normal}. The prefactor $I_{\nu}( \tau_2 )/I_{-\nu}( \tau_2 )$ is again harmless. The key observation is that
\begin{align*}
\rho^\alpha\frac{\Phi_{Y}'(\rho)}{\Phi_{Y}(\rho)}
&= \mathfrak a \rho^{\alpha-1}\tau_\rho \frac{I_{1-\nu}(\tau_\rho)}{I_{-\nu}(\tau_\rho)}.
\end{align*}
Also, the same argument used for \eqref{eq:normal2} shows that
\begin{equation*}
    \Bigl|\rho^\alpha \frac{\psi_Y'(\rho)}{\psi_Y(\rho)} \Bigr| \le C (1+\tau_\rho) \rho^{\alpha-1}.
\end{equation*}
Writing $t:=\tau_\rho$, we therefore obtain, for $0<\rho\le 1$, from Lemma~\ref{lem:Bessel:integral}
\begin{align*}
    &\int_0^2 |D_r G_{e,Y}(\rho,\sigma)| \, d\sigma
    \le C(1+\tau_\rho) \rho^{\alpha-1} \int_0^\rho \psi_Y(\rho) \Phi_Y(\sigma) \, d\sigma + C\tau_\rho \rho^{\alpha-1} \frac{I_{1-\nu}(\tau_\rho)}{I_{-\nu}(\tau_\rho)} \int_\rho^2 \psi_Y(\sigma) \Phi_Y(\rho) \, d\sigma \\
    &\le C (1+\tau_\rho) \rho^{\alpha-1-\mathfrak b} K_\nu(\tau_\rho) \int_0^\rho \sigma^{-\mathfrak b} I_{-\nu}(\tau_\sigma) \, d\sigma + C\tau_\rho \rho^{\alpha-1-\mathfrak b} I_{1-\nu}(\tau_\rho) \int_\rho^\infty \sigma^{-\mathfrak b} K_\nu(\tau_\sigma) \, d\sigma\\
    &= C |Y|^{-\theta} (1+\tau_\rho) \tau_\rho^{\theta-\nu-\frac{1}{\mathfrak a}} K_\nu(t) \int_0^{\tau_\rho} u^{\nu+\frac{1}{\mathfrak a}-1}I_{-\nu}(u) \, du + C |Y|^{-\theta} t^{\theta+1-\nu-\frac{1}{\mathfrak a}} I_{1-\nu}(t) \int_t^\infty u^{\nu+\frac{1}{\mathfrak a}-1} K_\nu(u) \, du\\
    &= C |Y|^{-\theta} \left[ II\left( \frac{\mathfrak b - \alpha}{\mathfrak a}, \frac{1-\mathfrak b}{\mathfrak a},-\nu,\tau_\rho \right)  +  II\left( \frac{\mathfrak b - \alpha}{\mathfrak a}+1, \frac{1-\mathfrak b}{\mathfrak a},-\nu,\tau_\rho \right)    + I\left( \frac{\mathfrak b - \alpha}{\mathfrak a}+1, \frac{1-\mathfrak b}{\mathfrak a}, -\nu,1,\tau_\rho \right) \right]\\
    &\le C |Y|^{-\theta}\le C,
\end{align*}
where $\theta:= \frac{1-\alpha}{\mathfrak a}\in (0,1]$.

To close, consider the base derivative $D_{y,j}=i\rho^{\beta}Y_j$. By Lemma~\ref{lem:Bessel:integral} and an argument parallel to the proof of \eqref{eq:prop:Green:even:2}, for $|Y|\ge C^*$,
\begin{align*}
    \int_0^2 |D_{y,j}G_{e,Y}(\rho,\sigma)| \, d\sigma \le C|Y|^{-\theta} \left[II\left(\frac{\beta-\mathfrak b}{\mathfrak a}, \frac{1-\mathfrak b}{\mathfrak a}, -\nu, \tau_\rho \right) + I\left(\frac{\beta-\mathfrak b}{\mathfrak a}, \frac{1-\mathfrak b}{\mathfrak a}, -\nu, 0, \tau_\rho \right) \right]\le C.
\end{align*}
This proves \eqref{eq:prop:Green:even:3} and completes the proof of Proposition~\ref{prop:Green:n=1:even}.
\end{proof}

It follows that

\begin{theorem}\label{thm:Green:n=1:secondtolast}
For every $N>m+3$, the following estimates hold for all $0<\rho\le 1$:
\begin{equation}
    \int_{\mathbb{R}^m} \int_0^2 \sum_{|\gamma|=N} \bigl| \partial_Y^{\gamma} D_* G_{\#,Y}(\rho,\sigma) \bigr| \, d\sigma \, dY \le C_N \kappa_{\#}(\rho),
\end{equation}
and
\begin{equation}
    \int_{\mathbb{R}^m} \int_0^2 |Y| \sum_{|\gamma|=N} \bigl| \partial_Y^{\gamma} D_* G_{\#,Y}(\rho,\sigma) \bigr| \, d\sigma \, dY \le C_N \kappa_{\#}(\rho),
\end{equation}
where $*\in\{r,y,j\}$, $\# \in \{ e,o\}$, and
\begin{equation*}
    \kappa_{\#}(\rho) = \begin{cases}
        1, & \# = e,\\[4pt]
        \rho^{-\alpha}, & \# = o.
    \end{cases}
\end{equation*}
\end{theorem}

\begin{proof}
The proof is the fixed-order analogue of Proposition~\ref{lem:second_to_last}
and Corollary~\ref{cor_weighted_Y_moment_Green}. We record the details needed to make the
reduction explicit.

For each fixed $Y\in\mathbb R^m$, the half-line operators defining
$G_{e,Y}$ and $G_{o,Y}$ depend on $Y$ only through the same potential
$|Y|^2\rho^{2\beta}$. Hence the resolvent identity used in the proof of
Proposition~\ref{lem:second_to_last} remains valid without modification:
\begin{equation*}
    \partial_{Y_j}R_Y = -R_Y\bigl(2Y_j\rho^{2\beta}\bigr)R_Y,
\end{equation*}
and the higher $Y$-derivatives are finite sums of products of resolvents and
multiplication operators of the form $2Y_j\rho^{2\beta}$ or $2\delta_{ij}\rho^{2\beta}$.

For the even kernel, Proposition~\ref{prop:Green:n=1:even} provides the same row bounds as in the higher-dimensional case, uniformly for $0<\rho\leq1$. Therefore, the argument of Proposition~\ref{lem:second_to_last} gives
\begin{equation*}
    \int_{\mathbb R^m}\int_0^2
\sum_{|\gamma|=N}\bigl|\partial_Y^\gamma D_*G_{e,Y}(\rho,\sigma)\bigr|
\, d\sigma \, dY
\leq C_N.
\end{equation*}
The corresponding weighted estimate, with one additional factor $|Y|$, follows by the same argument.

For the odd kernel, Proposition~\ref{prop:Green:n=1} has the same structure, except for the additional factor $\rho^{-\alpha}$ in the intrinsic derivative bounds. Since this factor is independent of $Y$, it passes unchanged through the
resolvent-differentiation argument. Consequently,
\begin{equation*}
    \int_{\mathbb R^m}\int_0^2
\sum_{|\gamma|=N}\bigl|\partial_Y^\gamma D_*G_{o,Y}(\rho,\sigma)\bigr|
\, d\sigma \, dY
\leq C_N\rho^{-\alpha},
\end{equation*}
and the same bound remains valid with one extra factor $|Y|$, provided
$N>m+3$. Combining the even and odd estimates proves both assertions.
\end{proof}

We conclude this section by proving Theorem~\ref{thm:RH2}.

\begin{proof}[Proof of Theorem~\ref{thm:RH2}]
By repeating the arguments used in Theorem~\ref{lem:Poisson_kernel}, Theorem~\ref{lem:Green_kernel}, and the proof of Theorem~\ref{thm:RH}, together with Theorem~\ref{thm:Green:n=1:secondtolast}, the following estimates are obtained: let $n=1$, $m\ge 1$, $\alpha\in (0,\frac{1}{2})$, and $\beta \ge 0$. Then for $0<|x|\le1$ and $y\in E_1$,
\begin{itemize}
    \item $|\nabla_L \mathcal{T}_e h_e(x,y)| + |x|^{\alpha} |\nabla_L \mathcal{T}_o h_o(x,y)| \le C \|h\|_\infty \le C \|v\|_{L^\infty(Q_3)}$,
    \item $|\nabla_L \mathcal{S}_e f_e(x,y)| + |x|^{\alpha} |\nabla_L \mathcal{S}_o f_o(x,y)| \le C \|v\|_{L^\infty(Q_3)}$.
\end{itemize}
Therefore,
\begin{align*}
    \biggl( \fint_{Q_1} |\nabla_L v|^p \, d\xi \biggr)^{\!\frac{1}{p}}
    &= \biggl( \fint_{Q_1} |\nabla_L \omega|^p \, d\xi \biggr)^{\!\frac{1}{p}}\\
    &\le C\biggl( \fint_{Q_1} |\nabla_L \mathcal{T}_e h_e|^p + |\nabla_L \mathcal{T}_o h_o|^p + |\nabla_L \mathcal{S}_e f_e|^p + |\nabla_L \mathcal{S}_o f_o|^p \biggr)^{\!\frac{1}{p}}\\
    &\le C \|v\|_{L^\infty(Q_3)} \biggl( \fint_{Q_1} |x|^{-\alpha p} \, d\xi \biggr)^{\!\frac{1}{p}}\\
    &\le C \biggl(\fint_{Q_4} |v| \, d\xi \biggr) \biggl( \int_0^1 s^{-\alpha p} \, ds \biggr)^{\!\frac{1}{p}}\\
    &\le C \fint_{Q_4} |v| \, d\xi,
\end{align*}
since $p <\alpha^{-1}$. This completes the proof of Theorem~\ref{thm:RH2}.
\end{proof}

\section{The unboundedness of the Riesz transform}

In this section we complete the proof of Theorem~\ref{thm:n=1:alpha<1/2} by showing the unboundedness of the Riesz transform. By Theorem~\ref{thm:CJKS}, the negative result is equivalent to the failure of \eqref{eq:RH:Jiang}. 

\begin{proposition}\label{prop_RH_failure_n1}
Let $n=1$, $m\ge 1$, $\alpha\in (0,\frac{1}{2})$, and $\beta\ge 0$. Then \eqref{eq:RH:Jiang} fails for every $p\ge \alpha^{-1}$.
\end{proposition}

\begin{proof}
Consider
\begin{equation}\label{eq_counterexample_function}
    u(x,y):=\operatorname{sgn}(x)|x|^{1-2\alpha}.
\end{equation}
Since $\alpha \in (0,\frac{1}{2})$, the function $u$ is bounded near $x=0$ and has locally finite energy. Indeed,
\begin{equation*}
\partial_xu(x,y) = (1-2\alpha)|x|^{-2\alpha}, \qquad x\neq0,
\end{equation*}
and therefore
\begin{equation*}
|\nabla_Lu(x,y)| = |x|^\alpha |\partial_x u(x,y)| = (1-2\alpha)|x|^{-\alpha} \in L^2_{\mathrm{loc}}.
\end{equation*}
We first verify that $u$ is weakly $L$-harmonic. For $x\neq0$,
\begin{equation*}
    |x|^{2\alpha}\partial_xu(x,y)
    =
    1-2\alpha.
\end{equation*}
Thus the weighted flux is the same constant on both sides of the singular line. Consequently, for every test function $\varphi\in C_c^\infty$,
\begin{align*}
\int L u \, \varphi \, dx\,dy &= \int |x|^{2\alpha} \partial_x u \, \partial_x \varphi \, dx\,dy\\
&= (1-2\alpha)\int \partial_x\varphi \, dx\,dy = 0.
\end{align*}
Hence $Lu=0$ in the weak sense.

Now let $B$ be any anchored ball centered on the singular line $\{x=0\}$. Such a ball contains a product sub-box of the form $\{|x|<\varepsilon\}\times B_y$ with $|B_y|>0$. It follows that
\begin{align*}
\int_B|\nabla_Lu|^p \, dx\,dy \ge c\int_{-\varepsilon}^{\varepsilon}|x|^{-\alpha p} \, dx.
\end{align*}
The latter integral diverges whenever $p\ge \alpha^{-1}$. On the other hand, $u$ is bounded on $\lambda B$, so
\begin{equation*}
\fint_{\lambda B}|u|<\infty, \qquad \lambda \ge 1.
\end{equation*}
This proves Proposition~\ref{prop_RH_failure_n1}.
\end{proof}

\begin{remark}\label{rem_counterexample_other_regimes}
The counterexample \eqref{eq_counterexample_function} is specific to the case $n=1$ and $0<\alpha<1/2$. Indeed, this construction is no longer available when $n=1$ and $\alpha\ge 1/2$. If
$\alpha>1/2$, the function $|x|^{1-2\alpha}$ is unbounded and
\begin{equation*}
\int_0^1
x^{2\alpha}
\bigl|
\partial_x x^{1-2\alpha}
\bigr|^2 \, dx
=
C\int_0^1x^{-2\alpha}\,dx
=
\infty.
\end{equation*}
At $\alpha=1/2$, the second independent solution is $\log|x|$, which again has
infinite energy. Although $\operatorname{sgn}(x)$ is then locally Friedrichs harmonic,
its form gradient vanishes, so it produces no failure of a reverse H\"older inequality. 
\end{remark}

\part{Reverse Riesz inequality}\label{part3}

Throughout this part, $n,m\ge1$, $0<\alpha<1$, and $\beta\ge0$.
We follow the geometric kernel decomposition of \cite{Gilles}.
The near-diagonal part is estimated locally; harmonic annihilation and
the weighted Hardy inequality control the remaining contribution.

\section{Analysis of the Riesz kernel}

Fix a sufficiently large structural constant $\kappa\ge2^{11}$; its final choice is made in Remark~\ref{remark_delta}. We split the double space $\mathbb{R}^{n+m}\times \mathbb{R}^{n+m}$ into
\begin{align*} 
    \mathcal{D}_1 &= \{(\xi,\eta)\in \mathbb{R}^{n+m}\times \mathbb{R}^{n+m}: |x-x'| \ge \kappa^{-1} |x| \ \text{and}\ |x'| \le \kappa |x| \},\\
    \mathcal{D}_2 &= \{(\xi,\eta)\in \mathbb{R}^{n+m}\times \mathbb{R}^{n+m}: |x-x'| \le \kappa^{-1} |x| \ \text{and}\ |x'| \le \kappa |x| \},\\
    \mathcal{D}_3 &= \{(\xi,\eta)\in \mathbb{R}^{n+m}\times \mathbb{R}^{n+m}: |x'| \ge \kappa |x| \}.
\end{align*}
Note that in $\mathcal{D}_1$, we have $|x-x'| \sim |x|$, and in $\mathcal{D}_3$, $|x|+|x'|\sim |x'|$. Set $\mathcal R:= \nabla_L L^{-1/2}$ and
\begin{align*}
    \mathcal{R}_i(\xi,\eta) := \nabla_L L^{-1/2}(\xi,\eta) \mathbf{1}_{\mathcal{D}_i} = c \int_0^\infty \nabla_L e^{-tL}(\xi,\eta) \mathbf{1}_{\mathcal{D}_i} \, \frac{dt}{\sqrt{t}}.
\end{align*}

\subsection{Estimates of the good part}

In this subsection, we estimate the good part of the Riesz transform, i.e., $\mathcal{R}_1$.

\begin{proposition}\label{R1}
The good part of the Riesz transform is bounded on $L^p$ for all $\frac{n}{n-1+\alpha}<p< \infty$, i.e.,
\begin{equation*}
    \| \mathcal{R}_1f\|_{p}\le C \|f\|_p,\qquad \forall\,
     \frac{n}{n-1+\alpha}<p< \infty,
\end{equation*}
for all $f\in C_c^\infty(\mathbb{R}^{n+m})$.
\end{proposition}

\begin{proof}
See Proposition~\ref{T1,T2} below.
\end{proof}

\begin{lemma}\label{kernel_good}
Let $n,m\ge 1$, $\alpha\in (0,1)$, and $\beta \ge 0$. Then the good part of the Riesz kernel satisfies the estimate
\begin{align}\label{eq_kernel_good}
    |\mathcal{R}_1(\xi,\eta)| \le C\mathbf{1}_{|x'|\le \kappa|x|} \begin{cases}
        |x|^{-\mathcal{Q}(1-\alpha)}, & |x|^{\mathfrak a} \ge |y-y'|,\\[4pt]
        |y-y'|^{-\frac{\mathcal{Q}(1-\alpha)}{\mathfrak a}} + |x|^{-(1-\alpha)} |y-y'|^{-\frac{(\mathcal{Q}-1)(1-\alpha)}{\mathfrak a}}, & |x|^{\mathfrak a} \le |y-y'|,
    \end{cases}
\end{align}
where $\xi=(x,y)$ and $\eta= (x',y')$.
\end{lemma}

\begin{proof}
For $(\xi,\eta)\in \mathcal{D}_1$, one has $\kappa^{-1}|x| \le |x-x'| \le (1+\kappa)|x|$, so $|x|\sim |x-x'|$ and $|x|+|x'|\sim |x|$. Split the integral at the threshold $t = |x|^{2-2\alpha}$. Then
\begin{align*}
    |\mathcal{R}_1(\xi,\eta)|
    &\le \int_0^\infty \bigl| \nabla_L e^{-tL}(\xi,\eta) \bigr| \mathbf{1}_{\mathcal{D}_1} \, \frac{dt}{\sqrt{t}}\\
    &\le C\int_0^{|x|^{2-2\alpha}} \frac{\mathbf{1}_{\mathcal{D}_1}}{V(\xi,\sqrt{t})} \exp\!\Bigl( - \frac{d(\xi,\eta)^2}{c t} \Bigr) \frac{dt}{t}\\
    &\quad + C\int_{|x|^{2-2\alpha}}^\infty \frac{\mathbf{1}_{\mathcal{D}_1}}{|x|^{1-\alpha} V(\xi,\sqrt{t})} \exp\!\Bigl( - \frac{d(\xi,\eta)^2}{c t} \Bigr) \frac{dt}{\sqrt{t}} \\
    &:= I_1 + I_2.
\end{align*}
For $I_1$, since $\sqrt{t}\le |x|^{1-\alpha}$, Lemma~\ref{le_volume} implies that
\begin{equation*}
    V(\xi,\sqrt{t}) \sim t^{\frac{n+m}{2}} |x|^{n\alpha+m\beta}.
\end{equation*}
Hence,
\begin{align*}
    I_1
    &\le C |x|^{-n\alpha - m\beta} d(\xi,\eta)^{-n-m}\mathbf{1}_{\mathcal{D}_1} \int_{d(\xi,\eta)^2/|x|^{2-2\alpha}}^\infty e^{-s} s^{\frac{n+m}{2}-1} \, ds\\
    &\le C |x|^{-n\alpha - m\beta} d(\xi,\eta)^{-n-m}\mathbf{1}_{\mathcal{D}_1} \exp\!\Bigl( - \frac{d(\xi,\eta)^2}{2|x|^{2-2\alpha}} \Bigr)\\
    &\le C |x|^{-n\alpha-m\beta+\varepsilon(1-\alpha)} d(\xi,\eta)^{-n-m-\varepsilon} \mathbf{1}_{\mathcal{D}_1},\quad \forall\, \varepsilon\ge 0\\
    &= C d(\xi,\eta)^{-\mathcal{Q}} \mathbf{1}_{\mathcal{D}_1},
\end{align*}
where in the last step we choose $\varepsilon = \frac{n\alpha+m\beta}{1-\alpha}$.

Next, consider $I_2$. Since $\sqrt{t}\ge |x|^{1-\alpha}$, Lemma~\ref{le_volume} gives
\begin{equation*}
    V(\xi,\sqrt{t}) \sim t^{\frac{\mathcal{Q}}{2}}.
\end{equation*}
Therefore,
\begin{align*}
    I_2
    &\le C |x|^{\alpha-1} d(\xi,\eta)^{-\mathcal{Q}+1} \mathbf{1}_{\mathcal{D}_1} \int_0^{d(\xi,\eta)^2/|x|^{2-2\alpha}} e^{-s} s^{\frac{\mathcal{Q}-1}{2}-1} \, ds\\
    &\le C |x|^{\alpha-1} d(\xi,\eta)^{-\mathcal{Q}+1} \mathbf{1}_{\mathcal{D}_1}.
\end{align*}
Now note that for $(\xi,\eta)\in \mathcal{D}_1$,
\begin{align*}
    d(\xi,\eta)
    &\sim \frac{|x-x'|}{(|x|+|x'|)^\alpha} + \frac{|y-y'|}{(|x|+|x'|)^\beta + |y-y'|^{\frac{\beta}{\mathfrak a}}}\\
    &\sim \begin{cases}
        |x|^{1-\alpha} + |x|^{-\beta} |y-y'|, & |x|^{\mathfrak a} \ge |y-y'|,\\[4pt]
        |x|^{1-\alpha} + |y-y'|^{\frac{1-\alpha}{\mathfrak a}}, & |x|^{\mathfrak a} \le |y-y'|.
    \end{cases}
\end{align*}
Substituting these expressions into the bounds for $I_1$ and $I_2$ yields the stated estimate.
\end{proof}

According to Lemma~\ref{kernel_good}, define the following operators. Note that
\begin{equation*}
    V\bigl((0,y), |x|^{1-\alpha}\bigr) \sim |x|^{\mathcal{Q}(1-\alpha)} \sim |B_n(0,|x|)| \times |B_m(y,|x|^{\mathfrak a})|.
\end{equation*}
Set
\begin{align}\label{def_T1}
    T_1&: u \mapsto \frac{1}{V\bigl( (0,y), |x|^{1-\alpha} \bigr)} \int_{B_n(0,\kappa|x|)} \int_{B_m(y,|x|^{\mathfrak a})} |u(\eta)| \, d\eta,\\[4pt]
    \label{def_T2}
    T_2&: u \mapsto |x|^{-\varepsilon} \int_{B_n(0,\kappa|x|)} \int_{B_m(y, |x|^{\mathfrak a})^c} \frac{|u(\eta)|}{|y-y'|^{\frac{\mathcal{Q}(1-\alpha)-\varepsilon}{\mathfrak a}}} \, d\eta, \qquad \varepsilon = 0,\, 1-\alpha.
\end{align}

It is immediate from Lemma~\ref{kernel_good} that Proposition~\ref{R1} follows once the $L^p$-boundedness of $T_1$ and $T_2$ is established.

\begin{proposition}\label{T1,T2}
Let $n,m\ge 1$, $\alpha\in (0,1)$, and $\beta \ge 0$. Then
\begin{enumerate}[label=(\roman*)]
    \item $T_1$ is bounded on $L^p$ for all $1<p\le \infty$,
    \item $T_2$ is bounded on $L^p$ for all $1<p\le \infty$ if $\varepsilon=0$,
    \item $T_2$ is bounded on $L^p$ for all $\frac{n}{n-1+\alpha}<p\le  \infty$ if $\varepsilon=1-\alpha$.
\end{enumerate}
\end{proposition}

\begin{proof}
For $(i)$, Lemma~\ref{le_volume} implies that
\begin{equation*}
    B_n(0,\kappa|x|) \times B_m\bigl(y,|x|^{\mathfrak a}\bigr) \subset  B\bigl( (0,y), \kappa c_3^{-1}|x|^{1-\alpha} \bigr).
\end{equation*}
The conclusion then follows from \eqref{D} and the maximal theorem.

We next prove $(ii)$ and $(iii)$. Applying a standard dyadic decomposition to the $y'$-space gives
\begin{align*}
    |T_2f(\xi)|
    &\le |x|^{-\varepsilon} \sum_{j=0}^\infty \int_{B_n(0,\kappa |x|)} \int_{B_m( y, 2^{j+1}|x|^{\mathfrak a} )\setminus B_m ( y, 2^{j}|x|^{\mathfrak a} )} |y-y'|^{-\frac{\mathcal{Q}(1-\alpha)-\varepsilon}{\mathfrak a}} |f(\eta)| \, dy' \, dx'\\
    &\le C |x|^{-\varepsilon} \sum_{j=0}^\infty 2^{-j\frac{\mathcal{Q}(1-\alpha)-\varepsilon}{\mathfrak a}} |x|^{-\mathcal{Q}(1-\alpha)+\varepsilon} \int_{B_n(0,\kappa|x|)} \int_{B_m( y, 2^{j+1}|x|^{\mathfrak a} )} |f(\eta)| \, d\eta.
\end{align*}
Let $q>1$. Applying H\"older's inequality twice yields
\begin{equation}
    |T_2f(\xi)|\le C |x|^{-\mathcal{Q}(1-\alpha)} \sum_{j=0}^\infty 2^{-j\frac{\mathcal{Q}(1-\alpha)-\varepsilon}{\mathfrak a} + j\frac{m}{q'}} |x|^{\frac{\mathcal{Q}(1-\alpha)}{q'}} \biggl( \int_{B_n(0,\kappa|x|) \times B_m(y, 2^{j+1}|x|^{\mathfrak a})} |f(\eta)|^q \, d\eta \biggr)^{\!\frac{1}{q}}.
\end{equation}
Furthermore, Lemma~\ref{le_volume} gives
\begin{align*}
    B_n(0,\kappa|x|) \times B_m\bigl(y, 2^{j+1}|x|^{\mathfrak a}\bigr)
    &\subset B_n\bigl(0,\kappa 2^{\frac{j+1}{\mathfrak a}}|x|\bigr) \times B_m\bigl(y, \kappa^{\mathfrak a} 2^{j+1}|x|^{\mathfrak a}\bigr)\\
    &\subset B\bigl((0,y), c_3^{-1} \kappa^{1-\alpha} 2^{\frac{(j+1)(1-\alpha)}{\mathfrak a}} |x|^{1-\alpha}\bigr).
\end{align*}
Hence,
\begin{equation*}
    |T_2f(\xi)|\le C M_q f(\xi) \sum_{j=0}^\infty 2^{-j\bigl( \frac{n}{q'(\mathfrak a)} - \frac{\varepsilon}{\mathfrak a} \bigr)} \le C M_q f(\xi),
\end{equation*}
provided that
\begin{equation*}
    q > \frac{n}{n-\varepsilon} = \begin{cases}
        1, &\varepsilon=0,\\[4pt]
        \frac{n}{n-1+\alpha}, &\varepsilon=1-\alpha.
    \end{cases}
\end{equation*}
The result now follows from the maximal theorem.
\end{proof}

\subsection{Estimates of the diagonal part}

\begin{proposition}\label{R2}
For every $2<p<\infty$, the diagonal part satisfies
\[
 \|\mathcal R_2f\|_p\le C_p\|f\|_p,
 \qquad f\in C_c^\infty((\mathbb R^n\setminus\{0\})\times\mathbb R^m).
\]
\end{proposition}

We use a remote covering, as in \cite[Section~4.3]{Gilles}, and retain
the off-diagonal errors produced by localization.

\begin{lemma}\label{lemma_remotecover}
Let $\varepsilon>0$ and $\delta\ge1$. There is a locally finite cover
$\{B_n^j=B_n(x_j,r_j)\}_{j\ge0}$ of $\mathbb R^n$ such that
$B_n^0=B_n(0,\varepsilon)$ and
\[
 2^{-10-\delta}|x_j|\le r_j\le2^{-9-\delta}|x_j|,
 \qquad j\ge1.
\]
Both $\{B_n^j\}$ and $\{4B_n^j:j\ge1\}$ have bounded overlap, with
constants depending only on $n,\delta$.
\end{lemma}
\begin{proof}
On each annulus
$\{\varepsilon2^{N-1}\le |x|\le\varepsilon2^N\}$, $N\ge1$,
choose a maximal family of disjoint balls of radius
$\varepsilon2^{N-\delta-13}$, with centers in the annulus.
The concentric balls of radius $\varepsilon2^{N-\delta-10}$ cover
the annulus and have the required radii. Euclidean volume comparison
bounds their number by a constant depending only on $n,\delta$.
A point belongs to their fourfold dilates for only a bounded number
of adjacent annuli, which proves the overlap assertion.
An expanded construction is given in \cite[Appendix~A]{HeGrushin}.
\end{proof}

\begin{remark}\label{remark_delta}
Fix $\delta$ sufficiently large that, with
\[
 C_\delta:=2^{-7-\delta}\vee2^{\,2-(9+\delta)\mathfrak a},
 \qquad
 \max\{9,6\gamma\}C_\delta c_1^{-1}<\varepsilon_0,
\]
the enlarged balls used below are remote. We now fix
$\kappa\ge2^{12+\delta}$. All subsequent constants may depend on these
fixed choices, but not on $\varepsilon$ or the covering indices.
\end{remark}

Let $\{\mathcal X_j\}$ be a nonnegative smooth partition of unity
subordinate to the cover, and put $\tau_j=r_j^{2-2\alpha}$. For $j\ge1$ set
\[
 R_jf=\mathbf1_{4B_n^j}\mathcal R(\mathcal X_jf)
     =R_{j,1}f+R_{j,2}f,
\]
where
\[
 R_{j,1}f=\frac{\mathbf1_{4B_n^j}}{\sqrt\pi}
       \int_0^{\tau_j}\nabla_Le^{-tL}(\mathcal X_jf)\,\frac{dt}{\sqrt t},
 \qquad
 R_{j,2}f=\frac{\mathbf1_{4B_n^j}}{\sqrt\pi}
       \int_{\tau_j}^\infty\nabla_Le^{-tL}(\mathcal X_jf)\,\frac{dt}{\sqrt t}.
\]
These are linear operators; absolute values will be taken only when
estimating them.

\begin{proposition}\label{Rj,2}
The operators $R_{j,2}$ are uniformly bounded on $L^p$ for
$1\le p\le\infty$.
\end{proposition}
\begin{proof}
Write $r=r_j$, $s=|y-y'|$, and
$\sigma=s/(r^\beta+s^{\beta/\mathfrak a})$.
On the kernel support, $|x|,|x'|\simeq r$. The volume formula gives
$V(\xi,\sqrt t)\gtrsim t^{\mathcal Q/2}$ for every $t>0$.
Lemma~\ref{Grushin_gradient} therefore yields
\[
 |R_{j,2}(\xi,\eta)|
 \lesssim \mathbf1_{4B_n^j}(x)\mathbf1_{B_n^j}(x')\,
 r^{\alpha-1}\int_{\tau_j}^\infty
 t^{-(\mathcal Q+1)/2}e^{-c\sigma^2/t}\,dt.
\]
The integral is bounded by
$C(\tau_j+\sigma^2)^{-(\mathcal Q-1)/2}$, since $\mathcal Q>1$.
Consequently,
\begin{equation}\label{eq:local-tail-Schur}
 |R_{j,2}(\xi,\eta)|
 \lesssim\mathbf1_{4B_n^j}(x)\mathbf1_{B_n^j}(x')
 \begin{cases}
 r^{-n-m\mathfrak a},&s\le r^{\mathfrak a},\\
 r^{\alpha-1}s^{-m-(n-1+\alpha)/\mathfrak a},&s>r^{\mathfrak a}.
 \end{cases}
\end{equation}
Integration in $y'$ gives $Cr^{-n}$ in both cases, because
$n-1+\alpha>0$. Integration in $x'$ then gives a uniform row integral.
The same computation gives a uniform column integral. Schur's test proves
the assertion, including the endpoints.
\end{proof}

Split the short-time kernel as
$R_{j,1}=R_{j,1}^1+R_{j,1}^2$, according as
$|y-y'|\le r_j^{\mathfrak a}$ or $|y-y'|>r_j^{\mathfrak a}$.

\begin{lemma}\label{Rj,1,2}
The absolute kernels $|R_{j,1}^2(\xi,\eta)|$ define uniformly bounded
operators on $L^p$, $1\le p\le\infty$.
\end{lemma}
\begin{proof}
On this support, the short-time gradient bound and the volume formula give
\[
 |R_{j,1}^2(\xi,\eta)|
 \lesssim\mathbf1_{4B_n^j}(x)\mathbf1_{B_n^j}(x')
 \mathbf1_{\{|y-y'|>r_j^{\mathfrak a}\}}
 |y-y'|^{-m-n/\mathfrak a}.
\]
Indeed, integrate the Gaussian bound up to $\tau_j$, use
$\sigma\simeq|y-y'|^{(1-\alpha)/\mathfrak a}$, and absorb the resulting
exponential with the power $(n\alpha+m\beta)/(1-\alpha)$.
The row and column integrals are uniformly bounded, so Schur's test applies.
\end{proof}

We recall the local criterion of Auscher--Coulhon--Duong--Hofmann.
For a relatively doubling set $E$, let
\[
 M_Ev(\xi):=\sup_{\substack{B\ni\xi\\\operatorname{center}(B)\in E}}
       \frac1{\mu(B\cap E)}\int_{B\cap E}|v|\,d\mu .
\]

\begin{theorem}[{\cite[Theorem~2.4]{ACDH}}]\label{thm_ACDH}
Let $T$ be sublinear and bounded on $L^2(M,\mu)$, and let $A_r$ be
linear operators on $L^2$. Suppose $E_1\subset E_2$, $\mu(E_2)<\infty$,
and $E_2$ has the relative doubling property. Define
\begin{equation}\label{ACDH1}
 M_{E_2}^{\#}f(\xi):=
 \sup_{\substack{B\ni\xi\\\operatorname{center}(B)\in E_2}}
 \left(\frac1{\mu(B\cap E_2)}
       \int_{B\cap E_2}|T(I-A_{r(B)})f|^2\,d\mu\right)^{1/2}.
\end{equation}
Assume that, for functions supported in $E_1$,
\begin{equation}\label{ACDH2}
 \sup_{z\in B\cap E_2}|TA_{r(B)}f(z)|
 \le C\left(M_{E_2}(|Tf|^2)(\xi)^{1/2}+Sf(\xi)\right),
 \qquad \xi\in B\cap E_2,
\end{equation}
where $S$ is sublinear. If $2<p<\infty$, $M_{E_2}^{\#}$ and $S$ are
bounded from $L^p(E_1)$ to $L^p(E_2)$, and $Tf\in L^p(E_2)$ a priori
for $f\in L^p(E_1)\cap L^2$, then $T$ is bounded from $L^p(E_1)$ to
$L^p(E_2)$. The bound depends only on these operator bounds, $p$, the
relative doubling constant, and the constant in \eqref{ACDH2}.
\end{theorem}

\begin{remark}
Relative doubling means
$\mu(B(z,2s)\cap E_2)\le C\mu(B(z,s)\cap E_2)$ for $z\in E_2$.
The norm in Theorem~\ref{thm_ACDH} does not depend on $\mu(E_2)$.
In its application below, the a priori assumption is supplied by a
lower time cutoff, and all estimates are uniform as this cutoff tends to zero.
\end{remark}

\begin{proposition}\label{Rj,1,1}
The operators $R_{j,1}^1$ are uniformly bounded on $L^p$ for
$2<p<\infty$.
\end{proposition}
\begin{proof}
Choose a maximal $r_j^{\mathfrak a}$-separated set of centers $y_k$,
write $B_m^k=B_m(y_k,r_j^{\mathfrak a})$, and choose a subordinate
nonnegative partition $\{\mathcal Y_k\}$. Both the balls and their
fourfold dilates have bounded overlap. Put
$E_1=B_n^j\times B_m^k$ and $E_2=4B_n^j\times4B_m^k$.
For $\xi,\eta\in E_2$, the remote ball--box comparison gives
\[
 d(\xi,\eta)\simeq
 |x_j|^{-\alpha}|x-x'|+|x_j|^{-\beta}|y-y'|.
\]
After rescaling to a fixed product of Euclidean balls, it follows that
\[
 |B(z,s)\cap E_2|\simeq r_j^{n+m\mathfrak a}
       \min\{1,(s/r_j^{1-\alpha})^{n+m}\},\qquad z\in E_2.
\]
Thus $E_2$ is relatively doubling uniformly in $j,k$.

Set $\tau=\tau_j$. For $0<\epsilon<1$, define
\[
 h_\epsilon=\int_{\epsilon\tau}^{\tau}e^{-tL}f\,\frac{dt}{\sqrt t},
 \qquad T_\epsilon f=|\nabla_Lh_\epsilon|,\qquad A_r=e^{-r^2L}.
\]
Spectral calculus gives $\|T_\epsilon\|_{2\to2}\le\sqrt\pi$.
On $E_2\times E_1$, the Gaussian gradient bound and Schur's test give
an a priori $L^p$ bound $C(1+|\log\epsilon|)$.
We verify the hypotheses of Theorem~\ref{thm_ACDH} uniformly in
$\epsilon$, writing $T=T_\epsilon$ and $h=h_\epsilon$.

For a ball $B=B(z_0,r)$ centered in $E_2$, write
$f=f\mathbf1_{2B}+f\mathbf1_{(2B)^c}=f_1+f_2$.
The $L^2$ bound and relative doubling give
\[
 \frac1{|B\cap E_2|}\int_{B\cap E_2}|T(I-A_r)f_1|^2
 \lesssim M_{E_2}(|f|^2)(\xi),\qquad \xi\in B\cap E_2.
\]
For $f_2$, the time weight of $T(I-A_r)$ is
\[
 g_{\epsilon,r}(t)=
 \frac{\mathbf1_{[\epsilon\tau,\tau]}(t)}{\sqrt t}
 -\frac{\mathbf1_{[\epsilon\tau+r^2,\tau+r^2]}(t)}{\sqrt{t-r^2}}.
\]
Since $d(z,\operatorname{supp}f_2)\ge r$ for $z\in B$, Gaussian
integration and doubling imply
\[
 |T(I-A_r)f_2(z)|
 \lesssim Mf(\xi)
 \int_0^\infty |g_{\epsilon,r}(t)|e^{-r^2/(ct)}
                  (t^{-1/2}+\tau^{-1/2})\,dt
 \lesssim Mf(\xi).
\]
Here the integral is bounded uniformly in $\epsilon,r,\tau$.
For the $\tau^{-1/2}$ term this follows from
$\int|g_{\epsilon,r}(t)|\,dt\le4\sqrt\tau$.
For the other term, write $g_{\epsilon,r}$ as the difference of
the weights with intervals $[0,\tau]$ and $[0,\epsilon\tau]$.
For either interval, scaling $t=r^2s$ and splitting at its endpoints
reduces the overlap contribution to
$\int_1^\infty((s-1)^{-1/2}-s^{-1/2})s^{-1/2}\,ds<\infty$;
the two boundary contributions are also uniformly bounded.
Consequently,
\begin{equation}\label{eq:sharp-local-corrected}
 (M_{E_2}^{\#}f)^2\lesssim M_{E_2}(|f|^2)+(Mf)^2.
\end{equation}

For \eqref{ACDH2}, let
$\widetilde B=B((x_j,y_k),C_\delta c_1^{-1}|x_j|^{1-\alpha})$.
Then $E_2\subset\widetilde B$, its radius is comparable to
$r_j^{1-\alpha}$, and $9\widetilde B$ and $6\gamma\widetilde B$ are remote.
Suppose first that $r\le r_j^{1-\alpha}$.
Conservativity \cite[Theorem~6.1]{RS} allows us to subtract $h(\xi)$
inside $e^{-r^2L}$. On $3\widetilde B$, Lemma~\ref{le_tele} gives
\[
 |h(\eta)-h(\xi)|\lesssim d(\xi,\eta)(G(\xi)+G(\eta)),
 \qquad G=M_2(|\nabla_Lh|\mathbf1_{6\gamma\widetilde B}).
\]
The gradient bound, $d(\xi,z)\le2r$, and the Gaussian maximal estimate
then bound this part of $TA_rf(z)$ by
$C(G(\xi)+MG(\xi))\lesssim G(\xi)$; the last step uses
\cite{CR}. Outside $3\widetilde B$, the Gaussian tail bounds the
remaining part by $Cr_j^{\alpha-1}(Mh(\xi)+|h(\xi)|)$.
Furthermore,
\[
 G(\xi)\lesssim
 M_{E_2}(|Tf|^2)(\xi)^{1/2}+M_2(Mf)(\xi).
\]
Indeed, split each averaging ball into its intersection with $E_2$
and its complement. For the first part, recentering the ball at $\xi$
and doubling give the relative maximal function. For the second,
$d(E_1,6\gamma\widetilde B\setminus E_2)\gtrsim r_j^{1-\alpha}$
and the short-time Gaussian bound give
$|\nabla_Lh(\eta)|\lesssim Mf(\eta)$.
This proves \eqref{ACDH2} with
\[
 Sf=C\bigl(M_2(Mf)+r_j^{\alpha-1}Mh+r_j^{\alpha-1}|h|\bigr).
\]
If $r\ge r_j^{1-\alpha}$, the gradient bound directly gives
$TA_rf(z)\lesssim r_j^{\alpha-1}Mh(\xi)$, so the same estimate holds.
The maximal theorem and
$\|h\|_p\le2r_j^{1-\alpha}\|f\|_p$ show that $S$ is uniformly
bounded on $L^p$, $p>2$. Equation~\eqref{eq:sharp-local-corrected}
gives the same conclusion for $M_{E_2}^{\#}$.
Theorem~\ref{thm_ACDH} now yields the required local estimate uniformly
in $\epsilon$. Letting $\epsilon\downarrow0$, using $L^2$ convergence
and Fatou's lemma, proves the estimate for the integral over $(0,\tau)$.

To finish the localization, set
\[
 A_jf:=\sum_k\mathbf1_{4B_m^k}(y)
                         R_{j,1}(\mathcal Y_kf).
\]
On $|y-y'|\le r_j^{\mathfrak a}$, the cutoff sum is one.
Thus $A_j=R_{j,1}^1+F_j$, where the kernel of $F_j$ is dominated
by $|R_{j,1}^2|$. The local estimate and bounded overlap give
$\|A_jf\|_p\lesssim\|\mathcal X_jf\|_p$; Lemma~\ref{Rj,1,2}
controls $F_j$. This proves the proposition.
\end{proof}

\begin{proof}[Proof of Proposition~\ref{R2}]
Take $f$ supported outside $B_n(0,\varepsilon)\times\mathbb R^m$.
Our choice of $\kappa$ ensures that
$(\xi,\eta)\in\mathcal D_2$ and $x'\in B_n^j$ imply $x\in4B_n^j$.
Hence, with
$\Theta(x,x')=\sum_{j\ge1}\mathbf1_{4B_n^j}(x)\mathcal X_j(x')$,
we have the exact identity
\[
 \sum_{j\ge1}R_jf=\mathcal R_2f+Ef,
 \qquad E(\xi,\eta)=\Theta(x,x')\mathcal R_1(\xi,\eta).
\]
On the support of $E$, $|x'|\simeq|x|$, and $0\le\Theta\le1$.
Lemma~\ref{kernel_good}, integrated first in $y'$, gives
\[
 \int_{\mathbb R^m}|E((x,y),(x',y'))|\,dy'
 \lesssim |x|^{-n}\mathbf1_{\{|x'|\simeq|x|\}}.
\]
The far $y'$-integral converges because $n-1+\alpha>0$.
This gives uniform row and column integrals, so $E$ is bounded
on every $L^p$, $1\le p\le\infty$.
Propositions~\ref{Rj,2}, \ref{Rj,1,1} and Lemma~\ref{Rj,1,2},
together with bounded overlap in both the input and output variables, give
\[
 \left\|\sum_{j\ge1}R_jf\right\|_p^p
 \lesssim\sum_{j\ge1}\|\mathcal X_jf\|_p^p
 \lesssim\|f\|_p^p,\qquad p>2.
\]
All constants are independent of $\varepsilon$. Since every test
function in the proposition vanishes near $x=0$, the proof is complete.
\end{proof}

\section{Proof of Theorem~\ref{thm_main2}}

\subsection{Hardy inequality}

We first establish the following weighted Hardy inequality.

\begin{theorem}\label{hardy_x}
Let $L$ be the Grushin-type operator with parameters $n,m\ge 1$, $\alpha \in (0,1)$,
and $\beta \ge 0$, as defined in \eqref{Grushin}, and let $\nabla_L$ be the associated
gradient operator given in \eqref{Grushin2}. Then the following Hardy-type inequality
holds:
\begin{equation}\label{eq_hardy_x}
    \int_{\mathbb R^{n+m}}
    \frac{|u(\xi)|^p}{|x|^{p(1-\alpha)}}\,d\xi
    \leq
    C
    \int_{\mathbb R^{n+m}}
    |\nabla_L u(\xi)|^p \, d\xi,
\end{equation}
for all
\[
    1<p<\frac{n}{1-\alpha}
\]
and all $u\in C_c^\infty(\mathbb R^{n+m})$.
\end{theorem}

\begin{proof}
The proof follows the idea in \cite{ambrosio}. Because of the singular set
$\{x=0\}$, we use an $\varepsilon$-regularization argument.

For $\varepsilon>0$, define
\begin{itemize}
    \item $r_\varepsilon=\bigl(\varepsilon^2+|x|^2\bigr)^{\frac12}$,
    \item $\nabla_L^\varepsilon=\bigl(r_\varepsilon^\alpha\nabla_x,\,
    r_\varepsilon^\beta\nabla_y\bigr)$,
    \item $\mathscr D_\varepsilon(h)
    =
    \operatorname{div}_x(r_\varepsilon^\alpha h_x)
    +
    \operatorname{div}_y(r_\varepsilon^\beta h_y)$, where $h=(h_x,h_y)$ is a sufficiently smooth vector field.
\end{itemize}
Thus $\mathscr D_\varepsilon$ is the positive-sign weighted divergence,
not the adjoint $\operatorname{div}_L$ used elsewhere. If
$\mathscr D_\varepsilon(h)>0$, then H\"older's inequality
(see \cite[Theorem~3.5]{ambrosio}) yields, for any
$u\in C_c^\infty(\mathbb R^{n+m})$,
\begin{equation}\label{divergence_x}
    \int_{\mathbb R^{n+m}}
    |u|^p\mathscr D_\varepsilon(h)\,d\xi
    \leq
    C
    \int_{\mathbb R^{n+m}}
    |h|^p
    \bigl(
        \mathscr D_\varepsilon(h)
    \bigr)^{1-p}
    |\nabla_L^\varepsilon u|^p \, d\xi.
\end{equation}
Choose
\begin{equation}\label{choice_h_x}
    h_\varepsilon(\xi)
    =
    \bigl(
        r_\varepsilon^{\alpha(p-1)-p}x,\,
        0
    \bigr).
\end{equation}
We claim that $\mathscr D_\varepsilon(h_\varepsilon)>0$. Indeed, a direct computation gives
\begin{align*}
    \mathscr D_\varepsilon(h_\varepsilon)
    &=
    \operatorname{div}_x
    \bigl(
        r_\varepsilon^\alpha r_\varepsilon^{\alpha(p-1)-p}x
    \bigr)
    =r_\varepsilon^{-p(1-\alpha)}
    \Bigl[
        n
        -
        p(1-\alpha)\frac{|x|^2}{r_\varepsilon^2}
    \Bigr].
\end{align*}
Since $0\leq \frac{|x|^2}{r_\varepsilon^2}\leq 1$, it follows that
\begin{equation}\label{>0_x}
    n
    -
    p(1-\alpha)\frac{|x|^2}{r_\varepsilon^2}
    \geq
    n-p(1-\alpha)>0,
\end{equation}
where the last inequality follows from the assumption $p<\frac{n}{1-\alpha}$. Consequently,
\begin{equation}\label{div_lower_x}
    \mathscr D_\varepsilon(h_\varepsilon)
    \geq
    C r_\varepsilon^{-p(1-\alpha)}
    >0,
\end{equation}
where $C>0$ is independent of $\varepsilon$.

We next estimate the right-hand side of \eqref{divergence_x}. First,
\[
    |h_\varepsilon|^p
    =
    r_\varepsilon^{p(\alpha(p-1)-p)}|x|^p.
\]
Second, using the lower bound for $\mathscr D_\varepsilon(h_\varepsilon)$, we obtain
\begin{align*}
    |h_\varepsilon|^p
    \bigl(
        \mathscr D_\varepsilon(h_\varepsilon)
    \bigr)^{1-p}
    &\leq
    C
    r_\varepsilon^{p(\alpha(p-1)-p)}
    |x|^p
    \bigl(
        r_\varepsilon^{-p(1-\alpha)}
    \bigr)^{1-p}                                      \\
    &=
    C
    |x|^p
    r_\varepsilon^{p(\alpha(p-1)-p)+p(1-\alpha)(p-1)}.
\end{align*}
A simple calculation shows that
\[
    p(\alpha(p-1)-p)+p(1-\alpha)(p-1)= -p.
\]
Therefore,
\begin{equation}\label{RHS_upper_x}
    |h_\varepsilon|^p
    \bigl(
        \mathscr D_\varepsilon(h_\varepsilon)
    \bigr)^{1-p}
    \leq
    C\frac{|x|^p}{r_\varepsilon^p}
    \leq C,
\end{equation}
where $C>0$ is independent of $\varepsilon$. Combining \eqref{divergence_x}, \eqref{div_lower_x}, and \eqref{RHS_upper_x}, we arrive at
\begin{equation}\label{hardy_x_final}
    \int_{\mathbb R^{n+m}}
    |u|^p r_\varepsilon^{-p(1-\alpha)}\,d\xi
    \leq
    C
    \int_{\mathbb R^{n+m}}
    |\nabla_L^\varepsilon u|^p \, d\xi,
\end{equation}
where $C>0$ is independent of $\varepsilon$.

Letting $\varepsilon\to0$, one has $r_\varepsilon\to |x|$ pointwise away from $\{x=0\}$. Applying Fatou's lemma to the left-hand side of
\eqref{hardy_x_final}, and the dominated convergence theorem to the right-hand side of
\eqref{hardy_x_final}, noting that $u$ has compact support, we obtain
\[
    \int_{\mathbb R^{n+m}}
    \frac{|u(\xi)|^p}{|x|^{p(1-\alpha)}}\,d\xi
    \leq
    C
    \int_{\mathbb R^{n+m}}
    |\nabla_L u(\xi)|^p \, d\xi.
\]
This proves \eqref{eq_hardy_x}.
\end{proof}

\begin{remark}\label{re:Hardy}
For later use, we record the following extension of the Hardy inequality. Assume $n=1$ and $1/2\leq\alpha<1$, and let $u$ be a distribution such that 
\begin{enumerate}[label=(\roman*)] 
    \item $|u(\xi)| \to 0$ as $d(\xi,0) \to \infty$,
    \item $|\nabla_L u| \in L^p(\mathbb{R}^{1+m})$ for some $1<p<\frac{1}{1-\alpha}$.
\end{enumerate}
Then we claim that the Hardy inequality \eqref{eq_hardy_x} also holds for $u$. Indeed, we prove the estimate on $\Omega_+:=(0,\infty)\times\mathbb R^m$; the
argument on $\Omega_-:=(-\infty,0)\times\mathbb R^m$ is identical. Change variables via $r:=x^{1-\alpha}$ and set $U(r,y):= u(r^{(1-\alpha)^{-1}},y)$. For a.e.\ $y$, the function $U(\cdot,y)$ is locally absolutely continuous on
$(0,\infty)$. Moreover, condition (i) implies $U(r,y)\to 0$ as $r\to \infty$.
Since $1<p<(1-\alpha)^{-1}$, the one-dimensional Hardy inequality (see, for example, \cite[Example~1.2.8]{G}) gives
\begin{equation*}
\int_0^\infty |U(r,y)|^p r^{\frac{1}{1-\alpha}-1-p} \, dr \le C_{p,\alpha} \int_0^\infty |\partial_r U(r,y)|^p r^{\frac{1}{1-\alpha}-1} \, dr.
\end{equation*}
Equivalently,
\begin{equation*}
\int_0^\infty \frac{|u(x,y)|^p}{x^{p(1-\alpha)}} \, dx \le C_{p,\alpha} \int_0^\infty x^{\alpha p}|\partial_x u(x,y)|^p \, dx. 
\end{equation*}
Integrating in $y$ yields
\begin{equation*}
\Bigl\| \frac{u}{x^{1-\alpha}} \Bigr\|_{L^p(\Omega_+)} \le C_{p,\alpha} \|x^\alpha\partial_xu\|_{L^p(\Omega_+)} \le C_{p,\alpha} \| \nabla_L u\|_{L^p(\Omega_+)}.
\end{equation*}
The same estimate on $\Omega_-$ gives \eqref{eq_hardy_x}.
\end{remark}

\begin{remark}
It is also worth mentioning that the $L^2$ Hardy and Rellich inequalities for local Dirichlet forms and weighted second-order operators, including certain Grushin-type examples, were studied by Robinson in the quadratic form setting; see \cite{Robinson_Hardy}. However, the form needed below is an $L^p$ inequality for the intrinsic Grushin gradient, valid in the sharp range. 
\end{remark}

\subsection{Harmonic annihilation}
We next prove Theorem~\ref{thm_main2} using the so-called harmonic annihilation method, developed from \cite{H2}, to eliminate the dangerous harmonic leading term arising in the regime $\mathcal{D}_3$. In fact, by \cite{RS}, \eqref{R_p} holds for all $1<p\le 2$, and therefore, by duality, \eqref{RR_p} holds for all $p\ge 2$.

\begin{theorem}\label{prop_RR}
Let $\alpha\in (0,1)$, $\beta \ge 0$, and $n,m\ge 1$. The reverse Riesz inequality associated with $L$ holds for all 
\begin{equation}\label{p3}
   p\in \begin{cases}
        \bigl(1,\frac{1}{1-\alpha}\bigr) \cup \bigl(\frac{1}{1-\alpha},\infty \bigr), & n=1 \text{ and } \alpha \in \bigl(0,\frac{1}{2}\bigr),\\[4pt]
        (1,\infty), & n\ge 2 \text{ or } n=1 \text{ and } \alpha\in \bigl[\frac{1}{2},1\bigr),
    \end{cases}
\end{equation}
and all $f\in C_c^\infty(\mathbb{R}^{n+m})$.
\end{theorem}

\begin{proof}
By Theorem~\ref{thm_main1} and duality, \eqref{RR_p} holds for all
$1<p<\infty$ when $n\ge 2$, so assume from now on that $n=1$.
Theorem~\ref{thm:n=1:alpha<1/2} and duality give \eqref{RR_p} for
$p>(1-\alpha)^{-1}$ when $0<\alpha<1/2$.  Moreover,
\cite[Theorem~8.1]{RS} yields \eqref{R_p} for $1<p\le 2$, which in turn
implies \eqref{RR_p} for $p\ge 2$. Set
$p_0:=\min\{2,(1-\alpha)^{-1}\}$. It remains to prove
\begin{equation*}
    \| L^{1/2}f\|_p \le C \| \nabla_L f\|_p,\qquad 1<p<p_0.
\end{equation*}

Let $\vartheta_\varepsilon$ ($\varepsilon>0$) be a smooth function on $\mathbb{R}^n$ such that
\begin{itemize}
    \item $\operatorname{supp}(\vartheta_\varepsilon) \subset B_n(0,4 \varepsilon/3)$,
    \item $\vartheta_\varepsilon=1$ on $B_n(0,\varepsilon)$,
    \item $\|\vartheta_\varepsilon\|_\infty + \varepsilon \|\nabla \vartheta_\varepsilon\|_\infty \lesssim 1$.
\end{itemize}
Then $\vartheta_{\kappa^{-1} |x'|}(x)$ is a smooth function supported in the region
\begin{equation*}
    \{(x,x')\in \mathbb{R}^n \times \mathbb{R}^n : 4|x'|\ge 3\kappa |x|\}.
\end{equation*}
We decompose the Riesz kernel as
\begin{align}\label{eq_Grushin_1}
    \mathcal{R}(\xi,\eta) = \int_0^\infty \vartheta_{\kappa^{-1} |x'|}(x)& \nabla_L e^{-tL}(\xi,\eta) \, \frac{dt}{\sqrt{\pi t}} \nonumber\\
    &+ \int_0^\infty \bigl[1 - \vartheta_{\kappa^{-1} |x'|}(x) \bigr] \nabla_L e^{-tL}(\xi,\eta) \, \frac{dt}{\sqrt{\pi t}}
    := \mathcal{S}_1(\xi,\eta) + \mathcal{S}_2(\xi,\eta).
\end{align}
The cutoff is zero on $\mathcal D_2$ and one on $\mathcal D_3$.
Consequently,
\[
 \mathcal S_2=\mathcal R_2+\mathcal E_1,\qquad
 \mathcal E_1(\xi,\eta)
 =(1-\vartheta_{\kappa^{-1}|x'|}(x))\mathcal R_1(\xi,\eta).
\]
The absolute kernel of $\mathcal E_1$ is controlled by the positive
majorants in Lemma~\ref{kernel_good} and Proposition~\ref{T1,T2}.
Thus $\mathcal S_2$ is bounded on $L^{p'}$ whenever
$p'>\max\{2,n/(n-1+\alpha)\}$, by Proposition~\ref{R2}.

Let $1<p<p_0$, let $f\in C_c^\infty(\mathbb{R}^{1+m})$, and let $g\in C_c^\infty((\mathbb{R}\setminus \{0\}) \times \mathbb{R}^m)$ satisfy $\|g\|_{p'}=1$. Since this test space is dense in $L^{p'}(\mathbb{R}^{1+m})$, duality reduces the problem to proving
\begin{align*}
    \bigl| \langle L^{1/2}f, g \rangle \bigr| \le C \|\nabla_L f\|_p \|g\|_{p'},\qquad 1<p<p_0,
\end{align*}
for some constant $C>0$ independent of $g$.

We use the identity
\begin{equation*}
    L^{1/2}f
    =
    \frac{1}{\sqrt{\pi}}
    \int_0^\infty
    Le^{-tL}f\,\frac{dt}{\sqrt t}.
\end{equation*}
Strictly speaking, this formula is understood in the truncated sense: for
$0<t_1<t_2<\infty$, the integral on the right-hand side is interpreted as
$\lim_{\substack{t_1\to 0\\ t_2 \to \infty}} \int_{t_1}^{t_2}$; see \cite{AC} for details. However, all estimates below are uniform in $t_1$ and $t_2$. For simplicity of notation, we continue to write the integral over $(0,\infty)$.

Next, by positivity and self-adjointness of $L$,
\begin{align*}
    \langle L^{1/2}f, g \rangle
    = \biggl\langle \int_0^\infty L e^{-tL}f \, \frac{dt}{\sqrt{\pi t}}, g \biggr\rangle
    &= \bigl\langle \nabla_L f, \int_0^\infty \nabla_L e^{-tL}g \, \frac{dt}{\sqrt{\pi t}} \bigr\rangle\\
    &= \langle \nabla_L f, \mathcal{S}_1g \rangle + \langle \nabla_L f, \mathcal{S}_2 g \rangle.
\end{align*}
Since $n=1$ and $1<p<p_0$ imply $p'>\max\{2,1/\alpha\}$, the preceding bound gives
\begin{align}\label{S2}
|\langle \nabla_L f, \mathcal{S}_2 g \rangle|
\le \|\nabla_L f\|_p \|\mathcal{S}_2g\|_{p'}
\le C \|\nabla_L f\|_p \|g\|_{p'}.
\end{align}
Consider the bilinear form
\begin{equation*}
    \mathcal{B}(f,g) := \int_{\mathbb{R}^{n+m}} \nabla_L f(\xi) \cdot \int_{\mathbb{R}^{n+m}} \int_0^\infty \vartheta_{\kappa^{-1} |x'|}(x) \nabla_L e^{-tL}(\xi,\eta) g(\eta)\,\frac{dt}{\sqrt{\pi t}}\,d\eta\,d\xi.
\end{equation*}
Integration by parts gives
\begin{align}\label{eq:HA}
    \mathcal{B}(f,g)
    &= \int_{\mathbb{R}^{n+m}} f(\xi) \int_{\mathbb{R}^{n+m}} \int_0^\infty \vartheta_{\kappa^{-1} |x'|}(x) L_{\xi} e^{-tL}(\xi,\eta) g(\eta) \, \frac{dt}{\sqrt{\pi t}} \, d\eta \, d\xi \nonumber\\
    &\quad - \int_{\mathbb{R}^{n+m}} f(\xi) \int_{\mathbb{R}^{n+m}} \int_0^\infty (\nabla_L)_x [\vartheta_{\kappa^{-1} |x'|}(x)] \cdot \nabla_L e^{-tL}(\xi,\eta) g(\eta) \, \frac{dt}{\sqrt{\pi t}} \, d\eta \, d\xi.
\end{align}
In view of \eqref{S2}, it is enough to prove
\begin{equation*}
    |\mathcal{B}(f,g)| \le C \|\nabla_L f\|_p \|g\|_{p'},\qquad 1<p<p_0.
\end{equation*}

\begin{lemma}\label{lemma_key2}
Under the assumptions of Theorem~\ref{prop_RR}, the following estimates hold:
\begin{align*}
    &\int_0^\infty  \bigl|(\nabla_L)_x [\vartheta_{\kappa^{-1} |x'|}(x)] \cdot \nabla_L e^{-tL}(\xi,\eta)\bigr| + \bigl|\vartheta_{\kappa^{-1} |x'|}(x) \partial_t e^{-tL}(\xi,\eta)\bigr| \, \frac{dt}{\sqrt{t}} \\
    &\qquad \lesssim 
    \begin{cases}
        |x'|^{-(1-\alpha)(\mathcal{Q}+1)} \mathbf{1}_{|x|\le c_1|x'|}, & |y-y'|\le |x'|^{\mathfrak a},\\[4pt]
        |x'|^{-2+2\alpha} |y-y'|^{-\frac{(1-\alpha)(\mathcal{Q}-1)}{\mathfrak a}} \mathbf{1}_{c_2|x|\le |x'| \le c_3|x|} + |y-y'|^{-\frac{(1-\alpha)(\mathcal{Q}+1)}{\mathfrak a}}\mathbf{1}_{|x|\le c_4|x'|}, & |y-y'|\ge |x'|^{\mathfrak a},
    \end{cases}
\end{align*}
for some constants $c_1,c_2,c_3,c_4>0$.
\end{lemma}

\begin{proof}
Note that if $|x'|\ge \frac{3\kappa}{4} |x|$, then $|x-x'|\sim |x'|$ provided $\kappa$ is sufficiently large. Hence, by Theorem~\ref{RS_Grushin},
\begin{align*}
    d(\xi,\eta) \sim |x'|^{1-\alpha} + \frac{|y-y'|}{|x'|^\beta + |y-y'|^{\frac{\beta}{\mathfrak a}}} := \sigma.
\end{align*}
Clearly,
\begin{align}\label{eq_delta}
    \sigma \sim \begin{cases}
        |x'|^{1-\alpha}, & |y-y'|\le |x'|^{\mathfrak a},\\[4pt]
        |y-y'|^{\frac{1-\alpha}{\mathfrak a}}, & |y-y'|\ge |x'|^{\mathfrak a}.
    \end{cases}
\end{align}
Moreover, for $x$ in the support of $\nabla_L \vartheta_{\kappa^{-1} |x'|}$, one has $|x|\sim |x'|$. In addition,
\begin{equation*}
    |\nabla_L \vartheta_{\kappa^{-1} |x'|}|\lesssim |x|^\alpha |x'|^{-1} \mathbf{1}_{|x|\sim |x'|}.
\end{equation*}
By Theorem~\ref{RS_Grushin} and Lemma~\ref{Grushin_gradient},
\begin{align*}
    \int_0^{|x|^{2-2\alpha}} \bigl|\nabla_L [\vartheta_{\kappa^{-1} |x'|}(x)]\bigr| \, \bigl|\nabla_L e^{-tL}(\xi,\eta)\bigr| \, \frac{dt}{\sqrt{t}}
    &\lesssim \int_0^{|x|^{2-2\alpha}} \frac{\mathbf{1}_{|x|\sim |x'|}}{|x'|^{1-\alpha}} \frac{e^{-\frac{\sigma^2}{c t}}}{|x|^{n\alpha+m\beta}t^{\frac{n+m}{2}}} \frac{dt}{t}\\
    &\lesssim |x'|^{\alpha-1-n\alpha-m\beta} \mathbf{1}_{|x|\sim |x'|} \int_0^{|x'|^{2-2\alpha}} t^{-\frac{n+m}{2}-1} e^{-\frac{\sigma^2}{c t}} \, dt\\
    &\lesssim |x'|^{\alpha-1-n\alpha-m\beta} \sigma^{-n-m} e^{-\frac{\sigma^2}{c|x'|^{2-2\alpha}}} \mathbf{1}_{|x|\sim |x'|}\\
    &\lesssim_\varepsilon |x'|^{\alpha-1-n\alpha -m\beta+(1-\alpha)\varepsilon} \sigma^{-n-m-\varepsilon} \mathbf{1}_{|x|\sim |x'|},\quad \forall\, \varepsilon \ge 0\\
    &\lesssim \sigma^{-\mathcal{Q}-1} \mathbf{1}_{|x|\sim|x'|},
\end{align*}
where we choose $\varepsilon = \mathcal{Q}+1-n-m>0$.

On the other hand,
\begin{align*}
    \int_{|x|^{2-2\alpha}}^\infty \bigl|\nabla_L [\vartheta_{\kappa^{-1} |x'|}(x)]\bigr| \, \bigl|\nabla_L e^{-tL}(\xi,\eta)\bigr| \, \frac{dt}{\sqrt{t}}
    &\lesssim \int_{|x|^{2-2\alpha}}^\infty \frac{\mathbf{1}_{|x|\sim |x'|}}{|x'|^{1-\alpha}} \frac{e^{-\frac{\sigma^2}{c t}}}{t^{\frac{\mathcal{Q}}{2}}} \frac{dt}{|x|^{1-\alpha}\sqrt{t}}\\
    &\lesssim |x'|^{-2+2\alpha} \mathbf{1}_{|x|\sim |x'|} \int_{|x'|^{2-2\alpha}}^\infty e^{-\frac{\sigma^2}{c t}} t^{-\frac{\mathcal{Q}+1}{2}} \, dt\\
    &\lesssim |x'|^{-2+2\alpha} \sigma^{-\mathcal{Q}+1} \mathbf{1}_{|x|\sim|x'|}.
\end{align*}
Similarly, using the time-derivative estimate for the heat kernel, namely
\begin{equation*}
    \bigl|\partial_t e^{-tL}(\xi,\eta) \bigr| \le \frac{C}{t V(\eta,\sqrt{t})} e^{-\frac{d(\xi,\eta)^2}{c t}}, \qquad \forall\, \xi,\eta\in \mathbb{R}^{n+m},\quad \forall\, t>0,
\end{equation*}
which is a standard consequence of the analyticity of the heat semigroup; see \cite{D1989} or \cite{G1}, we obtain
\begin{align*}
    \int_0^\infty \bigl|\vartheta_{\kappa^{-1} |x'|}(x)\bigr| \, &\bigl|L_{\xi} e^{-tL}(\xi,\eta) \bigr| \, \frac{dt}{\sqrt{t}} \\
    &\lesssim  \int_0^\infty \frac{t^{-\frac{3}{2}} \mathbf{1}_{|x|\le c|x'|}}{V(\eta,\sqrt{t})} e^{-\frac{d(\xi,\eta)^2}{c t}} \, dt\\
    &\lesssim \int_0^{|x'|^{2-2\alpha}} \frac{ \mathbf{1}_{|x|\le c|x'|}}{|x'|^{n\alpha+m\beta}(\sqrt{t})^{n+m}} e^{-\frac{\sigma^2}{c t}} \frac{dt}{t^{3/2}}\\
    &\quad + \int_{|x'|^{2-2\alpha}}^\infty \frac{ \mathbf{1}_{|x|\le c|x'|}}{(\sqrt{t})^{\mathcal{Q}}} e^{-\frac{\sigma^2}{c t}} \frac{dt}{t^{3/2}}\\
    &\lesssim |x'|^{-n\alpha-m\beta+(1-\alpha)\varepsilon} \sigma^{-n-m-1-\varepsilon} \mathbf{1}_{|x|\le c|x'|} + \sigma^{-\mathcal{Q}-1} \mathbf{1}_{|x|\le c|x'|},\quad \forall\, \varepsilon\ge 0\\
    &\lesssim \sigma^{-\mathcal{Q}-1} \mathbf{1}_{|x|\le c|x'|},
\end{align*}
where we choose $\varepsilon = \mathcal{Q}-n-m \ge 0$. The result follows from \eqref{eq_delta}.
\end{proof}

With Lemma~\ref{lemma_key2} in mind, define the operators
\begin{gather*}
    T_3: u \mapsto |x|^{1-\alpha} \int_{B_n(0,c|x|)^c} |x'|^{-(1-\alpha)(\mathcal{Q}+1)} \int_{B_m(y,|x'|^{\mathfrak a})} |u(x',y')| \, dy' \, dx',\\
    T_4: u \mapsto |x|^{1-\alpha} \int_{B_n(0,c|x|)^c} \int_{B_m(y,|x'|^{\mathfrak a})^c} |y-y'|^{-\frac{(1-\alpha)(\mathcal{Q}+1)}{\mathfrak a}}|u(x',y')| \, dy' \, dx',\\
    T_5: u \mapsto |x|^{1-\alpha} \int_{c_1|x|\le |x'|\le c_2|x|} |x'|^{-2+2\alpha} \int_{B_m(y,|x'|^{\mathfrak a})^c} |y-y'|^{-\frac{(1-\alpha)(\mathcal{Q}-1)}{\mathfrak a}}|u(x',y')| \, dy' \, dx'.
\end{gather*}

\begin{lemma}\label{T3T4}
$T_3$ and $T_4$ are bounded on $L^p$ for all $1<p\le \infty$.
\end{lemma}

\begin{proof}
The proof is similar to that of Proposition~\ref{T1,T2}. By dyadic decomposition,
\begin{equation*}
    |T_3 f(\xi)|\le |x|^{1-\alpha} \sum_{j\ge 0} \int_{B_n(0,2^{j+1}\kappa |x|) \setminus B_n(0, 2^j \kappa |x|)} \bigl(2^j |x| \bigr)^{-(1-\alpha)(\mathcal{Q}+1)} \int_{B_m(y,|x'|^{\mathfrak a})} |f(\eta)| \, d\eta.
\end{equation*}
The $j$th term is bounded by
\begin{align*}
    2^{-j(1-\alpha)(\mathcal{Q}+1)} |x|^{-(1-\alpha)(\mathcal{Q}+1)} \int_{ B_n(0, 2^{j+1}\kappa |x|) \times B_m (y, (2^{j+1}\kappa |x|)^{\mathfrak a}) } |f(\eta)| \, d\eta.
\end{align*}
By Lemma~\ref{le_volume},
\begin{equation*}
    B_n\bigl(0, 2^{j+1}\kappa |x|\bigr) \times B_m \bigl(y, (2^{j+1}\kappa |x|)^{\mathfrak a}\bigr) \subset  B\bigl( (0,y), c_3^{-1} 2^{(j+1)(1-\alpha)} \kappa^{1-\alpha} |x|^{1-\alpha} \bigr).
\end{equation*}
Therefore the $j$th term is bounded by
\begin{equation*}
    2^{-j(1-\alpha)} |x|^{\alpha-1} Mf(\xi),
\end{equation*}
and the $L^p$ estimate for $T_3$ follows by summing over $j\ge 0$ and applying the maximal theorem.

Next, for $T_4$, applying dyadic decomposition twice shows that $|T_4 f(\xi)|$ is bounded by
\begin{equation*}
    |x|^{1-\alpha} \sum_{i\ge 0} \sum_{j\ge 0} \int_{B_n(0,2^{i+1}\kappa |x|) \setminus B_n(0, 2^i \kappa |x|)} \int_{B_m(y,2^{j+1}|x'|^{\mathfrak a}) \setminus B_m(y, 2^j |x'|^{\mathfrak a})} \frac{|f(\eta)|\,d\eta}{|y-y'|^{\frac{(\mathcal{Q}+1)(1-\alpha)}{\mathfrak a}}}.
\end{equation*}
For each $i,j\ge 0$, one has
\begin{equation*}
    B_n\bigl(0,2^{i+1}\kappa |x|\bigr) \times B_m\bigl(y,2^{j+1}|x'|^{\mathfrak a}\bigr)  \subset B \bigl( (0,y), c_3^{-1} 2^{\frac{(j+1)(1-\alpha)}{\mathfrak a} + (i+1)(1-\alpha)} |x|^{1-\alpha} \bigr),
\end{equation*}
and an argument parallel to that for $T_3$ yields the bound
\begin{equation*}
    2^{-i(1-\alpha)} 2^{-j \frac{1-\alpha}{\mathfrak a}} Mf(\xi).
\end{equation*}
The conclusion follows at once.
\end{proof}

By Lemma~\ref{T3T4}, it remains to treat $T_5$. Since $|x|\sim |x'|$, one has $T_5 f(\xi) \lesssim T_2f(\xi)$ with $\varepsilon = 1-\alpha$; see \eqref{def_T2}. Hence, by Proposition~\ref{T1,T2}, $T_5$ is bounded on $L^q$ for $q> \frac{n}{n-1+\alpha} = \Bigl(\frac{n}{1-\alpha}\Bigr)'$. Therefore, by Lemma~\ref{lemma_key2} and Theorem~\ref{hardy_x}, the bilinear form satisfies
\begin{align}\label{eq_final}
    |\mathcal{B}(f,g)|
    &\le C \int_{\mathbb{R}^{n+m}} \frac{|f(\xi)|}{|x|^{1-\alpha}} \, \bigl|(T_3+ T_4 + T_5)g(\xi)\bigr| \, d\xi \nonumber\\
    &\le C \Bigl\| \frac{f(\xi)}{|x|^{1-\alpha}} \Bigr\|_p \sum_{j=3}^5 \|T_j g\|_{p'} \nonumber\\
    &\le C \| \nabla_L f\|_p \|g\|_{p'}
\end{align}
for all $1<p<p_0$, as required. Together with the ranges established at
the beginning of the proof, this completes the proof of
Theorem~\ref{prop_RR}.
\end{proof}

\begin{remark}\label{re:harmonic:annihilation}
For later use, consider $n=1$ and $\frac12\le\alpha<1$.  Our convention is
\begin{equation*}
\operatorname{div}_L F
=-\partial_x(|x|^\alpha F_x)
 -\operatorname{div}_y(|x|^\beta F_y),
\qquad F=(F_x,F_y),
\end{equation*}
so that $L=\operatorname{div}_L\nabla_L$.  Let
$u=L^{-1}\operatorname{div}_L\omega$, with
$\omega\in C_c^\infty(\mathbb R^{1+m};\mathbb C^{1+m})$, and let
$g\in C_c^\infty((\mathbb R\setminus\{0\})\times\mathbb R^m)$.

The integration by parts in \eqref{eq:HA} is most cleanly justified at the
level of the Friedrichs form, without assigning pointwise boundary values
at $x=0$.  For fixed $t>0$ and $\eta=(x',y')$, set
$w(\xi)=e^{-tL}(\xi,\eta)$ and
$\vartheta(\xi)=\vartheta_{\kappa^{-1}|x'|}(x)$.  The separated
Friedrichs form is the orthogonal sum of the forms on the two half-spaces.
Thus $\vartheta u$ is an admissible form test function on each component,
even when the two traces of $u$ differ.  Semigroup smoothing and the form
identity give
\begin{equation*}
 \mathcal q(w,\vartheta u)=\langle Lw,\vartheta u\rangle.
\end{equation*}
If necessary, this identity is first tested after inserting compact
cutoffs in $y$ and at large intrinsic radius; the heat-kernel bounds and
the assumed decay of $u$ permit the cutoffs to be removed.
Expanding the left-hand side by the form product rule yields
\begin{equation*}
 \int \vartheta\,\nabla_Lu\cdot\nabla_Lw
 =\int u\,\vartheta Lw
  -\int u\,\nabla_L\vartheta\cdot\nabla_Lw.
\end{equation*}
There is no extra interface term: the zero-flux condition is already built
into the direct-sum Friedrichs form.  Integrating this identity against
$g(\eta)\,d\eta\,dt/\sqrt{\pi t}$ on
$t_1<t<t_2$ gives the truncated version of \eqref{eq:HA}.  The kernel
estimates used in the proof of Theorem~\ref{prop_RR} are uniform in
$t_1,t_2$; hence letting $t_1\downarrow0$ and $t_2\uparrow\infty$ proves
\eqref{eq:HA} for $u$.
\end{remark}

\begin{remark}\label{re:applicable:HA}
More generally, this harmonic annihilation method (Theorem~\ref{prop_RR}) applies to a broader class of functions. In fact, the argument of Theorem~\ref{prop_RR} works verbatim for a function $u$ if
\begin{itemize}
    \item \eqref{eq:HA} holds for $u$,
    \item the Hardy inequality \eqref{eq_hardy_x} holds for $u$.
\end{itemize}
In particular, by Remarks~\ref{re:Hardy} and \ref{re:harmonic:annihilation} above, in the case $n=1$ and $1/2\le \alpha<1$, the reverse Riesz inequality \eqref{RR_p} holds for $u=L^{-1} \operatorname{div}_L \omega$ with $\omega \in C_c^\infty(\mathbb{R}^{1+m}; \mathbb{C}^{1+m})$ provided $u(\xi) \to 0$ as $d(\xi,0)\to \infty$ and $|\nabla_L u|\in L^p$.
\end{remark}

\subsection{Failure at the critical exponent}
\label{ssec:critical-RR}

We complete the proof of Theorem~\ref{thm_main2} by treating the one
remaining exponent.

\begin{theorem}[Critical failure of the reverse Riesz inequality]\label{thm:critical-RR-failure}
Suppose that $n=1$, $m\ge1$, $0<\alpha<\frac12$, and $\beta\ge0$.  Then
\eqref{RR_p} fails at
\begin{equation*}
 p=p_c:=\frac1{1-\alpha}.
\end{equation*}
\end{theorem}

\begin{proof}
Let $\Omega_+=(0,\infty)\times\mathbb R^m$, and let $L_D$ be the
Dirichlet realization of $L$ on $\Omega_+$.  If $A$ is either $L$ or
$L_D$ and $h\in\mathcal D(A^{1/2})$, the spectral theorem gives
\begin{equation}\label{eq:critical-semigroup-square-root}
 A^{1/2}h
 =\frac1{\sqrt\pi}\int_0^\infty Ae^{-tA}h\,\frac{dt}{\sqrt t}
 =\frac1{2\sqrt\pi}\int_0^\infty(h-e^{-tA}h)\,\frac{dt}{t^{3/2}}.
\end{equation}
Both integrals are limits of truncated integrals in $L^2$; the second
identity follows by integration by parts.

We also use odd reflection.  If $g\in C_c^\infty(\Omega_+)$ and
\begin{equation*}
 \widetilde g(x,y)=\operatorname{sgn}(x)g(|x|,y),
\end{equation*}
then invariance under $x\mapsto-x$, uniqueness of the form heat equation,
and the zero trace of odd form functions give
\begin{equation}\label{eq:odd-Dirichlet-reflection}
 e^{-tL}\widetilde g(x,y)
 =\operatorname{sgn}(x)(e^{-tL_D}g)(|x|,y).
\end{equation}
This is the form version of the Dirichlet reflection principle; see, for
example, \cite{Ouhabaz,RS_1D}.  Applying
\eqref{eq:critical-semigroup-square-root} to
\eqref{eq:odd-Dirichlet-reflection} yields
\begin{equation}\label{eq:odd-square-root-reflection}
 L^{1/2}\widetilde g(x,y)
 =\operatorname{sgn}(x)(L_D^{1/2}g)(|x|,y).
\end{equation}

Write $T_D^t=e^{-tL_D}$.  Since this semigroup is sub-Markovian,
$T_D^t1$ is well defined through its positive-contraction extension to
$L^\infty$.  The function $1$ is independent of $y$, and uniqueness of
bounded weak solutions gives
\begin{equation*}
 T_D^t1(x,y)=e^{-tA_D}1(x),\qquad
 A_D=-\partial_x(x^{2\alpha}\partial_x)
\end{equation*}
with the Dirichlet condition at $0$.  Under
\begin{equation*}
 \rho=\frac{x^{1-\alpha}}{1-\alpha},\qquad
 \nu=\frac{1-2\alpha}{2(1-\alpha)}\in(0,\tfrac12),
\end{equation*}
$A_D$ becomes the Dirichlet Bessel operator
\begin{equation*}
 -\partial_\rho^2-\frac{\alpha}{1-\alpha}\frac1\rho\partial_\rho
\end{equation*}
on $L^2((0,\infty),\rho^{1-2\nu}\,d\rho)$.  Its heat kernel is
\begin{equation}\label{eq:critical-Bessel-heat-kernel}
 p_t^D(\rho,s)
 =\frac{(\rho s)^\nu}{2t}
   \exp\left(-\frac{\rho^2+s^2}{4t}\right)
   I_\nu\left(\frac{\rho s}{2t}\right),\qquad \rho,s>0.
\end{equation}
The formula is checked directly from the Bessel equation, the Dirichlet
behavior at zero, and the approximate-identity normalization.  Integrating
in $s$ against $s^{1-2\nu}\,ds$ gives
\begin{equation}\label{eq:critical-survival}
 T_D^t1(x,y)
 =\frac{\gamma(\nu,\rho^2/(4t))}{\Gamma(\nu)},
\end{equation}
where $\gamma$ is the lower incomplete gamma function.  Consequently,
Tonelli's theorem and the substitution $s=\rho^2/(4t)$ give
\begin{align}
 \frac1{2\sqrt\pi}\int_0^\infty(1-T_D^t1(x,y))\,\frac{dt}{t^{3/2}}
 &=\frac1{2\sqrt\pi\,\Gamma(\nu)}
   \int_0^\infty\Gamma\left(\nu,\frac{\rho^2}{4t}\right)
                \frac{dt}{t^{3/2}}\nonumber\\
 &=\frac{\Gamma(\nu+\frac12)}{\sqrt\pi\,\Gamma(\nu)}
   \frac2\rho
 =c_\alpha x^{-(1-\alpha)},\label{eq:critical-barrier}
\end{align}
where
\begin{equation*}
 c_\alpha
 =\frac{2(1-\alpha)\Gamma(\nu+\frac12)}
        {\sqrt\pi\,\Gamma(\nu)}>0.
\end{equation*}

The barrier gives a lower bound on plateaux.  Indeed, let
$F\in C_c^\infty(\Omega_+)$ satisfy $0\le F\le1$ and $F=1$ on a measurable
set $E$.  Positivity and sub-Markovianity imply, for a.e. $z\in E$,
\begin{equation*}
 F(z)-T_D^tF(z)\ge1-T_D^t1(z)\ge0.
\end{equation*}
Apply the truncated form of
\eqref{eq:critical-semigroup-square-root} on
$[\varepsilon,\varepsilon^{-1}]$ and pass to an a.e. convergent
subsequence as $\varepsilon\downarrow0$.  By
\eqref{eq:critical-barrier},
\begin{equation}\label{eq:critical-plateau}
 L_D^{1/2}F(x,y)\ge c_\alpha x^{-(1-\alpha)}
 \qquad\text{for a.e. }(x,y)\in E.
\end{equation}

Choose $\chi\in C^\infty([0,\infty))$ with $0\le\chi\le1$,
$\chi=0$ on $[0,1]$, and $\chi=1$ on $[2,\infty)$.  Choose also
$\eta\in C_c^\infty(\mathbb R^m)$ with $0\le\eta\le1$ and
$\eta=1$ on $B_m(0,1)$.  For $R\ge8$, put
\begin{equation*}
 u_R(x)=\chi(x)\chi(R/x),\qquad
 F_R(x,y)=u_R(x)\eta(y/R^{\mathfrak a}).
\end{equation*}
Then $F_R\in C_c^\infty(\Omega_+)$ and
\begin{equation*}
 F_R=1\quad\text{on}\quad
 [2,R/2]\times B_m(0,R^{\mathfrak a}),
\end{equation*}
while
\begin{equation*}
 |u_R'(x)|\le C\mathbf1_{[1,2]}(x)
 +CR^{-1}\mathbf1_{[R/2,R]}(x).
\end{equation*}
Since $p_c(1-\alpha)=1$, scaling gives
\begin{align*}
 \|x^\alpha\partial_xF_R\|_{p_c}^{p_c}&\le CR^{m\mathfrak a},\\
 \|x^\beta\nabla_yF_R\|_{p_c}^{p_c}&\le CR^{m\mathfrak a}.
\end{align*}
Thus
\begin{equation}\label{eq:critical-gradient-upper}
 \|\nabla_LF_R\|_{p_c}^{p_c}\le CR^{m\mathfrak a}.
\end{equation}
On the plateau, \eqref{eq:critical-plateau} yields
\begin{equation}\label{eq:critical-square-root-lower}
 \|L_D^{1/2}F_R\|_{p_c}^{p_c}
 \ge CR^{m\mathfrak a}\int_2^{R/2}x^{-p_c(1-\alpha)}\,dx
 \ge CR^{m\mathfrak a}\log R.
\end{equation}

Finally, define
$f_R(x,y)=\operatorname{sgn}(x)F_R(|x|,y)$.  Because $F_R$ vanishes for
$0<x\le1$, we have $f_R\in C_c^\infty(\mathbb R^{1+m})$.  By
\eqref{eq:odd-square-root-reflection},
\eqref{eq:critical-gradient-upper}, and
\eqref{eq:critical-square-root-lower},
\begin{equation*}
 \frac{\|L^{1/2}f_R\|_{p_c}}{\|\nabla_Lf_R\|_{p_c}}
 \ge C(\log R)^{1/p_c}\longrightarrow\infty.
\end{equation*}
Hence \eqref{RR_p} fails at $p=p_c$.
\end{proof}

\begin{remark}\label{rem:critical-Poincare-route}
Let $p_c=(1-\alpha)^{-1}$. In the weakly degenerate regime
$0<\alpha<1/2$, a direct adaptation of the intrinsic-box argument in
\cite[Theorem~1.1]{RS2} shows that, for every $q\in(p_c,2)$,
\begin{equation}\label{eq:subcritical-Pq}
    \left(\fint_B|f-f_B|^q\,d\xi\right)^{1/q}
    \leq
    C_qr(B)
    \left(\fint_{\lambda B}|\nabla_Lf|^q\,d\xi\right)^{1/q}
\end{equation}
for a fixed $\lambda\geq1$. The restriction on $q$ appears when a path
crosses the singular line. If $q'=q/(q-1)$, then
\begin{equation*}
 |f(x,y)-f(x',y)|
 \leq
 \left(\int_{x'}^x|s|^{\alpha q}|\partial_sf(s,y)|^q\,ds\right)^{1/q}
 \left(\int_{x'}^x|s|^{-\alpha q'}\,ds\right)^{1/q'},
\end{equation*}
and the second factor is finite across $s=0$ precisely when
\begin{equation*}
    \alpha q'<1,
    \qquad\text{equivalently}\qquad q>p_c.
\end{equation*}
Together with the tangential estimate on intrinsic boxes and the remote-ball
estimate of Lemma~\ref{PE}, this proves \eqref{eq:subcritical-Pq}.

Adapting \cite[Theorem~0.7]{AC} to the fixed-enlargement inequality
\eqref{eq:subcritical-Pq} gives $(\mathrm{RR}_p)$ for $q<p<2$.
For $p\geq2$, the estimate follows independently by duality from the known
$L^{p'}$ boundedness of the Riesz transform (with $p=2$ following directly
from the quadratic form). Thus one recovers $(\mathrm{RR}_p)$ for every
$p>q$, and letting $q\downarrow p_c$ gives only the upper range $p>p_c$.
Harmonic annihilation supplies the complementary range $1<p<p_c$.

Theorem~\ref{thm:critical-RR-failure} also shows that failure at
$p=p_c$ is not a limitation of the Poincar\'e method. In the odd sector,
the full-line problem becomes the half-line Dirichlet problem, and
\eqref{eq:critical-survival} measures the loss of heat through the
absorbing boundary $x=0$. Since
\begin{equation*}
    \bigl(x^{-(1-\alpha)}\bigr)^{p_c}=x^{-1},
\end{equation*}
this boundary leakage is exactly critical and produces the logarithmic
divergence in \eqref{eq:critical-square-root-lower}.
\end{remark}

\part{Riesz transform for \texorpdfstring{$n=1$}{n=1} and \texorpdfstring{$\frac{1}{2}\le \alpha<1$}{1/2<=alpha<1}}\label{part4}

We now assume $n=1$ and $1/2\le\alpha<1$.
The separated zero-flux kernels yield a Hodge projector estimate, which
we combine with the reverse inequality from Part~\ref{part3}.

\section{Proof of Theorem~\ref{thm:n=1:alpha>1/2}}

\subsection{From reverse H\"older to Hodge projector}

Consider the Hilbert space $L^2(\mathbb R^{1+m};\mathbb C^{1+m})$. We denote by
\begin{equation*}
    \operatorname{div}_L:=\nabla_L^*
\end{equation*}
the $L^2$-adjoint of $\nabla_L$. Thus, for smooth compactly supported
$u$ and vector fields $F=(F_x,F_y)$,
\begin{equation*}
    \int_{\mathbb R^{1+m}}
     \nabla_Lu \cdot F \, d\xi =
    \int_{\mathbb R^{1+m}}
    u \operatorname{div}_L F \, d\xi ,
\end{equation*}
where explicitly
\begin{equation*}
\operatorname{div}_L F = -\partial_x(|x|^\alpha F_x) - \sum_{j=1}^m\partial_{y_j}(|x|^\beta F_{y_j}).
\end{equation*}
With this convention, $L=\operatorname{div}_L\nabla_L$ in the sense of the Friedrichs form. The Hodge projector is defined on $L^2(\mathbb{R}^{1+m}; \mathbb{C}^{1+m})$ by
\begin{equation*}
   \mathcal P:=\nabla_L L^{-1}\operatorname{div}_L =  \nabla_L L^{-1/2} \bigl(\nabla_L L^{-1/2}\bigr)^*.
\end{equation*}
Since $\nabla_L L^{-1/2}$ is a partial isometry on $L^2$, the operator $\mathcal P$ is an orthogonal projection on $L^2$. In particular,
\begin{equation*}
    \mathcal P^2= \mathcal P,\qquad \mathcal P^*= \mathcal P,\qquad \|\mathcal P\|_{2\to2}\le 1,
\end{equation*}
and we say the Hodge projector is bounded on $L^p$ if
\begin{equation}\label{Pi_p}\tag{$\Pi_p$}
    \| \nabla_L L^{-1} \mathrm{div}_L \omega \|_p \le C \| \omega \|_p.
\end{equation}
The aim of this subsection is first to prove \eqref{Pi_p} in the strongly degenerate one-dimensional regime $n=1$ and $1/2 \le \alpha <1$. Once this is obtained, we combine it with the reverse Riesz inequality from
Theorem~\ref{prop_RR} to prove Theorem~\ref{thm:n=1:alpha>1/2}.

The following property is due to Shen \cite{Shen}.

\begin{theorem}[{\cite{Shen}}]\label{thm:Shen}
Let $(E,\mu)$ be a measured metric space satisfying the doubling property \eqref{D}. Let $T$ be a bounded sublinear operator on $L^2(E,\mu)$. Assume that for some $q\in (2,\infty]$, $1<\lambda_1<\lambda_2$, and $C>0$, we have
\begin{equation}\label{eq:Shen}
    \biggl( \fint_B |Tf|^q \, d\mu \biggr)^{\!\frac{1}{q}} \le C \biggl( \fint_{\lambda_1 B} |Tf|^2 \, d\mu \biggr)^{\!\frac{1}{2}}
\end{equation}
for all balls $B \subset E$ and all $f\in L^2(E,\mu)$ supported in $E\setminus \lambda_2 B$. Then $T$ is bounded on $L^p(E,\mu)$ for $2<p<q$.  
\end{theorem}

To apply Theorem~\ref{thm:Shen} to $\mathcal{P}$, we first recall the following more natural reverse H\"older inequality used in \cite{Shen}:
\begin{equation}\label{eq:RH:Shen}\tag{$\widetilde{\mathrm{RH}}_p$}
     \biggl( \fint_B |\nabla_L u|^p \, d\mu \biggr)^{\!\frac{1}{p}} \le C \biggl( \fint_{\lambda B} |\nabla_L u|^2 \, d\mu \biggr)^{\!\frac{1}{2}}.
\end{equation}

Under \eqref{D} and \eqref{UE}, the Caccioppoli inequality together with the mean value estimate \eqref{meanvalue} gives the implication $\eqref{eq:RH:Shen} \implies \eqref{eq:RH:Jiang}$. However, for the opposite direction, one usually needs \eqref{P2}. 

From our discussion in Part~\ref{part2}, we are able to establish the following.

\begin{lemma}
\label{lem:strong-zero-flux-green}
Let $m\ge1$, $\frac12\le\alpha<1$, and $\beta\ge0$.  Put
\begin{equation*}
 \nu=\frac{\mathfrak b}{\mathfrak a}\ge0,\qquad
 \tau_\rho=\frac{|Y|\rho^{\mathfrak a}}{\mathfrak a},\qquad
 \Phi_Y(\rho)=\rho^{-\mathfrak b}I_\nu(\tau_\rho),
\end{equation*}
and, for $Y\ne0$,
\begin{equation*}
 \Psi_Y(\rho)=\rho^{-\mathfrak b}
 \left[K_\nu(\tau_\rho)-
       \frac{K_\nu(\tau_2)}{I_\nu(\tau_2)}I_\nu(\tau_\rho)\right].
\end{equation*}
The zero-flux Green kernel on $(0,2)$ is
\begin{equation}\label{eq:strong-zero-flux-green}
G_Y^0(\rho,\sigma)=\mathfrak a^{-1}
\begin{cases}
 \Phi_Y(\rho)\Psi_Y(\sigma),&0<\rho<\sigma<2,\\
 \Phi_Y(\sigma)\Psi_Y(\rho),&0<\sigma<\rho<2.
\end{cases}
\end{equation}
for $Y\ne0$.  At $Y=0$ it is defined by continuous extension, namely
\begin{equation}\label{eq:strong-zero-flux-green-zero}
 G_0^0(\rho,\sigma)
 =\int_{\max\{\rho,\sigma\}}^2r^{-2\alpha}\,dr,
\end{equation}
with the logarithmic interpretation
$G_0^0(\rho,\sigma)=\log(2/\max\{\rho,\sigma\})$ when $\alpha=1/2$.
For $D_r=\rho^\alpha\partial_\rho$ and
$D_{y,j}=\rho^\beta Y_j$, this kernel satisfies
\begin{align}
 \sup_{0<\rho<2}\int_0^2G_Y^0(\rho,\sigma)\,d\sigma&\le C,
       && |Y|\le C^*,\label{eq:strong-green-low}\\
 \sup_{0<\rho<2}\rho^{2\beta}
       \int_0^2G_Y^0(\rho,\sigma)\,d\sigma&\le C|Y|^{-2},
       && |Y|>C^*,\label{eq:strong-green-high}\\
 \int_0^2|D_*G_Y^0(\rho,\sigma)|\,d\sigma
   &\le\begin{cases}C,&|Y|\le C^*,\\ C|Y|,&|Y|>C^*,\end{cases}
       &&0<\rho\le1,\quad *=r,y,j.\label{eq:strong-green-gradient}
\end{align}
Moreover, for every integer $N>m+3$,
\begin{equation}\label{eq:strong-green-Y-derivatives}
 \int_{\mathbb R^m}\!\int_0^2(1+|Y|)
 \sum_{|\gamma|=N}|\partial_Y^\gamma D_*G_Y^0(\rho,\sigma)|
 \,d\sigma\,dY\le C_N,
 \qquad 0<\rho\le1.
\end{equation}
\end{lemma}

\begin{proof}
Because $\mathfrak b=\mathfrak a\nu$, the branch $\Phi_Y$ tends to a
constant at the origin and has zero conormal flux; the other branch has
infinite Friedrichs energy (with the usual logarithmic interpretation when
$\nu=0$).  Thus \eqref{eq:strong-zero-flux-green} is the Green kernel of the
strong Friedrichs half-line problem.

We give the estimates separately because the weak even formula from
Proposition~\ref{prop:Green:n=1:even} uses the different order $-\nu$.
Suppose first that $\nu>0$.  Monotonicity of
$t^{-\nu}I_\nu(t)$ and $t^\nu K_\nu(t)$, together with the fixed-order
bound $I_\nu(t)K_\nu(t)\le C_\nu(1+t)^{-1}$, gives
\eqref{eq:strong-green-low}.  At the endpoint $\nu=0$ (equivalently
$\alpha=1/2$), the boundary subtraction must be retained.  The Wronskian
identity gives, with $t=\tau_\rho$ and $T=\tau_2$,
\begin{equation}\label{eq:strong-K0-cancellation}
 \Psi_Y(\rho)
 =I_0(t)\int_t^T\frac{ds}{sI_0(s)^2}.
\end{equation}
When $0<|Y|\le C^*$, $I_0$ is bounded above and below on the relevant
compact interval, and hence
\begin{equation*}
 0\le\Phi_Y(\rho)\le C,
 \qquad
 0\le\Psi_Y(\rho)\le C\log\frac{T}{t}
 =C\mathfrak a\log\frac2\rho.
\end{equation*}
Splitting the row integral at $\sigma=\rho$ now yields
\begin{equation*}
 \int_0^2G_Y^0(\rho,\sigma)\,d\sigma
 \le C\left[\rho\log\frac2\rho
       +\int_\rho^2\log\frac2\sigma\,d\sigma\right]\le C.
\end{equation*}
This also agrees with the explicit zero-frequency kernel
\eqref{eq:strong-zero-flux-green-zero}.

In the high-frequency estimate, the two
changes of variables $u=\tau_\sigma$ reduce the two rows to the quantities
in Lemma~\ref{lem:Bessel:integral} with
\begin{equation*}
 q=\frac{2\beta-\mathfrak b}{\mathfrak a},\qquad
 s=\frac{1-\mathfrak b}{\mathfrak a}.
\end{equation*}
Here $q+s=2$, $s+\nu=\mathfrak a^{-1}>0$, and
$s>\nu$ because $\alpha<1$.  The hypotheses of both parts of that lemma
are therefore satisfied, proving \eqref{eq:strong-green-high}, including
$\nu=0$.  The
logarithmic derivative identities
\begin{equation*}
 \rho^\alpha\frac{\Phi_Y'(\rho)}{\Phi_Y(\rho)}
 =\mathfrak a\rho^{\alpha-1}\tau_\rho
   \frac{I_{\nu+1}(\tau_\rho)}{I_\nu(\tau_\rho)}
\end{equation*}
and the analogous estimate for $\Psi_Y$ give
\eqref{eq:strong-green-gradient} by the same two Bessel integrals when
$\nu>0$ and in the high-frequency region.  In the remaining case
$\nu=0$, $|Y|\le C^*$, differentiating
\eqref{eq:strong-K0-cancellation} gives
\begin{equation*}
 \rho^\alpha|\Phi_Y'(\rho)|
 \le C\rho^{\alpha-1}\tau_\rho^2,
 \qquad
 \rho^\alpha|\Psi_Y'(\rho)|
 \le C\rho^{\alpha-1}
       \left(1+\tau_\rho^2\log\frac2\rho\right).
\end{equation*}
Since $\tau_\rho\le C\rho^{\mathfrak a}$, splitting at
$\sigma=\rho$ as above gives the radial derivative row bound; the base
derivative follows from $|Y|\rho^\beta\int G_Y^0\le C$.  Thus
\eqref{eq:strong-green-gradient} holds at the endpoint as well.

Finally, the resolvent identity for the strong half-line operator is
\begin{equation*}
 \partial_{Y_j}R_Y^0=-R_Y^0(2Y_j\rho^{2\beta})R_Y^0.
\end{equation*}
Iteration, \eqref{eq:strong-green-low}--\eqref{eq:strong-green-gradient},
and the corresponding weighted row estimate
\eqref{eq:strong-green-high} yield
\eqref{eq:strong-green-Y-derivatives}.  This is exactly the fixed-order
resolvent calculation of Proposition~\ref{lem:second_to_last}, now based on
the strong zero-flux branch just established.
\end{proof}

\begin{theorem}\label{thm:RH:n1:alpha>1/2}
Let $n=1$, $m\ge1$, $\alpha \in [\frac{1}{2},1)$, and $\beta \ge0$. Then \eqref{eq:RH:Jiang} holds for all $1<p\le \infty$.
\end{theorem}

\begin{proof}
We prove the anchored estimate; the remote case follows from Lemma~\ref{RH_remote} and the reduction from arbitrary balls to remote and anchored balls is the same as in Lemma~\ref{reduce_RH}. After rescaling an anchored ball, it is enough to prove the following fixed-cylinder estimate. If $v$ is $L$-harmonic in $Q_4=(-4,4)\times B_m(0,4^{\mathfrak a})$, then
\begin{equation*}
    \sup_{Q_1}|\nabla_L v| \le C\|v\|_{L^\infty(Q_3)}.
\end{equation*}
The $L^p$ reverse H\"older inequality then follows immediately for every
$1<p\le \infty$. Let $\chi\in C_c^\infty(B_m(0,3^{\mathfrak a}))$ be such that $\chi=1$ on $B_m(0,2^{\mathfrak a})$, and set $\omega(x,y):=\chi(y)v(x,y)$. Then $\omega=v$ on $Q_2$, and
\begin{equation*}
L\omega=f:=[L,\chi]v = |x|^{2\beta}v\Delta_y\chi - 2|x|^{2\beta}\nabla_y\chi\cdot\nabla_yv.
\end{equation*}
The function $f$ is supported in the annulus $2^{\mathfrak a}\leq |y|\leq 3^{\mathfrak a}$.

Since $1/2\le \alpha<1$, the Friedrichs realization separates the two half-lines. Thus the restrictions
\begin{equation*}
    \omega_\pm(\rho,y):=\omega(\pm\rho,y),
    \qquad
    f_\pm(\rho,y):=f(\pm\rho,y),
    \qquad
    h_\pm(y):=\omega(\pm2,y)
\end{equation*}
solve independent half-line Friedrichs problems on $(0,2)\times\mathbb R^m$.
Here $h_\pm$ belong to the natural outer trace space and $f_\pm$, understood
through the commutator identity, belong to the dual of the zero-trace form
space and have compact support in a $y$-annulus.  The Friedrichs
representation is therefore
\begin{equation*}
    \omega_\pm=\mathcal T h_\pm+ \mathcal S f_\pm,
\end{equation*}
in the local form domain, where $\mathcal T$ and $\mathcal S$ are
respectively the Poisson and Green operators for the half-line problem
displayed below.  Indeed, subtracting the right-hand side leaves a
zero-trace homogeneous remainder; testing its form equation against the
remainder itself gives zero energy and hence zero remainder.  The explicit
Fourier formulas below follow first for smooth data and then for the present
trace/form-dual data by density.
\begin{align*}
\begin{cases}
    Lg=0, & 0<\rho<2,\\
    g(2,y)=h_\pm(y),\\
    \displaystyle \lim_{\rho\to0}\rho^{2\alpha}\partial_\rho g(\rho,y)=0,
\end{cases}\qquad \mathrm{and}\qquad \begin{cases}
    Lg=f_\pm, & 0<\rho<2,\\
    g(2,y)=0,\\
    \displaystyle \lim_{\rho\to 0}\rho^{2\alpha}\partial_\rho g(\rho,y)=0.
\end{cases}   
\end{align*}
We now identify the Fourier--Bessel branch. Taking the Fourier transform in
the $y$-variable gives the ordinary differential equation
\begin{equation*}
-\partial_\rho\left(\rho^{2\alpha}\partial_\rho \hat g(\rho,Y)\right)+\rho^{2\beta}|Y|^2\hat g(\rho,Y)=0,\qquad 0<\rho<2.
\end{equation*}
Equivalently,
\begin{equation*}
\hat g''(\rho,Y)+\frac{2\alpha}{\rho}\hat g'(\rho,Y)-|Y|^2\rho^{2(\beta-\alpha)}\hat g(\rho,Y)=0.
\end{equation*}
Set $\nu:=\frac{\mathfrak b}{\mathfrak a}$ and $\tau_\rho:=\frac{|Y|\rho^\mathfrak a}{\mathfrak a}$. Note that $\mathfrak b\ge 0$, $\nu\ge 0$, and $\mathfrak b=\mathfrak a \nu$. With the substitution $\hat g(\rho,Y)=\rho^{-\mathfrak b}U(\tau_\rho)$, the equation becomes the modified Bessel equation
\begin{equation*}
    \tau^2U''(\tau)+\tau U'(\tau)-(\tau^2+\nu^2)U(\tau)=0.
\end{equation*}
Hence the two local branches are $\rho^{-\mathfrak b}I_\nu(\tau_\rho)$ and $\rho^{-\mathfrak b}K_\nu(\tau_\rho)$.

The Friedrichs finite-energy condition excludes the $K_\nu$ branch, while the regular branch
\begin{equation*}
    \Phi_Y(\rho):=\rho^{-\mathfrak b}I_\nu(\tau_\rho)
\end{equation*}
satisfies the zero-flux endpoint condition. Indeed, as $\rho\to 0$, $\Phi_Y(\rho)=C_Y\left(1+O(\rho^{2\mathfrak a})\right)$, and therefore
\begin{equation*}
\rho^{2\alpha}\Phi_Y'(\rho) = O\left(\rho^{2\alpha+2\mathfrak a-1}\right) = O\left(\rho^{2\beta+1}\right) \to 0.
\end{equation*}
Thus the Poisson multiplier is
\begin{equation*}
M(\rho,Y) := \frac{\Phi_Y(\rho)}{\Phi_Y(2)} = \frac{\mathcal I_\nu(\tau_\rho)}{\mathcal I_\nu(\tau_2)}, \qquad \mathcal I_\nu(t) := t^{-\nu} I_\nu(t).
\end{equation*}
The normalized quotient on the right provides its continuous value at
$Y=0$.
For the Green operator and $Y\ne0$, let
\begin{equation*}
\Psi_Y(\rho) := \rho^{-\mathfrak b} \left[ K_\nu(\tau_\rho) - \frac{K_\nu(\tau_2)} {I_\nu(\tau_2)} I_\nu(\tau_\rho) \right].
\end{equation*}
Then $\Psi_Y(2)=0$, and the half-line Green kernel is
\begin{equation*}
G_Y(\rho,\sigma) = \mathfrak a^{-1}
\begin{cases}
\Phi_Y(\rho)\Psi_Y(\sigma), & 0<\rho<\sigma<2,\\
\Phi_Y(\sigma)\Psi_Y(\rho), & 0<\sigma<\rho<2.
\end{cases}
\end{equation*}
for $Y\ne0$. At $Y=0$ it is understood by continuous extension:
\begin{equation*}
 G_0(\rho,\sigma)
 =\int_{\max\{\rho,\sigma\}}^2r^{-2\alpha}\,dr,
\end{equation*}
which equals $\log(2/\max\{\rho,\sigma\})$ when $\alpha=1/2$.
The zero-flux Poisson estimate extends to this non-negative order. More precisely, the proof in
Theorem~\ref{prop:Poisson:n=1}, together with
Remark~\ref{re:even:Poisson}, gives
\begin{equation*}
\sup_{Q_1} |\nabla_L \mathcal T h_\pm| \le C\|h_\pm\|_{L^\infty(\mathbb R^m)} \le C\|v\|_{L^\infty(Q_3)}.
\end{equation*}
It remains to estimate the Green term. Lemma~\ref{lem:strong-zero-flux-green}
was proved specifically for the present non-negative-order branch.  Its
$Y$-derivative estimates, followed by the same Fourier integration by parts
as in Theorem~\ref{lem:Green_kernel}, give the row estimates for
$\nabla_L \mathcal G$ and for
$\nabla_L\partial_{y_k'} \mathcal G$.  Consequently,
\begin{equation*}
\sup_{Q_1}|\nabla_L \mathcal S f_\pm| \le C\|v\|_{L^\infty(Q_3)}.
\end{equation*}
Indeed, the part of $f_\pm$ containing $v\Delta_y\chi$ is controlled directly by the row bound for $\nabla_L \mathcal G$. The part containing $\nabla_y\chi\cdot\nabla_yv$ is treated by integrating by parts in the
$y_k'$-variable, exactly as in the proof of Theorem~\ref{thm:RH};
the support of $\nabla_y\chi$ is separated from $Q_1$ in the $y$-variable, and
the row bound for $\nabla_L\partial_{y_k'} \mathcal G$ gives the desired estimate.

Combining the Poisson and Green estimates gives
\begin{equation*}
\sup_{Q_1}|\nabla_L\omega| \le C\|v\|_{L^\infty(Q_3)}.
\end{equation*}
Since $\omega=v$ on $Q_2$, we conclude that
\begin{equation*}
\sup_{Q_1}|\nabla_Lv| \le C\|v\|_{L^\infty(Q_3)}.
\end{equation*}
The local mean-value estimate used in \eqref{meanvalue}, together with
box--ball comparability, gives
\begin{equation*}
 \|v\|_{L^\infty(Q_3)}
 \le C\fint_{Q_4}|v|\,d\xi.
\end{equation*}
After scaling back (and choosing the fixed dilation in
\eqref{eq:RH:Jiang} large enough that its image contains $Q_4$), this is
the required reverse H\"older estimate.
This completes the proof of Theorem~\ref{thm:RH:n1:alpha>1/2}

\end{proof}

Next, we pass from \eqref{eq:RH:Jiang} to \eqref{eq:RH:Shen}.

\begin{proposition}\label{prop_u_RH_to_Shen_RH_large_alpha}
Let $n=1$, $m\ge1$, $\alpha \in [\frac{1}{2},1)$, and $\beta \ge0$. Then \eqref{eq:RH:Shen} holds for all $1<p\le \infty$.
\end{proposition}

\begin{proof}
The key point is that although the global scale-invariant Poincar\'e inequality across the singular set $\{x=0\}$ fails, a separated $L^2$-Poincar\'e inequality remains valid. More precisely, for $1/2\le \alpha<1$, Robinson--Sikora \cite[Theorem~1.1 III]{RS2} proves that
\begin{equation}\label{eq_separated_Poincare_ball}
\int_B \bigl| f - f_{B^+}\mathbf 1_{B^+} - f_{B^-}\mathbf 1_{B^-}\bigr|^2 \, d\xi \le C r(B)^2 \int_{\lambda B}|\nabla_Lf|^2 \, d\xi
\end{equation}
for some $\lambda \ge 1$, where $B^+:= B \cap \{x>0\}$ and $B^-:=B\cap \{x<0\}$.

We first consider the case where $B$ is remote. Then the implication \eqref{eq:RH:Jiang} $\implies$ \eqref{eq:RH:Shen} is immediate. Indeed, if $B$ is remote, then necessarily $B\cap \{x=0\} = \emptyset$. Moreover, by Lemma~\ref{PE}, the $L^2$-Poincar\'e inequality holds on $B$. Now suppose that $u$ is $L$-harmonic in $\lambda B$; here $\varepsilon_0$ is chosen sufficiently small so that $\lambda B$ stays away from $\{x=0\}$ and remains in the ball--box regime. Then Theorem~\ref{thm:RH:n1:alpha>1/2}, together with H\"older's inequality, yields
\begin{align*}
    \biggl( \fint_B |\nabla_L u|^p \, d\xi \biggr)^{\!\frac{1}{p}}
    &=\biggl( \fint_B |\nabla_L (u-u_{\lambda B})|^p \, d\xi \biggr)^{\!\frac{1}{p}} \\
    &\le \frac{C}{r(B)}  \fint_{\lambda B} |u-u_{\lambda B}| \, d\xi  \\
    &\le \frac{C}{r(B)} \biggl(\fint_{\lambda B} |u-u_{\lambda B}|^2 \, d\xi \biggr)^{\!\frac{1}{2}}\\
    &\le C \biggl(\fint_{\lambda' B} |\nabla_L u|^2 \, d\xi \biggr)^{\!\frac{1}{2}}.
\end{align*}

Next, consider an anchored ball $B=B((0,y_0), r)$, and let $u$ be $L$-harmonic in $\lambda B$. Define the separated constant
\begin{equation*}
\widetilde{c} := u_{(\lambda B)^+}\mathbf 1_{\{x>0\}}(x) + u_{(\lambda B)^-} \mathbf 1_{\{x<0\}}(x).
\end{equation*}
Since $\alpha\in [\frac{1}{2},1)$, the function $\widetilde{c}(x)$ belongs to $\mathcal{F}_{\mathrm{loc}}$ and is locally $L$-harmonic, while
\begin{equation*}
    \nabla_L u=\nabla_L(u-\widetilde{c})
\end{equation*}
(see Remark~\ref{rem_separated_constant_Friedrichs} below). Therefore, the same argument as in the remote case gives
\begin{align*}
    \biggl( \fint_B |\nabla_L u|^p \, d\xi \biggr)^{\!\frac{1}{p}}
    &=\biggl( \fint_B |\nabla_L (u-\widetilde{c})|^p \, d\xi \biggr)^{\!\frac{1}{p}} \\
    &\le \frac{C}{r(B)}  \fint_{\lambda B} |u-\widetilde{c}| \, d\xi  \\
    &\le \frac{C}{r(B)} \biggl(\fint_{\lambda B} |u-\widetilde{c}|^2 \, d\xi \biggr)^{\!\frac{1}{2}}\\
    &= \frac{C}{r(B)} |\lambda B|^{-\frac{1}{2}} \biggl( \sum_{\pm} \int_{(\lambda B)^{\pm}} |u-u_{(\lambda B)^{\pm}}|^2 \, d\xi \biggr)^{\!\frac{1}{2}} \\
    &\le C \biggl(\fint_{\lambda' B} |\nabla_L u|^2 \, d\xi \biggr)^{\!\frac{1}{2}}.
\end{align*}

Finally, if $B$ is neither anchored nor remote, then the same reduction argument as in Lemma~\ref{reduce_RH} applies. We omit the details.
\end{proof}

\begin{remark}\label{rem_separated_constant_Friedrichs}
The separated constant
$\widetilde c=c_+\mathbf1_{\{x>0\}}+c_-\mathbf1_{\{x<0\}}$
belongs to $\mathcal F_{\mathrm{loc}}$ when $\alpha\ge1/2$, by the
zero-capacity cutoff construction in Subsection~\ref{ssec:Friedrichs}.
Its form gradient vanishes, so it is locally harmonic and
$\nabla_L(u-\widetilde c)=\nabla_Lu$. Global $L^2$ membership is not needed
for the local reverse H\"older estimate.
\end{remark}

We can now apply Shen's criterion.

\begin{theorem}\label{thm:Hodge:n1}
Let $n=1$, $m\ge1$, $\alpha \in [\frac{1}{2},1)$, and $\beta \ge0$. Then \eqref{Pi_p} holds for all $1<p<\infty$, i.e.,
\begin{equation*}
    \| \mathcal{P} \omega \|_p = \| \nabla_L L^{-1} \operatorname{div}_L \omega \|_p \le C \| \omega \|_p
\end{equation*}
for all $1<p<\infty$.
\end{theorem}

\begin{proof}
The argument is based on \cite[Section~2.2]{AC}. Define
$T\omega:=|\mathcal P\omega|$.  Then $T$ is sublinear and bounded on
$L^2$. Fix $2<p<q<\infty$. Choose $1<\lambda_1<\lambda_2$ large enough
for Proposition~\ref{prop_u_RH_to_Shen_RH_large_alpha}, and let $B$ be a
Grushin ball. Take $\omega\in C_c^\infty(\mathbb R^{1+m};\mathbb C^{1+m})$ with
\begin{equation*}
    \operatorname{supp}\omega\subset(\lambda_2B)^c.
\end{equation*}
Let $u$ be the canonical extended-energy potential, defined up to the
zero-energy constants on the separated components, such that
\begin{equation*}
    Lu = \operatorname{div}_L \omega,
    \qquad \nabla_Lu=\mathcal P\omega.
\end{equation*}
By the support condition, $u$ is $L$-harmonic on the enlarged ball
required by Proposition~\ref{prop_u_RH_to_Shen_RH_large_alpha}.
That proposition gives \eqref{eq:Shen} with right-hand side over
$\lambda_1B$. The same estimate holds for $L^2$ data with this support
by approximation and the $L^2$ boundedness of $\mathcal P$. Together with
\eqref{D} and the $L^2$-boundedness of $\mathcal{P}$,
Theorem~\ref{thm:Shen} applies immediately.
This proves the claim for $2<p<\infty$.  Since
$\mathcal P^*=\mathcal P$ on $L^2$, duality gives the boundedness of
$\mathcal P$ on $L^p$ for $1<p<2$, while $p=2$ is already part of the
orthogonal-projection construction.
\end{proof}

\subsection{From Hodge projector to Riesz transform}

\begin{proof}[Proof of Theorem~\ref{thm:n=1:alpha>1/2}]
The estimate \eqref{R_p} for $1<p\leq2$ follows from
\cite[Theorem~8.1]{RS}. It remains to treat $p>2$. For such a $p$,
\begin{equation*}
    1<p'<2\le\frac{1}{1-\alpha}.
\end{equation*}
For $\omega\in C_c^\infty(\mathbb R^{1+m};\mathbb C^{1+m})$, define
\begin{equation*}
    u=L^{-1}\operatorname{div}_L \omega.
\end{equation*}
We first verify that Theorem~\ref{prop_RR} applies to $u$. By Remark~\ref{re:applicable:HA} (see also Remarks~\ref{re:Hardy} and \ref{re:harmonic:annihilation}), $u$ must satisfy the following two conditions:
\begin{enumerate}[label=(\roman*)]
    \item $|\nabla_L u|\in L^{p'}$,
    \item $u(\xi) \to 0$ as $d(\xi,0)\to \infty$.
\end{enumerate}
Condition (i) follows from Theorem~\ref{thm:Hodge:n1}. Indeed, $(\mathrm{\Pi}_{p'})$ implies
\begin{equation*}
    \|\nabla_L u\|_{p'} = \| \nabla_L L^{-1} \mathrm{div}_L \omega\|_{p'}\le C \|\omega\|_{p'} <\infty.
\end{equation*}
For condition (ii), set
\begin{equation*}
    f:=\operatorname{div}_L\omega.
\end{equation*}
The only possible singular coefficient in $f$ is a constant multiple of
$|x|^{\alpha-1}$, which is locally integrable because $\alpha>0$.
Consequently, $f\in L_c^1(\mathbb R^{1+m})$. If $\mathcal G$ denotes the
Green kernel of $L^{-1}$, then
\begin{equation}\label{eq:u:scalar-Green}
    u(\xi)
    =\int_{\mathbb R^{1+m}}\mathcal G(\xi,\eta)f(\eta)\,d\eta.
\end{equation}
Here we choose the Green-normalized representative of the extended-energy
solution $Lu=f$. This choice fixes the componentwise additive constants of
the separated realization and does not change
$\nabla_Lu=\mathcal P\omega$. The identity
\eqref{eq:u:scalar-Green} follows from the zero-resolvent construction:
since the space is transient (as verified below by $\mathcal Q>2$),
$(L+\varepsilon)^{-1}f$ converges, locally away from
$\operatorname{supp}f$, to the Green potential as
$\varepsilon\downarrow0$. If $\alpha>1/2$, then in fact
$f\in L_c^2$ and this is the usual spectral resolvent identity. At the
borderline $\alpha=1/2$, where generally only $f\in L_c^1$, the same
identity follows by applying the positive resolvent kernels to compactly
supported bounded approximations of the positive and negative parts of
the real and imaginary parts of $f$, and then using monotone convergence.

Recall that in the present one-dimensional regime
\begin{equation*}
    \mathcal Q
    =\frac{1+m(\beta+1-\alpha)}{1-\alpha}
    \geq1+\frac1{1-\alpha}\geq3.
\end{equation*}
The two cases in the volume formula of Theorem~\ref{RS_Grushin} give the
global lower bound
\begin{equation}\label{eq:global-volume-lower-Q}
    V(\zeta,s)\geq c s^{\mathcal Q},
    \qquad \zeta\in\mathbb R^{1+m},\quad s>0.
\end{equation}
Indeed, this is the anchored volume estimate when
$s\geq |x|^{1-\alpha}$; in the remote range,
\begin{equation*}
 V((x,y),s)\simeq s^{m+1}|x|^{\alpha+m\beta}
 \geq s^{m+1+(\alpha+m\beta)/(1-\alpha)}=s^{\mathcal Q}.
\end{equation*}
Combining \eqref{eq:global-volume-lower-Q} with the Gaussian heat-kernel
upper estimate and integrating in time, we obtain
\begin{align}\label{eq:scalar-Green-decay}
 0\leq\mathcal G(\xi,\eta)
 &=\int_0^\infty p_t(\xi,\eta)\,dt\notag\\
 &\leq C\int_0^\infty
       t^{-\mathcal Q/2}
       \exp\left(-c\frac{d(\xi,\eta)^2}{t}\right)dt
 \leq C d(\xi,\eta)^{2-\mathcal Q}.
\end{align}
Let $K=\operatorname{supp}f$. As $d(\xi,0)\to\infty$,
$d(\xi,\eta)\simeq d(\xi,0)$ uniformly for $\eta\in K$. Hence
\begin{equation*}
    |u(\xi)|
    \leq
    C d(\xi,0)^{2-\mathcal Q}\|f\|_1
    \longrightarrow0.
\end{equation*}
This proves condition (ii).

Therefore, $(\mathrm{RR}_{p'})$ applies to $u$, namely
\begin{equation*}
   \|L^{-1/2}\operatorname{div}_L\omega\|_{p'} = \| L^{1/2} u\|_{p'} \le C \| \nabla_L u\|_{p'}.
\end{equation*}
Since $\nabla_L u=\mathcal P\omega$, Theorem~\ref{thm:Hodge:n1} gives
\begin{equation*}
\|L^{-1/2}\operatorname{div}_L\omega\|_{p'} \le C\|\mathcal P\omega\|_{p'} \le
C\|\omega\|_{p'}.
\end{equation*}
By duality, this yields the boundedness of $\nabla_L L^{-1/2}$ on $L^p$.
\end{proof}

\begin{remark}
In fact, the implication \eqref{RR_p} $+$ \eqref{Pi_p}
$\Longrightarrow(\mathrm{R}_{p'})$ is more subtle than it looks.  A
counterexample is given in \cite[Theorem~1.5]{HeSi} within the framework of manifolds with ends.  On the
connected sum $\mathbb{R}^2\#\mathbb R^2$, \eqref{R_p} holds if and only
if $1<p\le2$, whereas \eqref{RR_p} and \eqref{Pi_p} both remain valid for
all $1<p<\infty$.
\end{remark}

\section{Hodge projector in the weakly degenerate regime}\label{sec:weak-Hodge}

We conclude with the following theorem.

\begin{theorem}\label{cor:hodge-projector-range}
Let $n=1$, $m\geq1$, $\alpha\in(0,\frac12)$, and $\beta\geq0$. Then
\eqref{Pi_p} holds if and only if
\begin{equation*}
 (1-\alpha)^{-1}<p<\alpha^{-1}.
\end{equation*}
\end{theorem}

\begin{proof}
Suppose first that $(1-\alpha)^{-1}<p<\alpha^{-1}$.  Then both $p$ and
$p'$ are strictly smaller than $\alpha^{-1}$.  By
Theorem~\ref{thm:n=1:alpha<1/2}, the Riesz transform
$T=\nabla_LL^{-1/2}$ is bounded on $L^p$ and on $L^{p'}$.  The
factorization
\begin{equation*}
 \mathcal P=TT^*
 =(\nabla_LL^{-1/2})(\nabla_LL^{-1/2})^*
\end{equation*}
therefore proves \eqref{Pi_p}.

We next rule out $p\ge\alpha^{-1}$.  Consider the weakly $L$-harmonic
function
\begin{equation*}
 u(x)=\operatorname{sgn}(x)|x|^{1-2\alpha}.
\end{equation*}
Its weighted conormal flux is constant:
\begin{equation}\label{eq:hodge-counterexample-flux}
 |x|^{2\alpha}u'(x)=1-2\alpha
 \quad\text{for a.e. }x\in\mathbb R.
\end{equation}
Choose $\chi\in C_c^\infty(\mathbb R)$ with $\chi=1$ near $0$, and
$G\in C_c^\infty(\mathbb R^m;\mathbb R^m)$ such that
\begin{equation*}
 \eta:=\operatorname{div}_yG\not\equiv0.
\end{equation*}
Set
\begin{equation*}
 v(x,y)=\chi(x)u(x)\eta(y).
\end{equation*}
Since $|x|^\alpha u'(x)=(1-2\alpha)|x|^{-\alpha}$ belongs to
$L^2_{\mathrm{loc}}(\mathbb R)$ and
$|x|^\beta u(x)=O(|x|^{\beta+1-2\alpha})$, the weighted gradient is locally
square integrable. To verify membership in the Friedrichs domain, choose
smooth odd functions $u_\delta$ such that $u_\delta=u$ for
$|x|\geq2\delta$ and, on $|x|\leq2\delta$,
\begin{equation*}
 |u_\delta(x)|\leq C\delta^{1-2\alpha},
 \qquad |u_\delta'(x)|\leq C\delta^{-2\alpha}.
\end{equation*}
For $\delta$ small enough that $\chi=1$ on $(-2\delta,2\delta)$, the
functions $v_\delta=\chi u_\delta\eta$ belong to
$C_c^\infty(\mathbb R^{1+m})$, and
\begin{equation*}
 \|v_\delta-v\|_{\mathcal F}^2
 \leq C\bigl(\delta^{1-2\alpha}
       +\delta^{3-4\alpha}
       +\delta^{2\beta+3-4\alpha}\bigr)\longrightarrow0.
\end{equation*}
The three terms control, respectively, the weighted $x$-energy, the $L^2$
error, and the weighted $y$-energy. Hence $v\in\mathcal F$.

Define the compactly supported vector field $\omega=(\omega_x,\omega_y)$ by
\begin{align*}
 \omega_x(x,y)&=|x|^\alpha\chi'(x)u(x)\eta(y),\\
 \omega_y(x,y)&=|x|^\beta\chi(x)u(x)\nabla_y\eta(y)
 +(1-2\alpha)\chi'(x)|x|^{-\beta}G(y).
\end{align*}
All terms containing $\chi'$ are supported away from $x=0$, while the
remaining term is bounded there.  Hence $\omega$ is bounded and compactly
supported, and in particular belongs to $L^2\cap L^r$ for every
$1\le r<\infty$.  Put $h=\omega-\nabla_Lv$.  Then
\begin{equation*}
 h=\left(-|x|^\alpha\chi(x)u'(x)\eta(y),\
 (1-2\alpha)\chi'(x)|x|^{-\beta}G(y)\right)\in L^2.
\end{equation*}
Using \eqref{eq:hodge-counterexample-flux} and
$\operatorname{div}_yG=\eta$, we obtain in distributions
\begin{align*}
 \operatorname{div}_Lh
 &=\partial_x\!\left(|x|^{2\alpha}\chi u'\eta\right)
   -\operatorname{div}_y\!\left((1-2\alpha)\chi'G\right)\\
 &=(1-2\alpha)\chi'\eta-(1-2\alpha)\chi'\eta=0.
\end{align*}
Thus $h$ is orthogonal to the closure of the $L^2$ range of $\nabla_L$.
Since $\mathcal P$ is the orthogonal projection onto this closed subspace,
\begin{equation}\label{eq:hodge-counterexample-projection}
 \mathcal P\omega=\mathcal P(\nabla_Lv+h)=\nabla_Lv.
\end{equation}

Let $\varepsilon>0$ be such that $\chi=1$ on
$(-\varepsilon,\varepsilon)$.  On this strip,
\begin{equation*}
 |x|^\alpha\partial_xv
 =(1-2\alpha)|x|^{-\alpha}\eta(y),
\end{equation*}
which does not belong to $L^p$ for $p\ge\alpha^{-1}$.  This contradicts
\eqref{Pi_p}, because $\omega\in L^p$ and
\eqref{eq:hodge-counterexample-projection} holds in $L^2$.  If
\eqref{Pi_p} is initially formulated only for smooth compactly supported
fields, approximate $\omega$ simultaneously in $L^2$ and $L^p$; the
$L^2$ identity and the assumed $L^p$ bound give the same contradiction.

Finally, $\mathcal P^*=\mathcal P$ on $L^2$, so boundedness on $L^p$
implies boundedness on $L^{p'}$.  If
$1<p\le(1-\alpha)^{-1}$, then $p'\ge\alpha^{-1}$, which has just been
excluded.  This also treats both endpoints.
\end{proof}

\begin{remark}\label{rem:hodge-counterexample-mechanism}
The preceding construction is forced by the one-dimensional harmonic
equation
\begin{equation*}
    (|x|^{2\alpha}u')'=0.
\end{equation*}
Thus the conormal flux is constant. Compatibility with the common trace
and flux matching across $0$ selects, up to an additive constant and
normalization, the odd mode
\begin{equation*}
    u_0(x)=\operatorname{sgn}(x)|x|^{1-2\alpha}.
\end{equation*}
Its intrinsic derivative is
$(1-2\alpha)|x|^{-\alpha}$, so this calculation already predicts the
upper endpoint $p=\alpha^{-1}$. It remains to hide this singular exact field inside the Hodge projection
of a bounded input. Localizing $u_0$ by $\chi$ creates only a cutoff flux
defect, because the uncut mode has constant flux. Choosing the transverse
factor as $\eta=\operatorname{div}_yG$ makes this defect a
$y$-divergence and therefore supplies the correction in the $y$-component
of $\omega$. With $h=\omega-\nabla_Lv$, the two defects cancel exactly
and $\operatorname{div}_Lh=0$. Thus $h$ lies in the orthogonal complement
of the exact subspace, whereas $\nabla_Lv$ lies in the exact part; the Hodge
projector removes $h$ but
retains the singular harmonic gradient:
\begin{equation*}
    \mathcal P\omega=\nabla_Lv.
\end{equation*}
The lower endpoint follows from the self-adjointness of $\mathcal P$ and
duality.
\end{remark}

\textbf{Acknowledgments.}
The author would like to thank Professor Adam Sikora for carefully reading this manuscript and for many helpful discussions. The author is also grateful to Professors Renjin Jiang and Lixin Yan for their interests in this work. D.~He is supported by the China Postdoctoral Science Foundation (No.~2026T190831) and the NNSF of China (No.~42381947).

\bibliographystyle{abbrv}
\bibliography{bib}

\end{document}